\documentclass[12pt,twoside]{amsart}
\usepackage[margin=3cm]{geometry}
\usepackage[colorlinks=false]{hyperref}
\usepackage[english]{babel}
\usepackage{graphicx,titling}
\usepackage{float}
\usepackage{lipsum}
\usepackage[T1]{fontenc}
\usepackage{fourier}
\usepackage{amsmath,amsfonts,amssymb,amsthm}
\usepackage{color}
\usepackage[latin1]{inputenc}
\usepackage{esint}
\usepackage{caption}
\usepackage{ccicons}

\makeatletter
\def\blfootnote{\xdef\@thefnmark{}\@footnotetext}
\makeatother

\newcommand\ccnote{
    \blfootnote{\copyright\,\, Odysseas Bakas, Francesco Di Plinio, Ioannis Parissis, and Luz Roncal}
    \blfootnote{\ccLogo\, \ccAttribution\,\, Licensed under a \href{https://creativecommons.org/licenses/by/4.0/}{Creative Commons Attribution License (CC-BY)}.}
}

\usepackage[export]{adjustbox}
\numberwithin{equation}{section}
\usepackage{setspace}
\renewcommand{\le}{\leqslant}
\renewcommand{\leq}{\leqslant}
\renewcommand{\ge}{\geqslant}
\renewcommand{\geq}{\geqslant}
\renewcommand{\mathbb}{\varmathbb}
\usepackage{fancyhdr}
\newtheorem{theorem}{Theorem}[section]
\newtheorem{lemma}[theorem]{Lemma}
\newtheorem{corollary}[theorem]{Corollary}
\newtheorem{proposition}[theorem]{Proposition}
\newtheorem{definition}[theorem]{Definition}
\newtheorem{remark}[theorem]{Remark}
\usepackage{bm}
\usepackage{graphicx}
\usepackage[dvipsnames]{xcolor}
\usepackage{color}
\usepackage[textsize=tiny]{todonotes}
\usepackage[normalem]{ulem}
\usepackage{mathrsfs}
\usepackage{mathtools}
\mathtoolsset{showonlyrefs}
\usepackage{lipsum}
\usepackage{wrapfig}
\usepackage{pgfplots}
\usepgfplotslibrary{colormaps}
\usepackage{tikz-3dplot}
\usepackage{pgf,tikz}
\usetikzlibrary{matrix,arrows}

\usepackage{sansmath}

\usepackage{upgreek}

  \def\b#1{{#1}}

\def\e{{\mathrm e}}
\def\eps{\varepsilon}

\def\d{{\mathrm d}}
\def\dist{{\mathrm {dist}}}

\def\supp{{\mathrm{supp}\, }}

\def\R{\mathbb{R}}
\def\N{\mathbb{N}}

\def\size{\mathrm{size}}
\def\dense{\mathrm{dense}}

\def\T {\mathbf{T}}

\def\Q {{\mathcal Q}}

\def\M {{\mathrm M}}

\def\SS {{\mathbb S}}
\def\Z {{\mathbb Z}}
\def\1 {{\mbox{\boldmath 1}}}

\def \l {\langle}

\def \r {\rangle}

\def \scl{\mathrm{scl}}
\def \Ann{\mathrm{Ann}}

\def\ind{\cic{1}}
\newcommand{\cic}{\bm}
\newcommand{\sh}{\mathbf{sh}}

\newcommand{\pr}{\Pi}

\def\Gr{\mathrm{\bf{Gr}}}

 \def\Xint#1{\mathchoice
	   {\XXint\displaystyle\textstyle{#1}}%
	   {\XXint\textstyle\scriptstyle{#1}}%
	   {\XXint\scriptstyle\scriptscriptstyle{#1}}%
	   {\XXint\scriptscriptstyle\scriptscriptstyle{#1}}%
	   \!\int}
	 \def\XXint#1#2#3{{\setbox0=\hbox{$#1{#2#3}{\int}$}
	     \vcenter{\hbox{$#2#3$}}\kern-.5\wd0}}
	 
	 \def\avgint{\Xint-}

\address{Odysseas Bakas, Department of Mathematics, University of Patras, 26504 Patras, Greece}
\email{obakas@upatras.gr}

\medskip
\address{Francesco Di Plinio, Dipartimento di Matematica e Applicazioni ``R.\ Caccioppoli'' Universit\`a di Napoli ``Federico II'', Via Cintia, Monte S.\ Angelo 80126 Napoli, Italy} 
\email{francesco.diplinio@unina.it}
\medskip
\address{Ioannis Parissis, Departamento de Matem\'aticas, Universidad del Pa\'is Vasco, Aptdo. 644, 48080 Bilbao, Spain and Ikerbasque, Basque Foundation for Science, Bilbao, Spain}
\email{ioannis.parissis@ehu.es}
\medskip
\address{Luz Roncal, BCAM - Basque Center for Applied Mathematics
48009 Bilbao, Spain,
 Ikerbasque, Basque Foundation for Science, Bilbao, Spain, and Departamento de Matem\'aticas, Universidad del Pa\'is Vasco, Aptdo. 644, 48080 Bilbao, Spain}
\email{lroncal@bcamath.org}

\begin{document}

\thispagestyle{empty}

\begin{minipage}{0.28\textwidth}
\begin{figure}[H]
%\centering
\includegraphics[width=2.5cm,height=2.5cm,left]{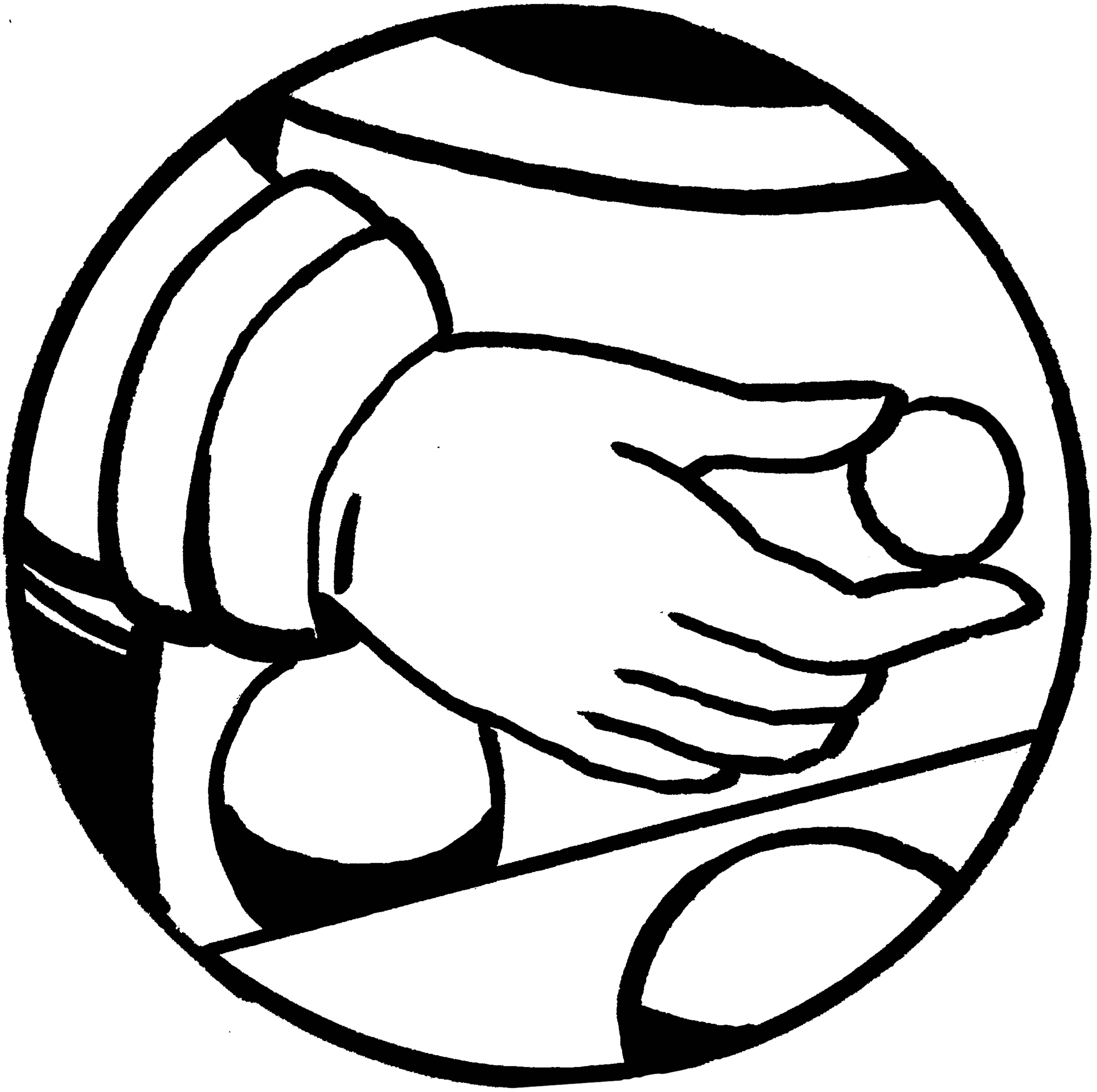}
\end{figure}
\end{minipage}
\begin{minipage}{0.7\textwidth} 
\begin{flushright}
%% The following metadata, in particular
%% the Paper No. and the DOI will be inserted by the journal
Ars Inveniendi Analytica (2024), Paper No. 2, 53 pp.
\\
DOI 10.15781/md28-ws10
\\
ISSN: 2769-8505
\end{flushright}
\end{minipage}

\ccnote

\vspace{1cm}

%%      -------------------------------------------------------------------------------
%%      -------------------------- TITLE ----------------------------
%%      -------------------------------------------------------------------------------
%% Authors, please put here the full title of the article

\begin{center}
\begin{huge}
\textit{Singular integrals along}

\textit{variable codimension one subspaces}

\end{huge}
\end{center}

\vspace{1cm}

%%      -------------------------------------------------------------------------------
%%      -------------------------- AUTHORS AND AFFILIATIONS ----------------------------
%%      -------------------------------------------------------------------------------
%% Authors, please put here your full names and affiliations

\begin{minipage}[t]{.28\textwidth}
\begin{center}
{\large{\bf{Odysseas Bakas}}} \\
\vskip0.15cm
\footnotesize{University of Patras}
\end{center}
\end{minipage}
\hfill
\noindent
\begin{minipage}[t]{.28\textwidth}
\begin{center}
{\large{\bf{Francesco Di Plinio}}} \\
\vskip0.15cm
\footnotesize{Universit\`a di Napoli ``Federico II''}
\end{center}
\end{minipage}
\hfill
\noindent
\begin{minipage}[t]{.28\textwidth}
\begin{center}
{\large{\bf{Ioannis Parissis}}} \\
\vskip0.15cm
\footnotesize{Universidad del Pa\'is Vasco and Ikerbasque} 
\end{center}
\end{minipage}

\medskip
\begin{center}
\begin{minipage}[t]{.28\textwidth}
\begin{center}
{\large{\bf{Luz Roncal}}} \\
\vskip0.15cm
\footnotesize{Basque Center for Applied Mathematics, Universidad del Pa\'is Vasco and Ikerbasque} 
\end{center}
\end{minipage}
\end{center}

\vspace{1cm}

%%% Please replace "James Mustard" below 
%%% with the name of the managing editor for your submission.
%%% If you are unsure about their identity
%%% please ask an editor-in-chief about.

\begin{center}
\noindent \em{Communicated by Monica Visan}
\end{center}
\vspace{1cm}

%%      -------------------------------------------------------------------------------
%%      -------------------------- BEGIN ABSTRACT ----------------------------
%%      -------------------------------------------------------------------------------
%% Authors, please put here the ABSTRACT and KEYBOARDS

\noindent \textbf{Abstract.} \textit{This article deals with maximal operators on $\R^n$ formed by taking arbitrary rotations of tensor products of a $d$-dimensional H\"ormander--Mihlin multiplier with the identity in $n-d$ coordinates, in the particular \emph{codimension 1} case $d=n-1$. These maximal operators are naturally connected to differentiation problems and maximally modulated singular integrals such as Sj\"olin's generalization of Carleson's maximal operator. Our main result, a weak-type $L^{2}(\R^n)$-estimate  on band-limited functions, leads to several corollaries. The first is a sharp $L^2(\R^n)$ estimate for the maximal operator restricted to a finite set of rotations in terms of the cardinality of the finite set. The second is a version of the Carleson--Sj\"olin theorem. In addition, we obtain that functions in the Besov space $B_{p,1}^0(\R^n)$, $2\le p <\infty$, may be recovered from their averages along a measurable choice of codimension $1$ subspaces, a form of Zygmund's conjecture in general dimension $n$.}
\vskip0.3cm

\vspace{0.5cm}

%%      -------------------------------------------------------------------------------
%%      -------------------------- BEGIN ARTICLE ----------------------------
%%      -------------------------------------------------------------------------------
%% Authors, copy the body of your paper here

\section{Introduction}

Let $n\geq 2$ be the linear dimension of the Euclidean space $\R^n$ endowed with the canonical basis $\{e_1,\ldots, e_n\}$ and let $d\in\{1,\ldots,n-1\}$. Singular integrals in $d$ variables embed in  $\R^n$ when made to act trivially in the $n-d$ variables perpendicular to a chosen $d$-dimensional subspace. More specifically,
 each pair $({U},\mathscr K)$ of  orientation preserving rotation ${U}\in \mathrm{SO}(n)$ and   tempered distribution $\mathscr{K}\in \mathcal S'(\R^d)$ induces a map $T: \mathcal S(\R^n) \to \mathcal S'(\R^n)$ by
\begin{equation}\label{eq:defdir-3} 
\langle T  f,g \rangle\coloneqq  \int_{\R^{n-d} } \left \langle  f\circ {U} (\cdot, y  ) * \mathscr{K}, g \circ {U} (\cdot,y)\right \rangle \d y =\left \langle f\circ  {U} * \left(\mathscr{K}\otimes \delta\right), g\circ {U} \right \rangle ,
\end{equation}
with $\delta\in \mathcal S'(\R^{n-d})$ denoting the Dirac distribution supported at $0\in\R^{n-d}$.
If the Fourier transform of $\mathscr{K}  $ is a bounded function $m\in L^\infty(\R^d)$ then  $Tf$ coincides with the continuous bounded  function
\begin{equation}\label{eq:defdir-2}
Tf(x)=\int_{\R^n} m (V \Pi_\sigma \xi) \widehat f(\xi) e^{2\pi i \langle x, \xi\rangle}\, \d \xi,\qquad x\in\R^n,
\end{equation} 
where $V=U^{-1}$, $\sigma$ is the preimage through {$V$} of the subspace $\R^d\coloneqq \mathrm{span}\{e_1,\ldots,e_d\}$, and $\Pi_\sigma$ denotes orthogonal projection on $\sigma$. 
Our interest lies in the further particular case where  $m$ is a  H\"ormander--Mihlin multiplier on $ \R^d $, namely \begin{equation}
\label{e:mihlin} \| m\|_{\mathcal M_A(d)} \coloneqq \sup_{0\leq |\alpha| \leq A}\, \sup_{\eta \in \R^d   } |\eta|^{|\alpha|} \left| D^\alpha   m(\eta )\right|
\end{equation}
is finite. If $A\geq d+1$ say,  an application of the H\"ormander--Mihlin theorem on each $\sigma$-fiber of $\R^n$ and Fubini's theorem   tell us that \eqref{eq:defdir-2} extends to a bounded operator on $L^p(\R^n), 1<p<\infty$.

The broad object of interest of this article are maximal versions of \eqref{eq:defdir-2} where the rotation $V\in \mathrm{SO}(n)$ depends measurably, or in some other specific fashion, on the point $x\in \R^n$. Note that the multiplier \eqref{eq:defdir-2} is allowed to exhibit singularities on the $(n-d)$-dimensional subspace
$
\sigma ^\perp = (V^{-1}\R^{d})^\perp$.   
In light of this fact and of the invariance of the Mihlin norms \eqref{e:mihlin}, an efficient parametrization of the multipliers \eqref{eq:defdir-2} must involve the corresponding oriented Grassmanian  $\Gr(d,n)$,  namely the space of oriented $d$-dimensional  subspaces of $\R^n$. When equipped with the canonical distance,  $\Gr(d,n)$ is a compact metric space which is isometrically isomorphic to the quotient group
\[\frac{\mathrm{SO}(n)}{ \mathrm{SO}(d) \otimes \mathrm{SO}(n-d)}.
\]
The description of the isomorphism is useful to us and can be loosely given as follows. Thinking again   $\R^d=\mathrm{span}\{e_1,\ldots,e_d\}\in\Gr(d,n)$, a subspace $\sigma  \in \Gr(d,n)$ is identified with  the class of matrices
  $\mathcal V(\sigma)\subset\mathrm{SO}(n)$ which map  $\sigma $ to $\R^d$.
  % and $\sigma^\perp $ to $(\R^d)^\perp$. OBVIOUS
  Any two  elements $V , O \in \mathcal V(\sigma) $ are related by
\begin{equation}
\label{e:SO(d)}V= ( Q \otimes P )O  , \qquad Q\in  \mathrm{SO}(d), \quad P  \in \mathrm{SO}(n-d).
\end{equation}
For this reason, fixing in the remainder of our discussion a family
\begin{equation}
\label{e:introOsig}
\{{O_\sigma}\in \mathcal V(\sigma):\,\sigma \in \Gr(d,n)\}
\end{equation}
of representatives, the class of rotated singular integrals  \eqref{eq:defdir-2} associated to a single multiplier $m\in \mathcal M_A(d)$ may be parametrized as
\begin{equation}\label{eq:defdir-1}
T_{m}f(x;\sigma,Q) \coloneqq \int_{\R^n} m (Q{O_\sigma}\Pi_\sigma \xi  ) \widehat f(\xi) e^{2\pi i \langle x, \xi\rangle}\, \d \xi,\qquad x\in\R^n,
\end{equation}
with parameters $\sigma \in \Gr(d,n)$ and $Q \in \mathrm{SO}(d)$.
Note that \eqref{eq:defdir-1} does not depend on the component $P\in \mathrm{SO}(n-d)$ in \eqref{e:SO(d)}, reflecting the trivial action of $T$ in the $\sigma^\perp$-variables, cf.\ \eqref{eq:defdir-2}. 
Also note that the family in \eqref{e:introOsig} can be chosen arbitrarily. A specific family is  explicitly constructed in Lemma~\ref{lem:rot} with the additional property  that the map $\sigma\mapsto O_\sigma$ is $C^1$; see  Remark~\ref{rmrk:Osmooth}. 
Finally, the directional operator \eqref{eq:defdir-1} may be further interpreted by rewriting formally  
\begin{equation}\label{eq:singlem}{}
T_{m} f(x;\sigma,Q)=\int_{\R^d}   f(x- {V^{-1}} t )\, {K} (t )  \d t, \qquad V= Q{O_\sigma},\quad x\in\R^n,
\end{equation} 
with ${K} (t)$ the  kernel of ${ \mathscr K=\widehat{m}}$. 

 The maps \eqref{eq:defdir-1}, or alternatively \eqref{eq:singlem}, are termed \emph{$\Gr(d,n)$-subspace singular integral operators.} For instance, they encompass $\sigma$-subspace averaging operators, corresponding to the choice 
$\mathscr{K}\in L^1(\R^d)$ in \eqref{eq:defdir-3}, as well as  the $d$-dimensional Riesz transform in the directions of  $\sigma$, obtained by choosing the vector distribution
\[
\mathscr{K}(y) \coloneqq \mathrm{p.v.} \frac{y}{|y|^{d+1}}\in \mathcal{S'}(\R^d)^d.
\]

Subspace singular integrals, along \emph{variable choices} of $d$-dimensional subspaces $\sigma$ are connected to several deep lines of investigations in harmonic analysis and partial differential equations. This article focuses on the   \emph{codimension 1} case $n=d+1$. Higher codimension cases are more singular and very few instances of operators have been treated, none beyond the $L^2$ threshold. A full account of the history of this circle of problems is postponed to the end of this introduction.
 
Below,  we consider  the maximal operator associated to a family of multipliers ${\bf{m}}=\{m_\sigma\in \mathcal M_A(d) :\,  \sigma \in \Gr(d,n)\}$, 
 \begin{equation} \label{e:thisismain}
T^\star_{{\bf{m}}} f(x) \coloneqq \sup_{\sigma\in \Gr(d,n)} \sup_{Q\in \mathrm{SO(d)}}  \left|T_{m_{\sigma} } f(x;\sigma, Q)\right|, \qquad x\in \R^n,
\end{equation}
or linearized variants thereof where the variable choices of $\sigma,Q$ are expressed by measurable functions $\sigma(\cdot):\R^n\to \Gr(d,n)$ and  $Q(\cdot):\R^n\to \mathrm{SO}(d)$, possibly under additional regularity assumptions on these functions. { When $d=1$, the operators $T_{{\bf{m}}} ^\star$ recover the familiar directional maximal averaging operator, and the maximal directional Hilbert transform, for suitable choices of symbols ${\bf{m}}$.} Concerning  \eqref{e:thisismain}, its maximal nature shows how the choice of the family ${O_\sigma}$ in \eqref{e:introOsig} is immaterial.

 Our main result concerns the action of  \eqref{e:thisismain} on  frequency band-limited functions. In order to state it we fix a smooth radial function $\zeta:\R^n\to \R_+$ such that $\mathrm{supp}(\zeta)\subset\{\xi\in\R^n:\, 1< |\xi|<\frac{3}{2}\}$ and define  for $k\in \R$
\begin{equation} \label{eq:Pann}
P_{k}  f(x)\coloneqq \int_{\R^n} \zeta\left( 2^{-k}|\xi| \right)\widehat f(\xi) e^{2\pi i {\langle x, \xi \rangle}}\, \d \xi,\qquad x\in\R^n.
\end{equation}

%%%%%%%%%%%%%%%%%%%%%%%%%%%%%% THEOREM THEOREM THEOREM
\begin{theorem} \label{thm:main}There exists $A=A(d)$ such that the following holds. Suppose that the family  ${\bf{m}}\coloneqq\{m_\sigma \in L^\infty( \R^d):\sigma \in \Gr(d,n)\}$ is such that 
\begin{equation}
\begin{split}
\label{e:admin4}  \|{\bf{m}}\|_{\mathcal M_A(\Gr(d,n) )} \coloneqq
\sup_{\sigma,\tau \in \Gr(d,n)  }\left[\|m_\sigma\|_{\mathcal M_A(d)} + \log\left(\e+ {\textstyle\frac{1}{\dist(\sigma,\tau)}}\right)\|m_\tau -m_\sigma\|_{\mathcal M_A(d)}  \right] & \leq 1.
\end{split}
\end{equation}
Referring to \eqref{e:thisismain}, there holds 
\begin{equation}
\label{eq:mainest}
\left\| T^\star_{{\bf{m}}} \circ P_0 :L^2(\R^n)\to L^{2,\infty}(\R^n) \right\| \lesssim 1, \qquad  \left\| T^\star_{{\bf{m}}} \circ P_0 :L^p(\R^n)  \right\|  \lesssim_p 1, \quad 2<p<\infty. 
\end{equation}
\end{theorem}
%%%%%%%%%%%%%%%%%%%%%%%%%%%%%% THEOREM THEOREM THEOREM

%%%%%%%%%%%%%%%%%%%%%%%%%%%%%% REMARK REMARK REMARK
\begin{remark} The norm \eqref{e:mihlin} is invariant under   isotropic scaling of $\R^d$.  
Condition \eqref{e:admin4} inherits such invariance, whence \eqref{eq:mainest} holds  with the same constants for $T^\star_{{\bf{m}}} \circ P_k$ as well, $k\in \R$.
\end{remark}
%%%%%%%%%%%%%%%%%%%%%%%%%%%%%% REMARK REMARK REMARK

%%%%%%%%%%%%%%%%%%%%%%%%%%%%%% REMARK REMARK REMARK
\begin{remark} A sufficient  condition for   assumption \eqref{e:admin4} to hold is the following.  Let 
\[
\|{\bf{m}}\|_{\mathcal M_{A,1}(\Gr(d,n))}\coloneqq \sup_{\sigma\in\Gr(d,n)} \sup_{0\leq |\beta|\leq 1 } \|{\mathscr D}_\sigma  ^\beta  m_\sigma \|_{\mathcal M_A(d)}
\]
where the differential operator $\mathscr{D}_\sigma$ is the natural invariant vector of $\sigma$-derivatives; an explicit description is given in Lemma~\ref{lem:alg}. Since 
\[
\|{\bf{m}}\|_{\mathcal M_{A}(\Gr(d,n))}\lesssim \|{\bf{m}}\|_{\mathcal M_{A,1}(\Gr(d,n))},
\]
 the conclusion of Theorem~\ref{thm:main}  holds under the assumption $\|{\bf{m}}\|_{\mathcal M_{A,1}(\Gr(d,n))}\leq 1$.
\end{remark}
%%%%%%%%%%%%%%%%%%%%%%%%%%%%%% REMARK REMARK REMARK
 
In order to further emphasize the interpretation of our operators as directional multipliers and draw a formal parallel to the two-dimensional case of \cite{LacLi:tams}, let us rewrite  \eqref{eq:singlem} in yet another form. Equip $\R^n$   with  the standard basis $(e_1,\ldots,e_d,e_n)$ and  define an orthonormal basis on $\sigma$ by setting 
\[
v_j ^\sigma\coloneqq V^{-1} e_j,\qquad j\in\{1,\ldots,d\}, \quad \mathrm{span}(v_1 ^\sigma,\ldots,v_d ^\sigma)= \sigma,\qquad   V\in\mathcal V(\sigma).
\]
Then, calculate
\[
V\Pi_\sigma\xi = V\sum_{j=1} ^d \l \xi,v_j ^\sigma\r v_j ^\sigma=\sum_{j=1} ^d \l \xi,v_j ^\sigma\r e_j = \left(\langle\xi, v_1 ^\sigma\rangle,\ldots,\langle\xi, v_d ^\sigma\rangle\right)\in\R^d 
\]
so that \eqref{eq:defdir-1} takes the form
\[
\begin{split}
T_{m}   f(x;\sigma,Q)=\int_{\R^n}  m \left(\langle\xi, v_1 ^\sigma\rangle,\ldots,\langle\xi, v_d ^\sigma\rangle \right) \widehat f (\xi) e^{2\pi i {\langle x, \xi \rangle}}\, \d \xi,\qquad  x\in \R^n.
\end{split}
\]
Then Theorem~\ref{thm:main} tells us that the operator
\[
f\mapsto \sup_{\sigma\in\Gr(d,n)} \sup_{v_1 ^\sigma ,\ldots,v_d ^\sigma} \left| \int_{\R^n}  m_\sigma \left(\langle\xi, v_1 ^\sigma\rangle,\ldots,\langle\xi, v_d ^\sigma\rangle \right) \widehat f (\xi) e^{2\pi i {\langle \cdot , \xi \rangle}}\, \d \xi \right|
\]
is of weak-type $(2,2)$ and strong type $(p,p)$, $p>2$, whenever ${\bf{m}}=\{m_\sigma\in L^\infty(\R^d): \,\sigma\in\Gr(d,n)\}$ satisfies $\|{\bf{m}}\|_{\mathcal M_A(\Gr(d,n))}\leq 1$; above, the inner supremum is over all orthonormal bases of $\sigma$. Specializing to a single multiplier $m\in \mathcal M_{ A}(d)$ we record below an immediate corollary.

%%%%%%%%%%%%%%%%%%%%%%%%%%%%%% COROLLARY COROLLARY COROLLARY
\begin{corollary} Let $m\in \mathcal M_A(d)$ and consider the maximal directional multiplier
\[
T^\star _m f(x)\coloneqq \sup_{(v_1,\ldots,v_d)}\left|\int_{\R^n}m\left(\l \xi, v_1\r,\ldots,\l\xi, v_d\r\right )\widehat f(\xi) e^{2\pi i \l x, \xi\r}\,\d \xi\right|,\qquad x\in\R^n,
\]
where the supremum is taken over all orthonormal $d$-tuples $(v_1,\ldots,v_d)\in (\R^n)^d$. Then $T^\star _m \circ P_0$ maps $L^2(\R^n)$ to $L^{2,\infty}(\R^n)$ and $L^p(\R^n)$ to $L^p(\R^n)$ for all $2<p<\infty$.
\end{corollary}
%%%%%%%%%%%%%%%%%%%%%%%%%%%%%% COROLLARY COROLLARY COROLLARY
We encourage the reader to compare with the very familiar two-dimensional directional multipliers, \cite{Dem,LacLi:tams}, given in the form
\[
Tf(x,v)\coloneqq \int_{\R^2} m\left(\langle \xi, v\rangle \right) \widehat f(\xi) e^{2\pi i {\langle x, \xi \rangle}}\,\d \xi,\qquad (x,v)\in\R^2\times \mathbb S^1,
\]
when $ m\in \mathcal M_A(1)$ is a one-dimensional H\"ormander--Mihlin multiplier. In  the two-dimen\-sional case, $\mathrm{SO}(1)$ is trivial, so that such dependence may be omitted.

 When the band limited restriction imposed by precomposing with $P_k$ is lifted  and no  regularity assumptions are placed on the subspace choice function,   such function $\sigma(\cdot)$ is allowed to be  oriented along a suitable Kakeya--Nikodym set and the subspace  singular integrals \eqref{e:thisismain}  are in general unbounded on $L^p(\R^n)$.
 One particularly deep line of investigation is seeking for  suitable regularity assumptions on $\sigma$  bypassing  the above mentioned counterexample and ultimately leading to $L^p$-bounds.  For instance, if $\sigma(\cdot)$ is $1$-Lipschitz and $ m_\sigma(\cdot)$ is supported on frequency scales $\gg 1$, then  Kakeya-type counterexamples are avoided. Zygmund suggested that --- in this context ---  a suitably truncated version of the averaging directional operator along a Lipschitz choice of subspace should be bounded, at least above a critical $L^p$-space which can be identified for each dimension $n$ and codimension $n-d$,
see \cite{DPP_adv}. The corresponding version of this conjecture for multipliers $m_\sigma$,  allowed to be singular along $\sigma ^\perp$ is usually attributed to Stein. These are the Zygmund and Stein conjectures alluded to above and underlying the investigations in this paper.

 For the case of directional singular integrals, namely when $m_\sigma$ is allowed to be singular along $\sigma^\perp$, further counterexamples of non-Kakeya type do exist even when the range of $\sigma(\cdot)$ has special (e.g. lacunary) structure as has been exhibited in \cite{Karag,LMP}. For this reason it is customary, at least in the case $d=1$, to study $\sigma$ with finite range and prove optimal bounds in terms of the cardinality of such range. Optimal cardinality bounds for directional singular integrals are mostly known in the two-dimensional case, see \cite{Dem,DeDP}, while in higher dimensions $n\geq 2$ and $d=1$ optimal bounds for directional multipliers and $\sigma$ with lacunary range are contained in \cite{ADP,DPP2}.

In the present paper we prove the first such result in higher dimensions $n\geq 2$ and codimension $n-d=1$. It is worth mentioning that while most of the results in the literature deal with maximal singular integrals generated by a single multiplier, the statement below allows for a log-H\"older dependence of the family $m_\sigma$ in $\sigma\in \Gr(d,n)$. A two-dimensional partial analogue has appeared in \cite{KaLa}, where the authors even allow for a measurable with respect to $\sigma$ choice of multipliers, albeit with sub-optimal bounds in terms of the cardinality.

%%%%%%%%%%%%%%%%%%%%%%%%%%%%%% COROLLARY COROLLARY COROLLARY
\begin{corollary}\label{cor:Nmaximal} Suppose that the family  ${\bf{m}}\coloneqq\{m_\sigma \in L^\infty( \R^d):\, \sigma \in \Gr(d,n)\}$ satisfies \[\|{\bf{m}} \|_{\mathcal M_A(\Gr(d,n))} \leq 1.\] With reference to \eqref{eq:defdir-1}, the maximal operator
\begin{equation}
\label{e:thisismainsig}
T^\star_{\Sigma,{\bf{m}}} f(x) \coloneqq \sup_{\sigma \in \Sigma} \sup_{Q\in \mathrm{SO}(d)} \left| T_{m_\sigma}f(x; \sigma,Q)\right|, \qquad x\in \R^n,\qquad \Sigma\subset \Gr(d,n),
\end{equation}
satisfies the norm inequality
	\[
	\sup_{\substack{\Sigma\subset \Gr(d,n)\\ \#\Sigma\leq N}} \big\| T^\star_{\Sigma, {\bf{m}}} \big \|_{L^2(\R^n)} \lesssim_n \log N 
	\]
and this is best possible up to the implicit dimensional constant. 
\end{corollary}
%%%%%%%%%%%%%%%%%%%%%%%%%%%%%% COROLLARY COROLLARY COROLLARY

From {the proof of} Theorem~\ref{thm:main}, it is also possible to deduce weak $(2,2)$ and strong $(p,p)$ bounds, with $2<p<\infty$, for maximally truncated versions of \eqref{e:thisismain}.
 When considering the multiplier $m_{\sigma}\equiv 1$, this corresponds to a maximal subspace averaging operator. The obtained mapping properties   can then be used  to deduce subspace Lebesgue differentiation theorems for functions in the Besov spaces $B_{p,1}^s(\R^n)$, $s\ge 0$. These recover past results due to  Murcko  \cite{Mur},   see also \cite{AFN,Naibo}, for the case $s>0,n=2$, and appear to be new in the cases $s=0$ or $n>2$.
 A detailed statement is given in Theorem \ref{thm:cvg},  Subsection \ref{sub:besov}.

 %%%%%%%%%%%%%%%%%%%%%%%%%%%%%% SECTION SECTION SECTION
Further motivation, as well as influence on the proof techniques, for Theorem \ref{thm:main} comes from the connection with maximally modulated singular integrals.   In the two-dimen\-sio\-nal case, $n=2=d+1$, there are well-explored   \cite{BaTh,LacLi:tams} ties between the band-limited behavior of the maximal directional Hilbert transform  and the boundedness of the Carleson maximal partial inverse Fourier transform operator \cite{car}. In a similar fashion,  Theorem~\ref{thm:main} implies   $L^p$-estimates for maximally modulated H\"ormander--Mihlin multipliers in the spirit of Sj\"olin \cite{sjolin}.

%%%%%%%%%%%%%%%%%%%%%%%%%%%%%% THEOREM THEOREM THEOREM
\begin{theorem}\label{thm:carsjo} Suppose $m \in \mathcal M_A(d)$.  
Then, the conclusion of Theorem~\ref{thm:main}  implies that the Carleson--Sj\"olin operator
\[
\mathrm{CS}[f](x)\coloneqq \sup_{N\in\R^d} \left|\int_{\R^d} \widehat {f}(\eta)m(\eta+N)  e^{2\pi i \langle \eta, x\rangle} \, \d \eta\right|, \qquad x\in \R^d,
\]
maps $L^2(\R^d)$ to $L^{2,\infty}(\R^d)$ and $L^p(\R^d)$ to $L^p(\R^d)$ for all $2<p<\infty$.
\end{theorem}
%%%%%%%%%%%%%%%%%%%%%%%%%%%%%% THEOREM THEOREM THEOREM

The transference type argument leading to the proof of Theorem~\ref{thm:carsjo} is a modified version of the observation, commonplace in the literature, that any horizontal frequency cutoff to a half-line of a function supported on a thin vertical frequency strip may be obtained by applying a   directional Hilbert transform with a suitably chosen slope. The connection to Carleson-type theorems and, more generally, modulation invariant operators, is in fact also apparent from the proof of Theorem~\ref{thm:main} itself, which, much like its predecessors in \cite{LacLi:mem,BatRMI}, borrows from the Lacey--Thiele argument for Carleson's maximal operator, \cite{LTC}. 

Due to the higher dimensional nature of the problem considered, several new conceptual and technical difficulties arise. A first element of proof, which is only relevant in ambient dimension $n\geq 3$, addresses the possibility of choosing several rotations mapping $\R^d$ to a given $\sigma\in\Gr(d,n)$. The underlying rotation invariance allows for arbitrary choices of coordinates on each $\sigma$, introducing artificial discontinuities to the problem in hand and is readily appreciated in formula \eqref{eq:defdir-1}. The necessary reduction which allows us to dispose of the arbitrary rotation $Q$ in \eqref{eq:defdir-1} and instead use canonical smooth rotations is contained in \S\ref{sec:redux}; the study of the model case of maximally rotated translation invariant Calder\'on--Zygmund operators is instructive and is also presented in \S\ref{sec:redux}. Additionally, the higher dimensional time-frequency analysis of the current paper combined with the directional nature of the operator under study leads to a novel model sum in terms of wave-packets conforming to the geometry of the operator, and requiring us to work with a choice function, dictating the measurable choice of subspaces from $\Gr(d,n)$, which localizes in higher dimensional frequency caps.

%%%%%%%%%%%%%%%%%%%%%%%%%%%%%% SECTION SECTION SECTION

The final paragraph of this introduction serves as a more comprehensive historical account of past progress on  maximal directional singular and averaging operators. This subject has attracted  considerable attention in the last fifty years, mainly  in connection to the Kakeya maximal conjecture. In this context, the dimension $d$ of the averaging manifold is 1, and  the connection may be described through the $L^p$-bounds, $1<p<\infty$, for the single and multi-scale directional maximal functions
\[
\M_{V,1} f(x)\coloneqq \sup_{v\in V}  \int_{-\frac12}^{\frac12} |f(x-vt)|\, \d t,\quad 
\M_V f(x)\coloneqq \sup_{r>0}\sup_{v\in V} \ \frac{1}{2r}\int_{-r} ^r |f(x-vt)|\, \d t, \quad x\in \R^n
\]
associated to a generic  $V\subseteq \SS^{n-1}$.
The existence of Kakeya sets implies that these bounds blow up as $\delta\to 0^+$ when $V$ is a  $\delta$-net on $\SS^{n-1}$. When $n=2$, tight upper bounds for $\delta$-nets are known and due by Str\"omberg, \cite{Stromberg}, and C\'ordoba, \cite{Cor77}, who e.g.\ prove sharp logarithmic blowup in terms of $\delta$ above $L^2(\R^2)$, which is the critical space for the Kakeya maximal function in $\R^2$. These results imply the maximal Kakeya conjecture in two dimensions and for example we have
\[
\|M_{V,1}\|_{L^2(\R^2)}\eqsim    \sqrt{|\log \delta|},\qquad \|\M_V\|_{L^2(\R^2)}\eqsim |\log \delta|.
\]
 In parallel it was known that if $V$ has additional structure, namely if it is a \emph{lacunary set of directions}, then $\M_V$ is bounded on all $L^p(\R^2)$ spaces, $1<p\leq \infty$, as was shown collectively by \cite{CorFeflac,NSW,SS}. Bateman \cite{Bat} showed that in fact $\M_V$ is bounded on some (equivalently all) $L^p(\R^2)$ spaces, $1<p\leq \infty$, if and only if $V$ is a lacunary set. This means that if $V$ is an arbitrary infinite set of directions then $\M_V$ will be generically unbounded so that one can assume that $\#V<+\infty$ and seek for best possible bounds in terms of the cardinality $\#V$. This was done by Katz who showed the best possible bounds for $\M_{V,1}$ in \cite{KB}, and for $\M_V$ in \cite{Katz}.  These bounds match the optimal bounds for $\delta$-nets stated above.

 The maximal directional Hilbert transform 
\[
H_Vf(x)\coloneqq \sup_{v\in V} |H_v f(x)|,\qquad H_v f(x) \coloneqq \mathrm{p.v.} \int_{\R} f(x- v t)\frac {\d t}{t},\qquad x\in\R^2,
\]
is always unbounded if $\#V=\infty$; this was shown for $L^2(\R^2)$ by Karagulyan in \cite{Karag} and for all $L^p(\R^n)$ by {\L}aba, Marinelli and Pramanik in \cite{LMP}.  The upper bounds for $H_V$ and $V$ lacunary were shown by Demeter and one of the authors, \cite{DeDP}; see also \cite{DPP}, resulting to the sharp estimate
\[
\sup_{\substack {V \,\textrm{lacunary}\\ \#V \leq N}}
\|H_V\|_{L^2(\R^n)}\eqsim \sqrt{\log N}.
\]

The contributions \cite{Dem,DeDP}, containing what is essentially the $n=2$ case of Corollary \ref{cor:Nmaximal}, have already been recalled. Finally the $n=2$ case of Theorem~\ref{thm:main} has a partial analogue in a result of Lacey and X.\ Li \cite{LacLi:tams}. Therein, the authors pursued this question as an intermediate step towards the two-dimensional version of the Stein conjecture. Further results are available under other  specific structural assumptions on the linearizing choice function for the directions $v$. The articles   \cite{BatRMI} by Bateman and \cite{BaTh} by Bateman and Thiele deal with the case where $v$ is almost horizontal and does not depend on the vertical variable. In particular, the latter article showed that this structure entails  $L^p(\R^2)$-bounds  for all $3/2<p<\infty$. Variants of these assumptions have been dealt with by S.\ Guo \cite{Guo, Guo2};  see also \cite{DPGTZK} for a comprehensive result of this type. In a different direction Bourgain \cite{Bourgain} has proved that a suitably truncated version  of $M_{v(\cdot)}$ is bounded on $L^2(\R^2)$ when $v$ satisfies a certain curvature condition, which in particular is verified by real analytic vector fields; see  \cite{Guo3} for a geometric proof. Analogous results for    the corresponding truncated directional Hilbert transform  along $v$ real analytic fields are due to obtained by Stein and Street, \cite{StStr}.

In higher dimensions, $n\geq 3$ most of the known results concern the case of codimension $n-1$ so that $d=1$. Parcet and Rogers \cite{PR2} extended the notion of lacunarity to any dimension and showed that $\M_V$ is bounded on all $L^p(\R^n)$-spaces under the assumption that $V$ is lacunary. A partial converse in the spirit of Bateman is also proved in \cite{PR2} but a full characterization is still pending; for $d>1$ there is currently no notion of lacunary subsets of $\Gr(d,n)$ in the literature. For arbitrary $V\subset \SS^{n-1}$ sharp $L^2(\R^n)$-bounds were shown for $\M_{V,1}$ by  two of the authors using the polynomial method, \cite{DPPalg}. Regarding $\M_{V}$, the special case where $V\subset \SS^{n-1}$ is a $\delta$-net is considered in \cite{Kim}. Sharp $L^2(\R^n)$-bounds for maximal single-scale subspace averages are proved in \cite{DPP_adv} for any $1\leq d \leq n-1$; these include the case of codimension $n-d=1$ where $L^2(\R^n)$ is the critical space. 

Before this work, maximal directional singular multipliers had only been studied for $d=1$. When $V$ is lacunary and $n\geq 3$, Accomazzo and two of the authors have proved  sharp $L^p(\R^n)$-bounds  in \cite{ADP}. In \cite{KP}, Kim and Pramanik prove the sharp $L^2(\R^n)$-bounds for $V\subset \SS^{n-1}$ arbitrary with $\#V<+\infty$. It should be noted that lacunary sets yield the same bound of the order $\sqrt{\log \#V}$ for all $L^p$-norms of the maximal directional singular integral in all dimensions. For $V\subset \SS^{n-1}$ arbitrary this is no longer the case and logarithmic cardinality bounds can only be expected above the critical exponent $p=n$. This problem would however entail the resolution of the Kakeya conjecture and is currently wide open. Note that Corollary~\ref{cor:Nmaximal} above provides the first sharp bound for maximal directional multipliers and $d>1$; since in our case the codimension $n-d=1$ the critical space for this problem is still $L^2(\R^n)$, as in the case $n=2=d+1$.

\subsection*{Structure of the paper} Section \ref{sec:geom} contains a few geometric preliminaries on the metric and differential structure of $\Gr(d,n)$. In Section \ref{sec:redux}, Theorem \ref{thm:main} is reduced to a simpler version by removing the $\mathrm{SO}(d)$ maximality, see Proposition \ref{p:mainpf} for a precise statement. Section \ref{sec:maxcarleson} is devoted to the deduction of Corollary \ref{cor:Nmaximal} and Theorem  \ref{thm:carsjo}. Section \ref{sec:tiles+model} further reduces Proposition \ref{p:mainpf} to the estimation of a time-frequency model operator, which is finally performed in Section \ref{sec:trees}. The concluding Section \ref{sec:compl} is devoted to complementary results and questions on maximally truncated and bi-parametric   variable $\Gr(d,n)$-subspace singular integrals.

{
\subsection*{Acknowledgments} The authors would like to thank the anonymous referee for an expert reading and for providing several helpful and detailed comments.

O. Bakas was partially supported by the projects CEX2021-001142-S, RYC2018-025477-I,  PID2021-122156NB-I00/AEI/10.13039/501100011033 funded by Agencia Estatal de Investigaci\'on and acronym ``HAMIP'', Juan de la Cierva Incorporaci\'on IJC2020-043082-I, grant BERC 2022-2025 of the Basque Government, and by the funding programme ``MEDICUS'' of the University of Patras.

F. Di Plinio has been partially supported by the National Science Foundation under the grants  NSF-DMS-2000510, NSF-DMS-2054863, as well as  NSF-DMS-1929284 while in residence at the Institute for Computational and Experimental Research in Mathematics in Providence, RI, during the \emph{Harmonic Analysis and Convexity} program in Fall 2022. F.\ Di Plinio is also partially supported by the FRA 2022 Program of University of Napoli Federico II, project ReSinAPAS -
Regularity and Singularity in Analysis, PDEs, and Applied Sciences.

I. Parissis is partially supported by the project funded by Agencia Estatal de Investigaci\'on PID2021-122156NB-I00/AEI/10.13039/501100011033  and  with acronym ``HAMIP'', grant T1615-22 of the Basque Government and IKERBASQUE.

L. Roncal is partially supported by the projects CEX2021-001142-S, CNS2023-143893, RYC2018-025477-I and PID2020-113156GB-I00/AEI/10.13039/501100011033 and  with  ac\-ro\-nym ``HAPDE'' funded by Agencia Estatal de Investigaci\'on, grant BERC 2022-2025 of the Basque Government and IKERBASQUE.}

%%%%%%%%%%%%%%%%%%%%%%%%%%%%%% SECTION SECTION SECTION
\section{Geometric preliminaries}\label{sec:geom} 
% \subsection{\texorpdfstring{$\Gr(d,n)$}{Gr(d,n)} as a metric space and notational remarks}Recall that $\Gr(d,n)$, $1\leq d <n$ is a compact metric space 
% when endowed with 
% \begin{equation}\label{eq:dist}
% \dist'(\sigma,\tau)\coloneqq  \sup_{\omega\in \SS^{n-1}}|\Pi_\sigma \omega - \Pi_{\tau} \omega|, \qquad \sigma, \tau \in \Gr(d,n),
% \end{equation} 
% and the map
{The remainder of this article is concerned with the codimension $1$ case, $d=n-1$. In this case $\Gr(d,n)$ is naturally identified with $\mathbb S^d$ via the isometric isomorphism
\[
\sigma \mapsto \sigma^\perp, \qquad \Gr(d,n)\to \Gr(n-d,n),
\]
This identification is exploited with the notation $v_\sigma \in \mathbb S^d$ as the unit normal vector to $\sigma$. We equip $\Gr(d,n)=\Gr(n-1,n)$ with the metric
\begin{equation}\label{eq:dist3}
\dist(\sigma,\tau)\coloneqq |v_\sigma-v_\tau| = \sqrt{2}\sqrt{1-\cos[\theta(v_\sigma, v_\tau)]} \eqsim|\sin [\theta(v_\sigma, v_\tau)]| 
\end{equation}
where $\theta(u,v)$ stands for the convex angle between $u,v\in \mathbb S^d$. The equivalence constant implied by $\eqsim$ is obviously independent of the dimension.}
 % Because of \eqref{eq:dist2} above it will be convenient for us to redefine the distance on $\Gr(d,n)$ for the special case $d=n-1$ in hand as follows
 % \begin{equation}\label{eq:dist3}
 % 	\dist(\sigma,\tau)\coloneqq |v_\sigma-v_\tau|,\qquad \sigma,\tau\in \Gr(d,n),\qquad d=n-1.
 % \end{equation}
We set up a unified notation for balls in metric spaces. The standing convention is  that $\xi\in X$ then $B(\xi,r)$ is the ball in the metric space $X$, centered at $\xi$ and of radius $r>0$. The default metric on linear subspaces of  $\R^n$ and  $\mathbb S^d$ is the Euclidean metric and we write $B^m(R)$ for the Euclidean ball in $\R^m$, centered at $0$ and having radius $R>0$.  We also adopt the following notation for annuli in a generic metric space $X$  
\[
\Ann(x,r,R)\coloneqq B(x,R)\setminus \overline{ B(x,r)}, \qquad 0<r<R,\quad  x\in X.
\]
When $X$ is a linear space the shorthand   $\mathrm{Ann}(r,R)$  is also used in place of $\mathrm{Ann}(0,r,R)$ and $B(R)$ is used in place of $B(0,R)$. 

%%%%%%%%%%%%%%%%%%%%%%%%%%%%%% SECTION SECTION SECTION
\subsection{Smooth families of rotations} The next lemma constructs a typical family of Lipschitz rotations of $\R^n$.  These are essentially two-dimensional rotations between $\sigma$ and $\tau$ acting in the plane perpendicular to the common $(d-1)$-dimensional subspace $\sigma\cap \tau =(\mathrm{span}\{v_\sigma, v_\tau\})^\perp$. 

%%%%%%%%%%%%%%%%%%%%%%%%%%%%%% LEMMA LEMMA LEMMA
\begin{lemma}\label{lem:rot}Let $\sigma,\tau\in \Gr(d,n)$ {with $\dist(\sigma,\tau)<1/2$}. There exist $O_{ \tau,\sigma}\in \mathrm{SO}(n)$ such that 
\[
O_{ \tau,\sigma} \sigma =\tau, \qquad O_{ \tau,\sigma} v_\sigma =v_\tau, \qquad \left\|O_{ \tau,\sigma}-\mathrm{Id}\right\| =\dist(\sigma,\tau).
\]
%{
%Furthermore,  
%\[
%\rho\left(v_\sigma,\frac{(O_{ \tau,\sigma}-I)y}{|(O_{ \tau,\sigma}-I)y|}\right)\leq \frac{\theta}{2}, \qquad \forall y\in \sigma\quad\text{with} \quad(O_{ \tau,\sigma}-I)y\neq 0.
%\]}
\end{lemma}
%%%%%%%%%%%%%%%%%%%%%%%%%%%%%% LEMMA LEMMA LEMMA

%%%%%%%%%%%%%%%%%%%%%%%%%%%%%% PROOF PROOF PROOF
\begin{proof} If $\sigma =\tau$ take $O_{\tau,\sigma}=\mathrm{Id}$. Otherwise define $\zeta\coloneqq\sigma \cap \tau\in\Gr(d-1,n)$; see Figure~\ref{fig:rot}. Since $\zeta^\perp= \mathrm{span}\{v_\sigma, v_\tau\}$ and  $  \sigma \cap \zeta^\perp $ is a one-dimensional subspace of $\zeta^\perp$, we may pick the unique unit vector $u_\sigma \in \sigma \cap \zeta^\perp$ satisfying the equation
\[
v_\tau = -(\sin \theta)u_\sigma  + (\cos \theta) v_\sigma, \qquad \theta\coloneqq \arccos(\langle v_\sigma, v_\tau \rangle).
\]
Let $O_{\tau,\sigma}$ be the rotation that acts as the identity on $\zeta$ and such that
\[
O_{\tau,\sigma}u_\sigma =u_\tau \coloneqq (\cos \theta)u_\sigma  +(\sin \theta) v_\sigma, \qquad O_{\tau,\sigma}v_\sigma =v_\tau. 
\]

%%%%%%%%%%%%%%%%%%%%%%%%%%%%%% FIGURE FIGURE FIGURE
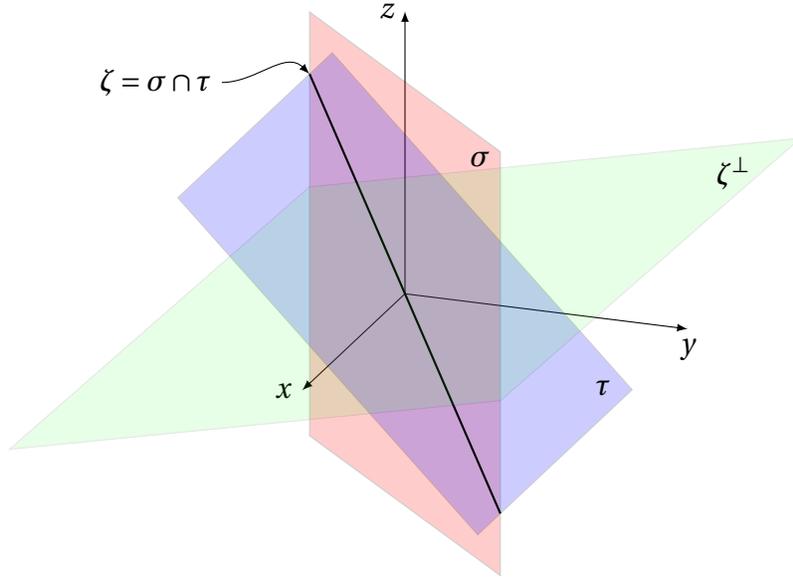
\begin{figure}[!ht]
\centering
\tdplotsetmaincoords{70}{110}
\begin{tikzpicture}[tdplot_main_coords,scale=1]
	\draw[-latex] (0,0,0) -- (4,0,0) node[left] {$x$};
	\draw[-latex] (0,0,0) -- (0,4,0) node[below] {$y$};
	\draw[-latex] (0,0,0) -- (0,0,4) node[left] {$z$};
\tdplotsetrotatedcoords{45}{0}{0}
\begin{scope}[tdplot_rotated_coords]
\draw[fill=red,opacity=0.2] (-3,0,-3) -- (-3,0,3) -- (3,0,3) -- (3,0,-3) -- cycle node at (-0.1,1.15,1.7) [opaque=0]{$\sigma$};
\end{scope}
 \tdplotsetrotatedcoords{90}{45}{0}
\begin{scope}[tdplot_rotated_coords]
\draw[fill=blue,opacity=0.2] (-3,-3,0) -- (-3,3,0) -- (3,3,0) -- (3,-3,0) -- cycle node at (2.7,1.1,.7) [opaque=0]{$\tau$} ;
\draw[thick](-3,{3/sqrt(2)},0) coordinate(x) --(3,{-3/sqrt(2)},0);
\draw[fill=green,opacity=0.1] (-2,{-4*sqrt(2)},-3) -- (2,{4*sqrt(2)},-3) -- (2,{4*sqrt(2)},3) -- (-2,{-4*sqrt(2)},3) -- cycle node at (1.7,1.7,4.0) [opaque=0]{$\zeta^\perp$}  ;
\end{scope}
\node[anchor=south east,align=center] (line) at (3,-1.5,3.5) {$\zeta=\sigma\cap \tau $};
\draw[-latex] (line) to[out=0,in=135] (x);
\end{tikzpicture}
\caption{A figure for the proof of Lemma~\ref{lem:rot}}
\label{fig:rot}
\end{figure}
%%%%%%%%%%%%%%%%%%%%%%%%%%%%%% FIGURE FIGURE FIGURE
Note that $u_\tau$ is a unit vector orthogonal to $\zeta$ and $v_\tau$, therefore $\tau = \textrm{span}\,\{\zeta,u_\tau\} =
O_{\tau,\sigma}\sigma$. As $O_{\tau,\sigma}-\mathrm{Id} = [O_{\tau,\sigma}-\mathrm{Id}] \Pi_{\zeta^\perp}$, the equality
\[
\left\|O_{ \tau,\sigma}-\mathrm{Id}\right\|  = \sqrt{2}\sqrt{1-\langle v_\sigma, v_\tau \rangle}= |v_\sigma-v_\tau| =  \dist(\sigma, \tau)
\]
follows by immediate two-dimensional trigonometry.
\end{proof}
%%%%%%%%%%%%%%%%%%%%%%%%%%%%%% PROOF PROOF PROOF

We give below an explicit description of the tangent bundle   $\mathrm{T}\Gr(d,n) $. This is used in the sequel to give a concrete notion of derivatives of functions defined on $\Gr(d,n)$. In turn this will help us establish the smoothness of the map $\Gr(d,n)\ni\sigma\mapsto O_{\rho,\sigma}$ for fixed $\rho\in\Gr(d,n)$.

%%%%%%%%%%%%%%%%%%%%%%%%%%%%%% LEMMA LEMMA LEMMA
\begin{lemma} \label{lem:alg}Let $\sigma\in  \Gr(d,n) $ and let $ \{v^ \sigma _1, \cdots, v^ \sigma _d\}$  be an orthonormal basis of $\sigma$ and as usual $\sigma^\perp=\R v_\sigma$. The tangent space $\mathrm{T}_\sigma \Gr(d,n) $ is the linear span of 
\[
X_{v_j ^\sigma}  \in\mathrm{so}(n),\quad j=1,\ldots, d, \qquad X_{v_j ^\sigma}  [v _k ^\sigma]=
\begin{cases} 
\vec 0, & k\neq j \vspace{.8em}
\\ 
v_\sigma, & k=j \end{cases}, \quad 1\leq k \leq d, \qquad X_{v_j ^\sigma} v_\sigma = - v^ \sigma _j,
\]
where $\mathrm{so}(n)$ stands for $n\times n$ real skew-symmetric matrices. Furthermore, if $\xi\in \R^n$ is a fixed  vector then the following hold,
\[
\partial_{ {v_j ^\sigma}	 } (\pr_\sigma \xi)  = X_{v_j ^\sigma} \xi, \qquad \partial_{ {v_j ^\sigma}} (\pr_{\sigma^\perp} \xi)  =- X_{v_j ^\sigma} \xi, \qquad \,j=1,\ldots,d,
\]
and the vector of derivatives $\mathscr D_\sigma$ can be described as
\[
\mathscr D _\sigma =\big(\partial_{ v_1 ^\sigma}	  ,\ldots,\partial_{ {v_d ^\sigma}	 } \big),\qquad \sigma\in\Gr(d,n).
\]
\end{lemma}
%%%%%%%%%%%%%%%%%%%%%%%%%%%%%% LEMMA LEMMA LEMMA

% %%%%%%%%%%%%%%%%%%%%%%%%%%%%%% REMARK REMARK REMARK
% \red{
% \begin{remark}
% \label{rem:Dbeta}
% Given $\beta\in\{0,1\}^d$, we use the notation
% $$
% D_{\sigma}^{\beta}F(\sigma):=\partial_{ v_1 ^\sigma}^{\beta_1}	  \cdots\partial_{ {v_d ^\sigma}}^{\beta_d}	 F(\sigma), \qquad \sigma \in \Gr(d,n).
% $$
% \end{remark}}
% %%%%%%%%%%%%%%%%%%%%%%%%%%%%%% REMARK REMARK REMARK

%%%%%%%%%%%%%%%%%%%%%%%%%%%%%% PROOF PROOF PROOF
 \begin{proof}
 % [Proof of Lemma \ref{lem:alg}] 
 A basis of the tangent space $\mathrm{T}_\sigma \Gr(d,n)$ at $\sigma\in\Gr(d,n)$ can be given by $(X_{v_1 ^\sigma},\ldots, X_{v_d ^\sigma})$ where we define $X_{v_j ^\sigma}$ for $1\leq j \leq d$ to be the tangent vector at $t=0$ to the $\Gr(d,n)$-valued curve
\[
\sigma(t;j)\coloneqq \mathrm{span}\,\{v^\sigma _1,\ldots ,  (\cos t)v^\sigma _j +(\sin t) v_\sigma,\ldots , v^\sigma _d\}.
\]
For the second claim we compute
\[
\begin{split}
&\quad \Pi_{\sigma(t;j)} \xi - \Pi_\sigma\xi \\ &  =  \l \xi,(\cos t) v^\sigma _j +(\sin t )v_\sigma\r ((\cos t )v^\sigma _j +(\sin t) v_\sigma)-\l\xi,v_j ^\sigma\r v_j ^\sigma
\\
&= (\sin t \cos t)\left[\l \xi, v^ \sigma _j\r v_\sigma + \l \xi, v_\sigma\r v^\sigma _j \right]  +[(\cos t)^2-1]\l \xi,v^\sigma _j\r v^\sigma _j + (\sin t)^2\l \xi,v_\sigma\r v_\sigma
\\ 
&=  (\sin t \cos t)\left[\l \xi, v^\sigma_ j\r v_\sigma + \l \xi, v_\sigma\r v^\sigma _j \right] + O(t^2)
\end{split}
\] 
so that 
\[
\partial_{ v_j ^\sigma}(\pr_\sigma \xi) = \lim_{t\to 0}\frac{\Pi_{\sigma(t;j)} \xi - \Pi_\sigma\xi }{t} =  \l \xi, v^\sigma _j\r v_\sigma + \l \xi, v_\sigma\r v^\sigma _j =X_{v_j ^ \sigma}  \xi
\] 
as claimed. Finally, 
\[0=\partial_{{v_j ^\sigma}} \xi  =  \partial_{{v_j ^\sigma}}  (\pr_\sigma \xi)  +\partial_{{v_j ^\sigma}}(\pr_{\sigma^\perp} \xi)
\] 
whence the  corresponding conclusion.
\end{proof}
%%%%%%%%%%%%%%%%%%%%%%%%%%%%%% PROOF PROOF PROOF

The  last lemma of this short section contains a computation of  the $\sigma$-derivatives of the $\mathrm{SO}(n)$-valued map
\[
\Gr(d,n)\ni\sigma \mapsto
R_\sigma \coloneqq O_{\sigma,\rho} \qquad \sigma \in \mathrm{Gr}(d,n),
\]
where  $O_{\sigma,\rho}$ references  Lemma~\ref{lem:rot} and $\rho\in  \mathrm{Gr}(d,n)$ is kept fixed and will henceforth be omitted from the notation. In general, for $\rho\in \Gr(d,n)$ fixed we consider the following mappings:
\begin{itemize}
	\item [$\cdot$] $\sigma\mapsto R_\sigma$ is the rotation on the plane $\mathrm{span}\{v_\sigma,v\}$ with $R_\sigma v =v_\sigma$; remember that $v=v_\rho$.
  \item [$\cdot$] $\sigma\mapsto \theta_\sigma\coloneqq \arccos(\langle v _\sigma,v\rangle).$
  \item [$\cdot$] $\sigma \mapsto u_\sigma$, where $u_\sigma$ is the unit vector of $\rho$ which is perpendicular to $\sigma \cap \rho$ with $\l u_\sigma,v_\sigma\r<0$, and therefore lies in the plane $\mathrm{span}\{v_\sigma,v\}$.
  \item [$\cdot$] $\sigma \mapsto U_\sigma\coloneqq R_\sigma u_\sigma$.
\end{itemize}  
%%%%%%%%%%%%%%%%%%%%%%%%%%%%%%% FIGURE FIGURE FIGURE
\begin{figure}[ht]
\centering
\def\svgwidth{330pt}
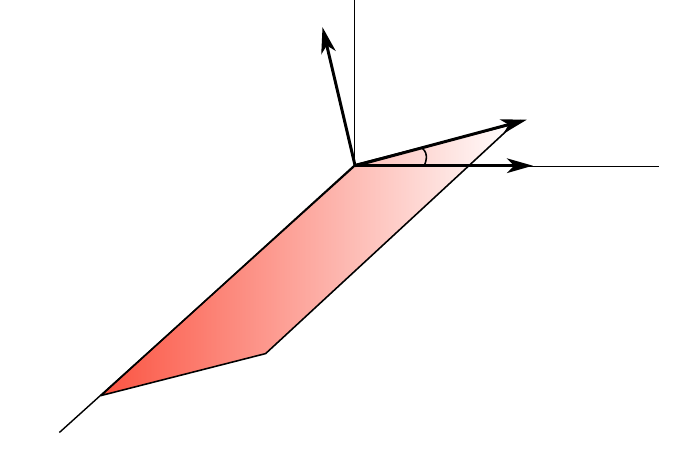
\caption{A figure for the proof of Lemma~\ref{l:derOsig}.}
\label{fig:smooth}
\end{figure}
%%%%%%%%%%%%%%%%%%%%%%%%%%%%%%% FIGURE FIGURE FIGURE
With these notations, depicted in Figure~\ref{fig:smooth}, there holds
\begin{equation}
\label{e:Osig}
U_\sigma= (\cos \theta_\sigma) u_\sigma + (\sin \theta_\sigma) v   , \quad 
R_\sigma v= v_\sigma, \quad R_\sigma w=w \quad \forall w \in \sigma \cap \rho.
\end{equation}
For the statement and proof of the lemma below it will be convenient to write $\R^n$ as the  direct sum 
\[
\R^n=(\sigma \cap \rho) \oplus \R u_\sigma \oplus \R  v.
\]
Note then that $\sigma$ can be described in the form $\sigma=\sigma\cap \rho \oplus \R U_\sigma$.  Now Lemma \ref{lem:alg} shows that in order to determine $\mathscr D _\sigma ^\beta O_\rho(\sigma)$ for $|\beta|\leq 1$ it suffices to determine $\partial_{v_j^\sigma} R_\sigma$ for each vector from a suitably chosen basis of $\sigma$. This is effectively accomplished by the next lemma, as one may complete $U_\sigma$ to an orthonormal basis of $\sigma$ by choosing $d-1$ orthonormal vectors in $\sigma \cap \rho$. 

%%%%%%%%%%%%%%%%%%%%%%%%%%%%%% LEMMA LEMMA LEMMA
\begin{lemma} \label{l:derOsig} Let $a,b,c\in \R$ and $\rho \in\Gr(d,n)$ be fixed. For every $\sigma\in\Gr(d,n)$ we have
\[
\partial_{U_\sigma} R_\sigma [w+ b u_\sigma + c v  ]= b v_\sigma -c U_\sigma, \qquad \forall w \in \sigma \cap \rho.
\]
Furthermore for all $ w \in \sigma \cap \rho$ there holds
\[
\partial_{w} R_\sigma [z+ a w+  b u_\sigma + c v  ]=   [b \varpi(\theta_\sigma)-c] w + a \varpi(\theta_\sigma) u_\sigma + a v   \qquad \forall z \in \sigma \cap \rho \cap w^\perp,
\]
for the bounded function $\varpi(\beta)\coloneqq \frac{\cos \beta -1}{\sin \beta}$.
\end{lemma}
%%%%%%%%%%%%%%%%%%%%%%%%%%%%%% LEMMA LEMMA LEMMA

%%%%%%%%%%%%%%%%%%%%%%%%%%%%%% REMARK REMARK REMARK
\begin{remark}\label{rmrk:Osmooth} In particular, $\sup_{0\leq |\beta|\leq 1}\|\mathscr D^\beta  _\sigma R_\sigma \|\lesssim  1+ |\varpi(\theta)| \lesssim_d 1 .$ Integrating along the geodesic connecting $\sigma,\tau \in \Gr(d,n)$, we obtain the operator norm estimate
\[
\left\| R_ \tau  -R_\sigma  \right\|= \left\| R_ \tau R_\sigma ^{-1} -\mathrm{Id} \right\|   \lesssim \dist(\tau,\sigma).
\]
Throughout the paper we reserve the notation $O_\sigma\in\mathcal V(\sigma)$ for the rotations
\[
O_\sigma = O_{\R^d,\sigma}\in \mathrm{SO}(n),\qquad O_\sigma \sigma=\R^d,\qquad O_\sigma v_\sigma=e_n.
\]
Then, if the family $\{O_\sigma\in\mathcal V(\sigma):\sigma \in \Gr(d,n)\}$ is constructed via Lemma~\ref{lem:rot}, we will have
\begin{equation}\label{e:opnormest}
\sup_{0\leq |\beta|\leq 1} \|\mathscr D^\beta _\sigma O_\sigma\| \lesssim  1,\qquad \|O_\sigma-O_\tau\|\lesssim \dist(\sigma,\tau),
\end{equation}
and these rotations will be used as the canonical family  in \eqref{e:SO(d)}.
\end{remark}
%%%%%%%%%%%%%%%%%%%%%%%%%%%%%% DEFINITION DEFINITION DEFINITION

%%%%%%%%%%%%%%%%%%%%%%%%%%%%%% PROOF PROOF PROOF
\begin{proof}[Proof of Lemma \ref{l:derOsig}] In order to compute $\partial_{U_\sigma } R_\sigma$, we need to compute the limit
\begin{equation}
\label{e:limiting}
\partial_{U_\sigma}R_\sigma=
\lim_{t\to 0} \frac{R_ {\tau(t)} - R_\sigma }{t}
\end{equation}
where $\tau:[0,\pi/2]\to \Gr(d,n)$ is the $\Gr(d,n)$-valued curve
\[
\begin{split}
\tau=\tau(t)\coloneqq(\sigma \cap \rho) \oplus & \mathrm{span}\{U_\tau \}, \quad   U_\tau \coloneqq (\cos t)U_\sigma +(\sin t)v_\sigma=R_\tau u_\sigma,
\\
&  v_{\tau}= -(\sin t)U_\sigma +(\cos t)v_\sigma,
\end{split}
\]  
and note that
\[
\tau(0)=\sigma,\qquad U _{\tau(0)}=U_\sigma,\qquad v_{\tau(0)}=v_\sigma,
\]
and that $R_\tau$ is a rotation by an angle $\arccos(\l v_\tau,v \r)=\theta_\sigma+t$ in the plane $\mathrm{span}\{U_\sigma,v_\sigma\}=\mathrm{span}\{v_\tau,v\}$, and $v_\tau=R_\tau v $. It is then quite easy to see that
\[
\partial_{U_\sigma} R_\sigma v = -U_\sigma,\qquad 
\partial_{U_\sigma} R_ \sigma  u_\sigma = v_\sigma,\qquad 
\partial_{U_\sigma}R_ \sigma  w = 0 \quad \forall w \in \sigma \cap \rho .
\]
This proves the first identity in the statement of the lemma.

In order to prove the second claim of the statement it will suffice to compute  $\partial_{w} R_\sigma$ for a generic vector $w\in\sigma \cap \rho $ of unit norm. Fix one such $w$ and set $\zeta\coloneqq \sigma \cap \rho \cap w^\perp $. The derivative $\partial_{w} R_\sigma$ will be calculated by repeating the limiting procedure \eqref{e:limiting} for
\[
\tau=\tau(t)\coloneqq \zeta \oplus \mathrm{span}\{w_{\tau(t)},U_\sigma\}, \quad 
w_{\tau} \coloneqq(\cos t) w+(\sin t) v_\sigma , \quad v_{\tau} \coloneqq -(\sin t) w+(\cos t) v_\sigma .
\]
Note again that we have
\[
\tau(0)=\sigma,\qquad w_{\tau(0)}=w,\qquad v_{\tau(0)}=v_\sigma.
\]
As $R_\sigma,R_\tau$ both act as the identity on $\zeta$, one has
\[
\partial_{w} R_\sigma z=0 \qquad \forall z\in \zeta
\]
and it suffices to test the action of $\partial_{w} R_\sigma$ on the basis vectors of $\zeta^\perp=\mathrm{span}\{u_\sigma,w,v\}$.
The vectors $\{\widetilde{w},\widetilde{u}\}$ given by
\begin{equation}
\label{e:wup}
\begin{split}&\widetilde{w}\coloneqq \frac{1}{\beta(t)} \left[(\cos t) w -\frac{\sin t}{\sin \theta_{\sigma}} u_\sigma\right], \qquad \widetilde{u}\coloneqq \frac{1}{\beta(t)} \left[\frac{\sin t}{\sin \theta_{\sigma}} w+ (\cos t) u_\sigma \right], \\ & \beta(t)\coloneqq \sqrt{1+(\sin t \cot \theta_{\sigma})^2}
\end{split}
\end{equation}
form an orthonormal basis of $\zeta^\perp \cap \rho=\mathrm{span}\{u_\sigma,w\}$ with, in addition, $\mathrm{span}\{\widetilde{w}\}=\zeta^\perp \cap\tau (t)\cap \rho$.  Therefore $R_\tau$ acts as a rotation by $\theta_\tau$ in the plane spanned by $\widetilde{u},v$, where
\[
\cos \theta_\tau = \langle v_{\tau},v \rangle = \cos t \cos \theta_\sigma, \qquad \sin \theta_\tau =\sqrt{1-(\cos t \cos \theta)^2},
\]
namely,
\[
R_\tau  \widetilde{w}=\widetilde{w}, \qquad R_\tau  \widetilde{u}=\cos [\theta_\tau] \widetilde{u} +\sin \theta_\tau v , \qquad R_\tau v=v_{\tau}.
\] 
Recalling our notation this means that $u_\tau=\widetilde{u}$ and $U_ \tau=R_\tau u_\tau=R_\tau \widetilde{u}$.
%%%%%%%%%%%%%%%%%%%%%%%%%%%%%% SECTION SECTION SECTION
\subsubsection*{Computing \texorpdfstring{$\partial_{w} R_\sigma  u_\sigma$}{}} The system \eqref{e:wup} may be easily inverted, giving
\begin{equation}
\label{e:wup2}
\begin{split}&w= \frac{1}{\beta(t)} \left[(\cos t) \widetilde{w} +\frac{\sin t}{\sin \theta} \widetilde{u}\right],
 \qquad u_\sigma=\frac{1}{\beta(t)} \left[-\frac{\sin t}{\sin \theta} \widetilde{w}+ (\cos t) \widetilde{u} \right].
\end{split}
\end{equation}
Also recalling \eqref{e:Osig}, \eqref{e:wup},
\[\begin{split} 
R_ \tau  u_\sigma &=   \frac{1}{\beta(t)} \left[-\frac{\sin t}{\sin \theta_\sigma} \widetilde{w}+ ((\cos t)^2\cos \theta_\sigma) \widetilde{u} + (\cos t)\sqrt{1-(\cos t\cos \theta_\sigma)^2} v \right] ,
\\
R_\sigma u_\sigma& = \frac{1}{\beta(t)} \left[-\frac{\cos \theta_\sigma \sin t}{\sin \theta_\sigma }\widetilde{w}+ (\cos t \cos \theta_\sigma )\widetilde{u} +  \beta(t) ( \sin \theta_\sigma) v \right].
\end{split}
\] 
Taking the difference and using some calculus
\[
[
R_\tau-R_\sigma] u_\sigma = \frac{(\sin t)\varpi(\theta_\sigma)}{\beta(t)} \widetilde{w} + \b{o(t )\widetilde{u}+o(t )v} .
 \]
Therefore, dividing by $t$ and  taking the limit 
\[
\partial_{w} R_\sigma  u_\sigma = \varpi(\theta_\sigma) w = \frac{\langle v_\sigma, v \rangle - 1 }{\sqrt{ 1 - \langle v_\sigma, v \rangle^2}}w.
\]

%%%%%%%%%%%%%%%%%%%%%%%%%%%%%% SECTION SECTION SECTION
\subsubsection*{Computing \texorpdfstring{$\partial_{w} R_\sigma w$}{}} Similarly
\[\begin{split}
R_\tau  w &=   \frac{1}{\beta(t)} \left[(\cos t)\widetilde{w}+ \frac{\sin t\cos t\cos \theta_\sigma}{\sin \theta_\sigma} \widetilde{u} +  \frac{\sin t}{\sin \theta_\sigma}\sqrt{1-(\cos t\cos \theta_\sigma)^2} v \right] ,
\\
R_ \sigma w &=w=  \frac{1}{\beta(t)} \left[(\cos t)\widetilde{w}+ \frac{\sin t}{\sin \theta_\sigma} \widetilde{u} \right].
\end{split}
\] 
Taking the difference and using some more calculus
\[
[
R_\tau -R_\sigma ] w= \frac{(\sin t)\varpi(\theta_\sigma)}{\beta(t)} \widetilde{u} +(\sin t) v+\b{ o(t ) v(t)}.
 \]
Therefore, dividing by $t$ and  taking the limit 
\[
\partial_{w} R_\sigma  w= \varpi(\theta_\sigma) u_\sigma+v  = \frac{\langle v_\sigma, v\rangle - 1 }{\sqrt{ 1 - \langle v_\sigma, v\rangle^2}}u_\sigma+ v.
\]

%%%%%%%%%%%%%%%%%%%%%%%%%%%%%% SECTION SECTION SECTION
\subsubsection*{Computing \texorpdfstring{$\partial_{w}R_\sigma v_\rho$}{}}
This one is the easiest, as
\[
\partial_{w} R_\sigma v =  \lim_{t\to 0}\frac{[(\cos t)-1 ]v_\sigma -(\sin t) w }{t}= -w
\]
and the proof is complete.
\end{proof}
%%%%%%%%%%%%%%%%%%%%%%%%%%%%%% PROOF PROOF PROOF

%%%%%%%%%%%%%%%%%%%%%%%%%%%%%% SECTION SECTION SECTION
\section{Removing the \texorpdfstring{$\mathrm{SO}(d)$}{SO(d)} invariance}\label{sec:redux}  The maximal operator \eqref{e:thisismain} is invariant under precompositions of the family $m_{\sigma}$ with elements $Q_\sigma \in \mathrm{SO}(d)$. In other words, $T^\star$ does not change if each $m_\sigma $ is replaced by $m_\sigma \circ Q_\sigma$ for some $Q_\sigma \in \mathrm{SO}(d) $ for each $\sigma \in \Gr(d,n)$. In this section,  Theorem \ref{thm:main} is deduced from an analogous estimate for the formally weaker maximal operator \eqref{e:thisissmooth} where $Q_\sigma\equiv $ the identity in $\mathrm{SO}(d)$. This deduction is summarized in  Proposition \ref{p:redux} below. The essential idea of the argument comes from an averaging argument showing the $L^p(\R^d)$-boundedness of maximally rotated $d$-dimensional H\"ormander--Mihlin multipliers, which is presented below as a warm-up.

%%%%%%%%%%%%%%%%%%%%%%%%%%%%%% PROPOSITION PROPOSITION PROPOSITION
\begin{proposition}
Let $m\in\mathcal M_A(d)$ for $A\geq  d+1 +\frac{d(d-1)}{2}$. Then  
\[
T^\star f(x)\coloneqq\sup_{Q\in\mathrm{SO}(d)}\left |\int_{\R^d} \widehat f(\eta) m(Q\eta) e^{2\pi i x\cdot \eta} \, \d\eta\right|,\qquad x\in\R^d,
\]
maps $L^p(\R^d) $ into itself for all $1<p<\infty$.
\end{proposition}
%%%%%%%%%%%%%%%%%%%%%%%%%%%%%% PROPOSITION PROPOSITION PROPOSITION

%%%%%%%%%%%%%%%%%%%%%%%%%%%%%% PROOF PROOF PROOF
\begin{proof} The idea of the proof is the same in all dimensions $d>1$. In order to make the argument more concrete we first  prove the case $d=2$.  A generic element  $O\in \mathrm{SO}(2)$ may be written  in the canonical basis of $\R^2$ as 
\[
O= \left[ \begin{array}{cc} \cos \theta & -\sin \theta \\ \sin \theta & \cos \theta   \end{array}\right]
\]
for some $\theta \in [0,2\pi)$. Let $O(\theta)$ be the element corresponding to $\theta \in [0,2\pi)$. We note that the coordinates of $O(\theta)$ are smooth functions of $\theta$ and we adopt the notation
\[
 O'(\theta)= \left[ \begin{array}{cc} -\sin \theta & -\cos \theta \\ \cos \theta & -\sin \theta   \end{array}\right]
\]
for the derivative. It is then useful to  introduce the family of multipliers
\begin{equation}
\label{e:mtheta}m_\theta(\eta)\coloneqq \langle \nabla m ( O(\theta)\eta), O'(\theta) \eta\rangle, \qquad \eta \in \R^2.
\end{equation}
   It is immediate to see that there exists a constant $c$ independent of $\theta$ such that 
$cm_\theta \in \mathcal{M}_{A-1} (2)$. Defining
 \[
Tf(x,\theta) \coloneqq \int_{\R^2} \widehat f(\eta) m(O(\theta)\eta) e^{2\pi i x\cdot \eta} \, \d\eta, \quad Sf(x, \theta) \coloneqq \partial_\theta Tf(x,\theta) =  \int_{\R^2} \widehat f(\eta) m_\theta( \eta) e^{2\pi i x\cdot \eta} \, \d\eta,
\]
one has the equality
\[
Tf(x,\theta)=Tf(x,0) + \int_0^\theta Sf(x, \tau) \, \d \tau, \qquad 0\leq \theta <2\pi,
\]
whence
\[
|Tf(\cdot,\theta)| \leq| Tf(\cdot,0) | + \int_0^{2\pi}  |Sf(\cdot, \tau)| \, \d \tau, \qquad 0\leq \theta <2\pi.
\]
Therefore, using Minkowski's inequality
\[\begin{split}
\left\| T^* f \right\|_p & =  \left\| \sup_{\theta\in [0,2\pi)} \left|Tf(\cdot,\theta)\right| \right\|_p \leq \| Tf(\cdot,0)\|_p + \left\| \int_0^{2\pi}  |Sf(\cdot, \tau)| \, \d \tau \right\|_p \\ &\leq \| Tf(\cdot,0)\|_p + 2\pi\sup_{\tau \in [0,2\pi)}\| Sf(\cdot,\tau)\|_p  \lesssim \|f\|_p
\end{split}
\]
using that  $m\in\mathcal{M}_{A} (2) $, $cm_\theta \in \mathcal{M}_{A-1} (2)$ uniformly, and the H\"ormander--Mihlin theorem. 

We then sketch the argument for $d>2$. Recall that $\mathrm{SO}(d)$ is a compact Lie group of dimension $D=\frac{d(d-1)}{2}$. Hence, there exist  $C>0$ and $0<\eps<1 $ so  that $\mathrm{SO}(d)$ may be covered by smooth charts
\[
Q_j=Q_j( \theta)=Q_j(\theta_1,\ldots, \theta_D):{[0,\eps)}^D\to \mathrm{SO}(d), \qquad 1\leq j \leq C.
\]
For example, if $d=3$ one may use axis-angle pairs to parametrize rotations. Namely, $ O_j(\theta_1, \theta_2, \theta_3) $ is the rotation by  angle $\theta_3$ in the plane perpendicular to the unit vector  
$$ 
N(\theta_1,\theta_2)=(\cos \theta_1 \sin \theta_2,\sin \theta_1 \sin \theta_2 , \cos\theta_2).
$$ 
Then $\mathrm{SO}(3)$ may be covered by $O(1)$ smooth axis angle charts with parameter $\eps=2^{-3}\pi$. 
 
By finite splitting it thus suffices to bound the maximal operator
\[
\sup_{ \theta \in [0,\eps)^D } |Tf(x, \theta)|\,\,\, \text{ where } \,\,\, Tf(x, \theta)\coloneqq \int_{\R^d} \widehat f(\eta) m(Q_j(\theta)\eta) e^{2\pi i x\cdot \eta} \, \d\eta, \quad x\in \R^n
\]
for each $j=1,\ldots, C$ fixed. For $S=\{k_1<\ldots<k_{\#S}\}\subset \{1,\ldots, D\}$ write $\partial_S= \partial_{k_1}\cdots \partial_{k_{\#S}}$. Let also $\Pi_S$ be the orthogonal projection on $\R^S\coloneqq \mathrm{span}\{e_k: \,k\in S\}$, so that $\theta = \Pi_S \theta \oplus \Pi_{S^\mathrm{c}} \theta.   $ If $ \tau \in \R^S$, it makes sense to write $ \theta = \tau \oplus  0_{S^\mathrm{c}} \in \R^D$ with the meaning that $\Pi_S  \theta= \tau $ and $\Pi_{S^\mathrm{c}} \theta = 0_{S^\mathrm{c}}$. 
The key of our argument is again that the multiplier operators
\begin{equation}
\label{e:mthetad}  T_Sf(x, \theta) =  \int_{\R^d} \widehat f(\eta) m_{ \theta,S}( \eta) e^{2\pi i x\cdot \eta} \, \d\eta, \qquad m_{\theta,S}\coloneqq \partial_S  m(Q_j( \theta)\cdot) ,
\end{equation}
satisfy $cm_{ \theta, S}\in \mathcal M_{A-\#S}(d) \subset \mathcal M_{A-D}(d)$ uniformly. Then one has the equality
\[
Tf(x, \theta) = Tf(x, 0)  +\sum_{\varnothing \subsetneq S \subseteq \{1,\ldots, D\} } \int_{\prod_{k \in S} [0,\theta_k]}   T_Sf(x, \tau \oplus  0_{S^\mathrm{c}} ) \d \tau, \qquad  \theta \in [0,\eps)^D ,
\]
and similarly to what we have done before,
\[
\left\|\sup_{ \theta \in [0,\eps)^D } \left|Tf(\cdot, \theta) \right|\right\|_p \leq 
\left\| Tf(\cdot, 0) \right\|_p + 2^D \sup_{\varnothing \subsetneq S \subset \{1,\ldots, D\} } \sup_{\tau\in [0,\eps]^{\#S}}\left\| T_Sf(\cdot , \tau \oplus  0_{S^\mathrm{c}} ) \right\|_p
\]
so that we end up with the latter supremum being controlled by $C\|f\|_p$ via an application of  the H\"ormander--Mihlin theorem. The proof is complete.
\end{proof}
%%%%%%%%%%%%%%%%%%%%%%%%%%%%%% PROOF PROOF PROOF

We are now ready to state and prove the anticipated reduction of Theorem \ref{thm:main}. 

%%%%%%%%%%%%%%%%%%%%%%%%%%%%%% PROPOSITION PROPOSITION PROPOSITION
 \begin{proposition} \label{p:redux} Let $\Sigma\subset \Gr(d,n)$ be arbitrary. Let $O_\sigma=O_{\R^d,\sigma}$, with reference to the family of rotations constructed in Lemma \ref{lem:rot}; cf. Remark~\ref{rmrk:Osmooth}. Consider the maximal operator
\begin{equation}\label{e:thisissmooth}
U^\star_{\Sigma,{\bf{m}}} f(x) \coloneqq \sup_{\sigma\in \Sigma}   | T_{m_\sigma}f(x,\sigma)|
\end{equation} 
where
$$
T_{m_\sigma}f(x,\sigma)\coloneqq  \int_{\R^n} m_\sigma (O_\sigma \Pi_\sigma \xi) \widehat f(\xi) e^{2\pi i \langle x, \xi\rangle}\, \d \xi, \quad x\in \R^n.
$$
For $1<p<\infty,1\leq q\leq \infty$, $D=\frac{d(d-1)}{2}$ and $T^\star_{\Sigma,{\bf{m}}}$ from \eqref{e:thisismainsig}, there holds 
\begin{multline*}
\sup\{\|T^\star _{\Sigma,{\bf{m}}}\|_{L^p(\R^n)\to L^{p,q}(\R^n)}: \, \|{\bf{m}}\|_{\mathcal M_{A+D}(\Gr(d,n))}\leq 1\}\\
\lesssim \sup\{\|U^\star _{\Sigma,{\bf{m}}}\|_{L^p(\R^n)\to L^{p,q}(\R^n)}: \, \|{\bf{m}}\|_{\mathcal M_A(\Gr(d,n))}\leq 1\}
\end{multline*}
with implicit constant depending upon dimension only. The same  statement holds if  $U^\star_{\Sigma, {\bf{m}}}$ is replaced by $U^\star_{\Sigma, {\bf{m}}}\circ P_0$ in the right hand side and $T^\star _{\Sigma,{\bf{m}}}$ is replaced by $T^\star_{\Sigma,{\bf{m}}}\circ P_0$ in the left hand side of the estimate above.
\end{proposition}
%%%%%%%%%%%%%%%%%%%%%%%%%%%%%% PROPOSITION PROPOSITION PROPOSITION

%%%%%%%%%%%%%%%%%%%%%%%%%%%%%% PROOF PROOF PROOF
\begin{proof} Fix a family ${\bf{m}}$ with $ \|{\bf{m}}\|_{\mathcal M_{A+D}(\Gr(d,n))}\leq 1$. 
By finite splitting, we obtain the claimed estimate for \eqref{e:thisismainsig} from the same claim on the restricted   maximal operator
\[
T^\star_{\Sigma,{\bf{m}}}  f(x) \coloneqq \sup_{\sigma \in \Sigma}
\sup_{Q\in \mathcal Q}|T_{m_\sigma}f(x;\sigma,Q)| 
\]
with
\[
T_{m_\sigma}f(x;\sigma,Q)\coloneqq \int_{\R^n} m_\sigma (QO(\sigma)\Pi_\sigma \xi) \widehat f(\xi) e^{2\pi i \langle x, \xi\rangle}\, \d \xi, \quad x\in \R^n, 
\]
where $\mathcal Q $ is range of the chart $Q_j$ for some $j=1,\ldots, C$;
the overloading of the symbol $T^\star$ creates no confusion. At this point, for $Q\in \mathcal Q$, let  $  \theta(Q)$ be such that $Q=  Q_j( \theta)$. 
Define the families 
\[
{\bf{m}}_{\theta, S}=\{m_{\sigma,  \theta, S}:\sigma \in \Gr(d,n)\}, \quad 
m_{\sigma,  \theta, S}(  \eta) \coloneqq \partial^S {m_\sigma( Q_j(\vec \theta) \eta)}, \quad  \theta \in [0, \eps)^D,\,S\subset\{1,\ldots, D\}.
\]
With this definition, observe that 
\[
T_{m_\sigma}f(x;\sigma,Q) = T_{m_{\sigma,  \theta(Q), \varnothing} }f(x;\sigma, \mathrm{Id}_{\R^d})= T_{m_{\sigma,  \theta(Q), \varnothing}} f(x,\sigma). 
\]
It is then routine to verify that
\begin{equation}
\label{e:averageest1}
\sup_{\theta\in  [0, \eps)^D} \sup_{S\subset \{1,\ldots,d\} }  \left\|{\bf{m}}_{\theta, S}\right\|_{\mathcal{M}_A(\Gr(d,n))} \lesssim \left\|{\bf{m}}\right\|_{\mathcal{M}_{A+D}(\Gr(d,n))}\leq 1.
\end{equation}
Arguing as in the previous proof we then have 
\[
T^\star_{{\bf{m}},\Sigma}  f(x) \leq U^\star_{
{\bf{m}}_{0,\varnothing },\Sigma} f(x) +  \sum_{\varnothing \subsetneq S \subseteq \{1,\ldots, D\} }   \int_{\prod_{k \in S} [0,\eps]}  U^\star_{\Sigma, 
{\bf{m}}_{  \tau \oplus   0_{S^\mathrm{c}},S}}f(x) \d \tau, \qquad \vec \theta \in [0,\eps)^D ,
\]
which turns into the norm inequality
\[
\|T^\star_{{\bf{m}}} f\|_{p,q} \lesssim_{p,q} \|U^\star_{
{\bf{m}}_{0,\varnothing },\Sigma} f\|_{p,q}  + 2^D \sup_{\varnothing \subsetneq S \subset \{1,\ldots, D\} } \sup_{\tau\in [0,\eps]^{\#S}}\left\|U^\star_{\Sigma, 
{\bf{m}}_{  \tau \oplus   0_{S^\mathrm{c}},S}}f \right\|_{p,q}.
\]
The norms above are then controlled by $\|f\|_p$ by assumption in view of estimate \eqref{e:averageest1}, and  the proof of the first claim is complete. For the second claim, it suffices to apply the above argument with $P_0 f$ in place of $f$.
\end{proof}
 %%%%%%%%%%%%%%%%%%%%%%%%%%%%%% PROOF PROOF PROOF
 Because of the reduction devised in Proposition \ref{p:redux} above, Theorem \ref{thm:main} with $A=A(d)+\frac{d(d-1)}{2}$ is obtained from the following proposition.

%%%%%%%%%%%%%%%%%%%%%%%%%%%%%% PROPOSITION PROPOSITION PROPOSITION
 \begin{proposition}\label{p:mainpf} Let $U^\star _{\Sigma,{\bf{m}}}$ be defined as in the statement of Proposition~\ref{p:redux}. There exists $A=A(d) $ such that the   operator $U^\star_{\Sigma,{\bf{m}}}  \circ P_0$ for $\Sigma=\Gr(d,n)$ maps $L^2(\R^n)$ to $L^{2,\infty}(\R^n)$ and  $L^p(\R^n)$ to $L^{p}(\R^n)$, $p>2$, uniformly over all families ${\bf{m}}$ satisfying $\|{\bf{m}}\|_{\mathcal M_A(\Gr(d,n))}\leq 1$.
\end{proposition}
%%%%%%%%%%%%%%%%%%%%%%%%%%%%%% PROPOSITION PROPOSITION PROPOSITION
The main line of proof of Proposition \ref{p:mainpf} occupies \S\ref{sec:tiles+model} and \S\ref{sec:trees}.

%%%%%%%%%%%%%%%%%%%%%%%%%%%%%% SECTION SECTION SECTION
\section{Maximal codimension-\texorpdfstring{$1$}{} multipliers and the Carleson--Sj\"olin theorem}\label{sec:maxcarleson} In this section we discuss the proofs of some of the consequences of our main result. 

%%%%%%%%%%%%%%%%%%%%%%%%%%%%%% SECTION SECTION SECTION
\subsection{Maximal codimension-\texorpdfstring{$1$}{1} multipliers and the proof of Corollary~\texorpdfstring{\ref{cor:Nmaximal}}{}}Let us recall our setup. Given $\Sigma\subset \Gr(d,n)$ with $\#\Sigma=N$, consider the maximal directional multiplier operator
\[
f\mapsto T ^\star _{\Sigma,{\bf{m}}} f(x)\coloneqq \sup_{\sigma \in\Sigma} \sup_{Q\in\mathrm{SO}(d)}|T_{m_\sigma} f(x;\sigma,Q)|,\qquad x\in \R^n,\qquad \Sigma\subset \Gr(d,n),
\]
where $T_{m_\sigma} f(x,\sigma,Q)$ is given by \eqref{eq:defdir-1} and the family of multipliers ${\bf{m}}=\{ m_ \sigma\in\mathcal M_{A}(d):\, \sigma\in\Gr(d,n)\}$ satisfies $\|{\bf{m}}\|_{\mathcal M_A(\Gr(d,n))}\leq 1$. We will prove that
\[
\sup_{\substack{\Sigma\subset \Gr(d,n)\\\#\Sigma\leq N}
}  \left \|  U_{\Sigma,{\bf{m}}} ^\star  f  \right \|_{L^2(\R^n)}  \lesssim \log N \|f\|_{L^2(\R^n)}
\]
which implies the conclusion of Corollary~\ref{cor:Nmaximal} by Proposition~\ref{p:redux}.

A first well known consequence of Theorem~\ref{thm:main} is that, in the case of $N$-multipliers, we can upgrade the single-annulus weak $(2,2)$ bound to a strong $(2,2)$ single-annulus  bound, at a cost of $\sqrt{\log N}$.

%%%%%%%%%%%%%%%%%%%%%%%%%%%%%% LEMMA LEMMA LEMMA
\begin{lemma}\label{lem:weak2strong} We have the bound
\[
\sup_{\substack{\Sigma\subset \Gr(d,n)\\\#\Sigma\leq N}} \sup_{k\in\Z} \|  U^\star_{\Sigma,{\bf{m}}} \circ P_kf  \|_{L^2(\R^n)}\lesssim \sqrt{\log N} \|f\|_{L^2(\R^n)},
\]
with implicit constant depending only upon dimension.

\end{lemma}
%%%%%%%%%%%%%%%%%%%%%%%%%%%%%% LEMMA LEMMA LEMMA

The lemma follows easily by the weak $(2,2)$ bound 
\[
\sup_{k\in\R}\| U^\star _{\Sigma,{\bf{m}}} (P_kf) \|_{L^{2,\infty}(\R^n)}\lesssim \|f\|_{L^2(\R^n)},
\]
  which is a special case of Theorem~\ref{thm:main}, together with trivial weak $(1,1)$ and strong $(p,p)$ estimates for $U ^\star _{\Sigma,{\bf{m}}} \circ  P_k$ for $2<p<\infty$, with bounds which are polynomial in $N$, and a Marcinkiewicz interpolation-type argument. The details can be found for example in \cite{Dem}*{Lemma 3.1}, but essentially the same argument has been used in several places, for example in \cite{Katz,Stromberg}.

With a strong $(2,2)$ bound in hand the proof of Corollary~\ref{cor:Nmaximal} follows a well known reduction, based on the Chang--Wilson--Wolff inequality \cite{CWW}, which allows us to essentially commute the supremum over $N$ Fourier multipliers with a Littlewood--Paley square function. The argument leading to this reduction was introduced in \cite{GHS}, while in the context of directional multipliers it has been extensively used; see for example \cite{ADP,Dem}. The punchline proposition is stated below. Here,  we use a smooth Littlewood--Paley partition of $\R^n$ given by means of
\[
( S_t f)^\wedge(\xi)\coloneqq \Psi\left(\frac{|\xi|}{t}\right )\widehat f(\xi),\qquad \xi\in\R^n,\quad t\in 2^\Z,
\]
where $\Psi$ is a smooth radial function with compact support $\mathrm{supp} \Psi \subset\{\frac12< |\xi|<2\}$ with the additional requirement that
\[
\sum_{t\in 2^\Z} \Psi\left(\frac{|\xi|}{t}\right)=1,\qquad \xi\in \R^n\setminus\{0\}.
\]

%%%%%%%%%%%%%%%%%%%%%%%%%%%%%% PROPOSITION PROPOSITION PROPOSITION
\begin{proposition}\label{prop:cww} Let $\{T_1,\ldots,T_N\}$ be Fourier multiplier operators in $\R^n$ with uniform bound
\[
\sup_{1\leq \nu \leq N} \|T_\nu\|_{L^2(\R^n)\to L^2(\R^n)}\leq 1.
\]
If $\{  S_t\}$ is a smooth Littlewood--Paley decomposition as described above then for $1<p<\infty$ there holds
\[
\left\| \sup_{1\leq \nu \leq N} |T_\nu f| \right \|_{L^p(\R^n)} \lesssim_{p,n}  \|f\|_{L^p(\R^n)}+	\sqrt{\log(N+10)}  \left\| \left( \sum_{t\in 2^\Z} \, \sup_{1\leq \nu \leq N} |(T_\nu  \circ S_t) f|^2 \right) ^{\frac12} \right\|_ {L^p(\R^n)} .
\]
\end{proposition}
%%%%%%%%%%%%%%%%%%%%%%%%%%%%%% PROPOSITION PROPOSITION PROPOSITION
The proof of this proposition can be found within \cite{DPGTZK}*{Proof of Corollary 1.14}; see also \cite{Dem}. Now note that our main result, Theorem~\ref{thm:main}, readily implies that
\[
\sup_{t\in2^{\Z}}\left \|  U^\star _{\Sigma,{\bf{m}}} \circ S_t f  \right \|_{L^{2,\infty}(\R^n)} \lesssim \|f\|_{L^2(\R^n)}
\]
by splitting the support of $\Psi$ into $O(1)$ pieces that match the size of the support of the auxiliary function $\zeta$ in the statement of Theorem~\ref{thm:main}, applying the conclusion of the theorem to each one of them, and adding them appropriately. Thus Lemma~\ref{lem:weak2strong} yields the strong $L^2(\R^n)$-bound
\[
\sup_{t\in2^{\Z}}\left \|  U^\star _{\Sigma,{\bf{m}}}\circ S_t f \right \|_{L^{2}(\R^n)} \lesssim
(\log \#\Sigma)^{\frac12} \|f\|_{L^2(\R^n)}
\]
for $N=\#\Sigma$ large. The conclusion of Corollary~\ref{cor:Nmaximal} now follows immediately by inserting the $L^2$-bound above in the conclusion of Proposition~\ref{prop:cww} for $p=2$. Note that the uniform $L^2$-bound, which is necessary for the application of Proposition~\ref{prop:cww},  is an immediate consequence of the standing assumption $\|{\bf{m}}\|_{\mathcal M_A(\Gr(d,n) )}\leq 1$ for the family of multipliers ${\bf{m}}$.

%%%%%%%%%%%%%%%%%%%%%%%%%%%%%%% SECTION SECTION SECTION
\subsection{The Carleson--Sj\"olin theorem} The purpose of this subsection is to provide the connection of our main theorem, Theorem~\ref{thm:main}, with the Carleson--Sj\"olin theorem, as formulated in Theorem~\ref{thm:carsjo}. We begin with a simple but useful geometric lemma.
%%%%%%%%%%%%%%%%%%%%%%%%%%%%%% LEMMA LEMMA LEMMA
\begin{lemma} \label{l:N} Let $\eps>0$ be a small parameter,  $\Sigma\subset \Gr(d,n)$ be a $\eps$-neighborhood of  \hspace{.2em}$\R^d=e_n^\perp $. For all $R>0$ 
the map \[
N:\Sigma \to \R^d, \qquad N(\sigma)\coloneqq R \left[\langle v_\sigma, e_n \rangle e_n-v_\sigma\right]
\] 
is onto the ball $B^d(c_o\eps R)\subset \R^d$, for some dimensional constant ${0<c_o<1}$. Furthermore
\[
|N(\sigma)-N(\tau)| \lesssim R \dist(\sigma, \tau), \qquad \sigma,\tau \in \Sigma
\]
with implicit absolute constant.
\end{lemma}
%%%%%%%%%%%%%%%%%%%%%%%%%%%%%% LEMMA LEMMA LEMMA

%%%%%%%%%%%%%%%%%%%%%%%%%%%%%% PROOF PROOF PROOF
\begin{proof} Clearly  $N(\R^{d})=0$. If $N\neq 0$ belongs to the ball $B^d(c_o\eps R)$, set
\[
v_\sigma\coloneqq-\frac{N}{R}  +\sqrt{1-\frac{|N|^2}{R^2}} e_n,\qquad \sigma\coloneqq v_\sigma ^\perp.
\] 
As $\dist(\sigma,\R^d)\sim \arccos\left( \langle v_\sigma,e_n \rangle \right) \sim \frac{|N|}{R}$, and ${|N|}\leq c_o\eps R$, it follows that $\sigma \in \Sigma$ if $c_o$ is suitably chosen. Also,
it is immediate to verify that $N(\sigma)=N$. 
The proof of the surjectivity claim is thus complete. The Lipschitz estimate is immediate, and its proof is thus omitted.
\end{proof}
%%%%%%%%%%%%%%%%%%%%%%%%%%%%%% PROOF PROOF PROOF
Note that the unit vector
\[
U_\sigma\coloneqq \frac{ \left[e_n-\langle v_\sigma, e_n \rangle v_\sigma\right]}{\sqrt{1-\langle v_\sigma, e_n \rangle^2}}, \qquad v_\sigma \neq e_n
\]
spans $\sigma \cap \mathrm{span}\{v_\sigma, e_n\}$. For $\sigma \neq e_n$, it is also convenient to define the unit vector $u_\sigma \coloneqq {N(\sigma)}/{|N(\sigma)|}$ spanning $\R^d \cap \mathrm{span}\{v_\sigma, e_n\}$. Below, let $\{O_\sigma:\, \sigma\in\Gr(d,n)\}$ be the $C^1$ family of rotations of Remark~\ref{rmrk:Osmooth}; we remember that $O_{\R^d}\equiv \mathrm{Id}$ and each $O_\sigma$ acts as the identity on $\sigma\cap \R^d$ and, in the orthogonal complement $\mathrm{span}\{v_\sigma, e_n\}$ as the rotation mapping $U_\sigma$ to $u_\sigma$, and consequently $v_\sigma$ to  $e_n$. With the notation of Lemma \ref{l:N} we have that, for ${|N|}\leq c_o\eps R$, ${\cos(|N|/R)\sim\langle v_{\sigma},e_n\rangle}$ and $\sin(|N|/R)\sim |N|/R$. A direct calculation entails the equality
\begin{equation}\label{eq:proj}
O_\sigma \Pi_{\sigma}\xi=\Pi_{\R^d}\xi+\left[\langle v_\sigma, e_n \rangle-1\right]\langle \xi, u_\sigma \rangle u_\sigma +\langle \xi, e_n\rangle \frac{N(\sigma)}{R}.
\end{equation}
The following definitions and observations readily yield the following lemma.

%%%%%%%%%%%%%%%%%%%%%%%%%%%%%% LEMMA LEMMA LEMMA
\begin{lemma}\label{lem:fam} Let $m\in \mathcal M_A(d)$ {be supported in $B^d(R_o)$ for some $R_o>1$} and consider the $\eps$-neighborhood  $\Sigma$ of $\,\R^d$ in $\Gr(d,n)$, constructed in Lemma~\ref{l:N} with $\eps$ sufficiently small depending upon dimension only. Let also $\{O_ \sigma: \,\sigma\in\Sigma\}$ be the family defined in Remark~\ref{rmrk:Osmooth}. Suppose  $F$ is a Schwartz function on $\R^n$ with $\supp \widehat{F} \subset \{\xi \in \R^n: \,  R/2<|\xi_n|< 2 R\}$ {for some $R>R_o$}. Define
\[
M(\xi,\sigma)\coloneqq m\left(\Pi_{\R^d}\xi+\left[\langle v_\sigma, e_n \rangle-1\right]\langle \xi, u_\sigma\rangle u_\sigma +\xi_n \frac{N(\sigma)}{R}\right),\qquad (\xi,\sigma)\in\R^n\times \Sigma,
\]
and let $T_{{\Sigma}}F(x,\sigma)\coloneqq (M(\cdot,\sigma)\widehat F)^\vee(x)$. There holds
\[
\left\|\sup_{\sigma\in\Sigma} |T_{{\Sigma}}F(\cdot,\sigma)|\,\right\|_{L^{2,\infty}(\R^n)}\lesssim \|F\|_{L^2(\R^n)},\quad \left\|\sup_{\sigma\in\Sigma}	 |T_{{\Sigma}}F(\cdot,\sigma)|\,\right\|_{L^p(\R^n)}\lesssim \|F\|_{L^p(\R^n)},\quad 2<p<\infty.
\]
\end{lemma}
%%%%%%%%%%%%%%%%%%%%%%%%%%%%%% LEMMA LEMMA LEMMA

%%%%%%%%%%%%%%%%%%%%%%%%%%%%%% PROOF PROOF PROOF
\begin{proof} {Consider the multiplier operator
% Let $\Gamma\in\mathcal S(\R)$ have compact support in $[-2^{-10},2^{-10}]$ and $\Gamma=1$ in $[-2^{-15},2^{-15}]$. By virtue of the support condition on $\widehat F$ we will have that
\[
F\mapsto \int_{\R^n}\widehat F(\xi) m( O_\sigma  \Pi_\sigma \xi)\e^{2\pi i \l \xi, x\r } \, \d \xi .
\]
For $\xi\in\R^n$ such that $m( O_\sigma \Pi_\sigma \xi)\widehat{F}(\xi)\neq 0$, the assumptions on the support on $m$ and $\widehat{F}$ imply that
\[
|\Pi_{\R^d}\xi|\leq |v_\sigma-e_n||\xi|+|\Pi_\sigma\xi|\leq \eps |\xi|+R_o \leq \eps|\Pi_{\R^d}\xi| +\eps|\xi_n|+R_o \leq \frac12 |\Pi_{\R^d}\xi|+2R
\]
if $\eps<1/2$. This shows that $|\Pi_{\R_d}\xi|\leq 4R$ and thus $R/2<|\xi|<6R$ for $\xi$ as above. It is then clear that $T_\Sigma F=T_\Sigma S_R F$, with $S_R$ a smooth frequency projection onto the annulus $\{R/10<|\xi|<10R \}$ whose symbol is identically one on  $\{R/6\leq |\xi|\leq 6R\}$. Because of \eqref{eq:proj} the lemma follows by a suitable application of Theorem~\ref{thm:main}.}
\end{proof}
%%%%%%%%%%%%%%%%%%%%%%%%%%%%%% PROOF PROOF PROOF

	%%%%%%%%%%%%%%%%%%%%%%%%%%%%%%% PROOF PROOF PROOF
	\begin{proof}[Proof of Theorem~\ref{thm:carsjo}] Let us fix a H\"ormander--Mihlin multiplier $  m:\R^d \to \mathbb C$ as in the statement. We will make the qualitative assumption that $  m$ is a smooth function with compact support in $\R^d$, which is always possible by a suitable approximation; this assumption will be removed at the end of the proof. Our goal is to show that the operator
	\[
	\mathrm {CS}[f](x)\coloneqq \sup_{N\in\R^d} |\mathrm {CS}_{N} [f](x)|,\qquad \mathrm {CS}_{N} [f](x)\coloneqq \int_{\R^d}\widehat f (\eta) m (\eta+N )e^{2\pi i \l \eta, x\r}\, \d \eta, \qquad x\in\R^d,
	\]
	maps $L^2(\R^d)$ to $L^{2,\infty}(\R^d)$ and $L^p(\R^d)$ to itself for all $2<p<\infty$. To that end let us also fix $p\in[2,\infty)$ and a function $f\in\mathcal S(\R^d)$ with compact frequency support.	By our qualitative assumptions there exists $R_o>0$ such that 
	\[
	\supp(\widehat f)\subset B^d(R_o),\qquad \supp(  m) \subset B^d(R_o),
	\]
	and thus only values $N\in B^d(2R_o)$ need be considered in the definition of $\mathrm {CS}[f]$. We now fix $\eps>0$ sufficiently small, as coming from the statement of Lemma~\ref{lem:fam}, but also satisfying the smallness condition
\[
	{\eps <  \left[\max\left(1,100 R_o ^d  \|  m\|_{C^{20d+1}},(100R_o)\right)\right]^{-1} .}
\]
Above we have denoted for positive integers $\nu$
\[
\|g\|_{C^\nu}\coloneqq \sum_{0\leq |\alpha|\leq \nu} \|\partial^\alpha g\|_{L^\infty(\R^d)}.
\]

We will use an auxiliary function $\psi$ to lift $f$ from $\R^d$ to $\R^n$. More precisely we let $\psi:\R\to\R$ be a smooth function with compact frequency support $\mathrm{supp}(\widehat \psi)\subset (-1,1)$ and such that $\psi\geq0$, $\psi \geq 1$ on $(-2^{-3},2^{-3})$, and $\widehat \psi\ge0$, and define
	\[
	\psi_R (t)\coloneqq  \int_{\R}  \widehat\psi \left(u-R\right) e^{2\pi i t u }\, \d u,\qquad t\in \R,
	\]
	 with $R>0$ to be chosen momentarily. We readily see that $\mathrm{supp}(\widehat \psi_R)\subset (R-1,R+1)$. We then define the smooth function with compact frequency support
	\[
	F (x)\coloneqq f(y)\psi_R (x_n),\qquad x=(y,x_n)\in\R^n=\R^d\times \R,
	\]
	where we use the conventions $x=(\Pi_{\R^d}x,x_n)\eqqcolon (y,x_n)\in \R^n$ and $\xi=(\Pi_{\R^d}\xi,\xi_n)\eqqcolon(\eta,\xi_n)\in\R^n$ in order to simplify our formulas.

	For any $N\in B^d(2R_o)$ we apply Lemma \ref{l:N} with parameters $\eps,R$ such that $2R_o < c_o \eps R<R$. We then know that there exists $\sigma=\sigma(N)\in\Sigma$ such that $N(\sigma)=N$. We fix such a sufficiently large value of $R$ and note that 
  \[
   \operatorname{supp}(\widehat \psi_R )\subset \{t:\, R/2 <t<2R\}.
  \]
 Now set
	\[
	\mathcal E(\xi,\sigma)\coloneqq \left[\l v_\sigma,e_n\r -1\right]\l\xi,u_\sigma\r u_\sigma +\left(\frac{\xi_n}{R}-1\right)N(\sigma), \qquad (\xi,\sigma)\in\R^n\times \Sigma,
	\]
	and immediately observe that for $\xi\in\mathrm{supp}(\widehat{F})$ we have, taking into account the facts in Lemma~\ref{l:N},
	\begin{equation}\label{eq:error}
	\begin{split}
|	\mathcal E(\xi,\sigma)| &\lesssim \dist(\sigma,\R^d)^2 R_o+ \dist(\sigma,\R^d) ,
\\
	|\partial_{\xi_j}   \mathcal E(\xi,\sigma)| &\lesssim \dist(\sigma,\R^d)^2 ,\qquad 1\leq j \leq d,
	\\
	|\partial_{\xi_n}  \mathcal E(\xi,\sigma)| &\lesssim  \dist(\sigma,\R^d),
	\\
	|\nabla ^\alpha   \mathcal E(\xi,\sigma)| &=0 \quad\text{if}\quad |\alpha|>1.
\end{split}
	\end{equation}
	Defining $TF(x,\sigma)$ as in the statement of Lemma~\ref{lem:fam}, with the choice of $\sigma\in\Sigma$ as above, we calculate 
	\[
	\begin{split}
	  TF(x,\sigma) &=\int_{\R} \int_{\R^d} m\left(\eta+N(\sigma)+\mathcal E(\xi,\sigma)\right)\widehat f(\eta)\widehat{\psi}_R(\xi_n)e^{2\pi i \l \eta , y\r} e^{2\pi i x_n \xi_n}\, \d \eta \, \d \xi_n
	\\
	& =\mathrm{CS}_N [f](y) \psi_R (x_n)+\int_{\R}\int_{\R^d}	\widehat f(\eta) \widehat{\psi}_R(\xi_n)\Delta(\xi,\sigma)e^{2\pi i \l\eta , y\r} e^{2\pi i x_n \xi_n}\, \d \eta \, \d \xi_n,
	\end{split}
	\]
	where 
	\[
	\Delta(\xi,\sigma)\coloneqq m(\eta+N(\sigma)+\mathcal E(\xi,\sigma))- m(\eta+N(\sigma)),\qquad (\xi,\sigma)\in \R^n \times \Sigma.
	\]
	Note that $\mathrm{supp}( \widehat {F} )\subseteq B^d(R_o)\times\{R-1<|\xi_n|<R+1\}$ which we want to keep memory of. For that reason we let $\lambda\in\mathcal S(\R^n)$ be identically one on $\mathrm{supp}( \widehat {F} )$ and vanish outside $ B^d(10R_o)\times\{R-10<|\xi_n|<R+10\}$, with the natural derivative bounds.
	Define the operator
	\[
	\mathrm{Err}_{\sigma}F(x)\coloneqq \int_{\R^n}	\widehat F(\xi)  \lambda(\xi)\Delta(\xi,\sigma)e^{2\pi i \l\xi , x\r}  \, \d \xi,\qquad x \in\R^n.
	\]
	We then have
	\[
	|\mathrm{CS}[f](y)\psi_R(x_n)|\leq \sup_{\sigma\in\Sigma}  |T F(x,\sigma)|+\sup_{\sigma\in\Sigma} |\mathrm{Err}_{\sigma}F(x)|
	\]
and the proof of the theorem will be complete once we control the error term $\mathrm{Err}_{\sigma}F $ in $L^p(\R^n)$.	We will do so by proving
\begin{equation}\label{eq:max}
\sup_{\sigma\in\Sigma} |\mathrm{Err}_{\sigma}F(x)| \lesssim \M F(x),\qquad x\in \R^n;
\end{equation}
here $\M$ denotes the $n$-dimensional Hardy--Littlewood maximal function. To that end let us define
\[
E(x,\sigma)\coloneqq \int_\R \int_{\R^d}\lambda(\xi)\Delta(\xi,\sigma)e^{2\pi i \l\eta,y\r}e^{2\pi i x_n\xi_n}\,\d \eta\, \d \xi_n,\qquad x\in \R^n,
\]
and note that \eqref{eq:error} remains valid for $\xi\in\mathrm{supp}(\lambda)$. Estimates \eqref{eq:error} combined with a calculation involving the chain rule, {the mean value theorem}, and the derivatives of $m$,  $\mathcal E(\xi,\sigma)$ and $\lambda$, imply 
\[
\begin{split}
   |\mathrm{supp}(\lambda(\cdot) \Delta(\cdot,\sigma))|&\lesssim R_o ^d\quad (\text{using that } \operatorname{supp}m\subset B^d(R_o)),
   \\
	|\nabla  ^\gamma (\lambda(\xi)\Delta(\xi,\sigma))| &\lesssim  \|m\|_{C^{|\gamma|+1}}\left[\dist(\sigma,\R^d)^2R_o+\dist(\sigma,\R^d)\right].
	\\
\end{split}
\]
Using these bounds and integrating by parts inside the integral defining $E(x,\sigma)$ we gather
\[
|E(x,\sigma)|\lesssim \frac{\|m\|_{{C^{20d+1}}}R_o ^d  \left[\dist(\sigma,\R^d)R_o+1\right]\dist(\sigma,\R^d) }{(1+|y|)^{10d} (1+|x_n|)^{10d}},\qquad x=(y,x_n)\in\R^d \times \R.
\]
Remembering our choice of $\eps>0$ and that $\dist(\sigma,\R^d)\lesssim \eps$, this proves \eqref{eq:max} with implicit constant depending only upon dimension. Thus the error term is under control
\[
	\|\sup_{\sigma\in\Sigma}|\mathrm{Err}_{\sigma}F|\|_{L^p(\R^n)}\lesssim \| F\|_{L^p(\R^n)}\eqsim\|f\|_{L^p(\R^d)}.
\]
We have proved the estimates
	\[
	\|\mathrm{CS}[f]\|_{L^{2,\infty}(\R^d)}\lesssim  \|f\|_{L^2(\R^d)},\qquad \|\mathrm{CS}[f]\|_{L^{p}(\R^d)}\lesssim  \|f\|_{L^p(\R^d)},\qquad 2<p<\infty,
	\]
	whenever  $f\in\mathcal S(\R^d)$ has compact Fourier support and $ m$ is a ${C^{20d+1}(\R^d)}$-function with compact support. One easily removes these assumptions on $m$ by approximating with a sequence $  m_k$ which is smooth and has compact Fourier support, and converges pointwise to $m$ a.e. in $\R^d$, and satisfies $\| m_k\|_{\mathcal M_A (d)}\lesssim \| m\|_{\mathcal M_A (d)}$ uniformly in $k$. Then dominated convergence shows that for any measurable map $\R^d\ni y\mapsto N(y)$ we have
	\[
	\int_{\R^d}  m(\xi+N(y))\widehat f(\xi)e^{2\pi i \l y, \eta\r}\,\d \eta =\lim_k \int_{\R^d}  m_k(\xi+N(y))\widehat f(\xi)e^{2\pi i \l y, \eta\r}\,\d \eta
	\]
	whenever $f\in\mathcal S(\R^d)$ has compact Fourier support.	Taking absolute values and applying Fatou's lemma inside the $L^p(\R^d)$-norm yields the conclusion of the theorem for $f\in\mathcal S(\R^d)$ with compact Fourier support. As such functions are dense in $L^p(\R^d)$ the proof is complete.
	\end{proof}
	 %%%%%%%%%%%%%%%%%%%%%%%%%%%%%% SECTION SECTION SECTION
\section{Tiles, adapted functions, and model operator} \label{sec:tiles+model}
This and the following sections are dedicated to the proof of Proposition~\ref{p:mainpf} which, by the reductions of \S\ref{sec:redux}, implies the main theorem, Theorem~\ref{thm:main}. In the remaining of the paper we use the $C^1$ family of rotations $\{O_\sigma:\, \sigma\in\Gr(d,n)\}$ of Remark~\ref{rmrk:Osmooth}, underlying the definition of the operator $U^\star _{\Sigma,{\bf{m}}}$ in \eqref{e:thisissmooth}.

%%%%%%%%%%%%%%%%%%%%%%%%%%%%%% SECTION SECTION SECTION
\subsection{Preliminaries} The family ${\bf{m}}=\{m_\sigma\in L^\infty(\R^d):\, \sigma\in\Gr(d,n)\}$ satisfying \eqref{e:admin4} is fixed for the rest of the paper. Recall, cf.\ \eqref{e:thisissmooth}, that the definition of  $
U^\star_{\Sigma,{\bf{m}}}
$
involves the singular integrals
\begin{equation}
\label{eq:Tsigma}
T_{m_\sigma}f(x,\sigma)\coloneqq  \int_{\R^n}  m_\sigma (O_\sigma \Pi_\sigma \xi) \widehat f(\xi) e^{2\pi i \langle x, \xi\rangle}\, \d \xi,\qquad x\in \R^n.
\end{equation}
The subscript $m_\sigma$ in \eqref{eq:Tsigma} will be omitted from now on.
Let ${\upalpha}\in 3^{-\mathbb N}$ and denote by
\[
 \Sigma_{\alpha}\coloneqq \{\sigma \in \Gr(d,n):\, |v_\sigma-e_n|=\dist(\sigma, \R^{d}\otimes \{0\}) < {\upalpha}\},
 \]
a small neighborhood of $e_n^{\perp}$.  The small  constant ${\upalpha}$ will depend on dimension only and its choice will be explained below. With such a choice, the parameter $\upalpha$ will be fixed along the paper, so we will drop the subindex $\upalpha$ and we will write $\Sigma$ from now on. Consider the maximal operator
\begin{equation}\label{eq:Tstar}
f\mapsto \sup_{\sigma\in \Sigma} |T [P_0 f](\cdot, \sigma)|
\end{equation}
with $Tf(\cdot, \sigma)$ as defined in \eqref{eq:Tsigma} and $P_0$ given by \eqref{eq:Pann}. Note that  the operator $U_{\Gr(d,n),{\bf{m}}} ^\star$ of Proposition \ref{p:redux} is controlled by $O_{\upalpha} (1)$ rotated copies of \eqref{eq:Tstar}. In order to prove Proposition \ref{p:mainpf}, it thus suffices to obtain the corresponding estimate for \eqref{eq:Tstar}. Restricting the choice of variable subspace to $\Sigma$ allows us to precompose $T\circ P_0$ in  \eqref{eq:Tstar} with a smooth restriction to a small frequency cone about $e_n$ as follows. 
Let ${ \Psi_{\mathrm{cn}}} \in\mathcal S(\R^n)$ satisfying
\begin{equation}{\label{eq:Pcone}}
\begin{split} 
{\Psi_{\mathrm{cn}} }\equiv 1 \; \textrm{on}\;\,   \Gamma_0 \cap \mathrm{Ann}\left(\textstyle 1,\frac32\right), \qquad 
&\Gamma_{0}\coloneqq\{\xi\in\R^n\setminus\{0\}: \,  |\xi' -e_n | <  3^4 {\upalpha} \}, 
\\ 
\supp {\Psi_{\mathrm {cn}}}\subset \Gamma_1 \cap   \mathrm{Ann}\left(\textstyle\frac12,2\right), \qquad  &\Gamma_1\coloneqq\{\xi\in\R^n\setminus\{0\}: \,  |\xi'-e_n|<  3^{5}{\upalpha}\},
\end{split}
\end{equation}
where $\xi'\coloneqq \xi/|\xi|$, and let ${P}_{\mathrm {cn}} f $ be the corresponding  frequency cutoff ${P}_{\mathrm {cn}} f\coloneqq ({\Psi_{\mathrm cn}} \widehat f)^\vee$. Then
\begin{equation}\label{eq:nonsing}
 \sup_{\sigma\in\Sigma}|T[(\mathrm{Id}-  P _{\mathrm{cn}}) \circ P_0 f](\cdot, \sigma)|\lesssim \M f
\end{equation}
where $\M$ denotes the Hardy--Littlewood maximal operator; indeed, the singularities of the symbols of the operators $f\mapsto Tf(\cdot,\sigma)$  for $\sigma\in\Sigma$ are all contained in $\Gamma_0$, namely away from the frequency support of $(\mathrm{Id}- P _{\mathrm{cn}} )\circ  P_0 f$. This reduces  the proof of  Proposition \ref{p:mainpf} to the proof of the corresponding  bound  for the maximal operator
\begin{equation}
\label{eq:realTstar}
U^\star f\coloneqq \sup_{\sigma \in \Sigma} \left|T [   P_0 \circ  P _{\mathrm{cn}}  f](\cdot, \sigma) \right|.
\end{equation}
The multiplier operators $g\mapsto T [   P_0	 \circ  P _{\mathrm{cn}}g](\cdot, \sigma)$ have Fourier support in the truncated cone 
\begin{equation}\label{eq:truncone}
\Delta\coloneqq \Gamma_1 \cap \mathrm{Ann}\left(1,\frac32\right).
\end{equation}
 In our analysis,  in accordance with the uncertainty principle and the Fourier support of the restriction to $\Delta$, the relevant spatial scales for each $\sigma\in\Sigma$ will then be
\begin{equation}
\label{e:spatialscales} s\in \mathbb{S}\coloneqq\left\{s\in 3^{\mathbb Z}, \, 3^{6} s \cdot {\upalpha} \geq 1\right\}
\end{equation} 
in the $\sigma$ variables, while the spatial scale will always be $\sim 1$ in the coordinate $\sigma^\perp$. Note that by taking $
\upalpha$ sufficiently small, depending on dimension only, we can and will always assume that $s\geq 3^{10d}$ for all relevant scales $s\in \mathbb S$.
The rest of the section is devoted to construction of a  model sum representation for $U^\star $.

%%%%%%%%%%%%%%%%%%%%%%%%%%%%%% SECTION SECTION SECTION
\subsection{Grids and tiles} \label{sec:dyadic}   One of the ways in which the simultaneous localization to  $\sigma\in \Sigma$  and to the homogeneous Fourier region $\Delta'=\{\xi':\xi \in \Delta\}\subset \mathbb S^d$ is exploited below is to conflate either region with a copy of $\R^d$ as follows. Let
$\widetilde {\Delta'}\subset \mathbb S^d$ be the $3^{6}\sqrt{d}$-neighborhood of $\Delta'$. Then,  {provided the constant ${\upalpha}$ is chosen sufficiently small depending on dimension, there is an approximate isometry 
\[
\xi \in \widetilde {\Delta'} \mapsto \Pi_{e_n^\perp}  \xi \in  e_n^\perp \equiv \R^d
\] 
in the sense that
\begin{equation}\label{e:ncomp}
\b{\frac{\theta(\xi,\eta)}{\left|\Pi_{e_n^\perp} \xi - \Pi_{e_n^\perp} \eta\right|}  },\, \frac{|\xi-\eta|}{\left|\Pi_{e_n^\perp} \xi - \Pi_{e_n^\perp} \eta\right|}  \in \left[\textstyle \frac13,3\right], \qquad \xi,\eta  \in \widetilde {\Delta'},\,\,\, \xi\neq \eta,
\end{equation}
where we remember that $\theta(\xi,\eta)$ denotes the convex angle between $\xi,\eta\in\SS^d$. This together with our previous restriction in the remark following \eqref{e:spatialscales} fixes the choice of ${\upalpha}$ to be a small dimensional constant.}
When $\sigma \in \Sigma$ we have that $v_\sigma \in \Delta'$ and the approximate isometry extends to $\Sigma$ as well. This allows us to construct a grid structure localizing the $\Sigma$ and $\Delta'$ components by pulling back the corresponding Euclidean structure. Rotated Euclidean grids will also localize the spatial component: relevant definitions follow.

 A standard triadic grid $\mathcal G$ on $\R^d$ of pace $(K,m)$, where $K\geq 1$ is an integer and $m\in [0 ,K)$ is \emph{not necessarily an integer},  is a collection of cubes $Q\subset \R^d$  satisfying the properties \vspace{.3em}
 \begin{itemize}  \setlength\itemsep{.5em}
 \item[g1.] (quantized length) $\mathcal G =\bigcup_{j\in \mathbb Z} \mathcal G_{Kj+m}$, where $\mathcal G_{u}=\{Q\in \mathcal G: \,\ell(Q) =3^{u}\}$,
 \item[g2.] (partition) $\displaystyle \mathbb \R^d=\bigcup_{Q\in \mathcal G_{Kj+m} } Q$ for all $j\in \mathbb Z$,
 \item[g3.] (grid) $L\cap Q \in\{L,Q,\varnothing\}$ for every $L,Q\in\mathcal G$.
  \end{itemize}
In particular, $\mathcal G$ is partially ordered by inclusion and we may define for each $Q\in \mathcal G$ 
\[
Q^{(1)}\coloneqq \textrm{minimal } L\in \mathcal G  \textrm{ with } Q\subsetneq L, 
\]
and inductively $Q^{(j)}\coloneqq [Q^{(j-1)}]^{(1)}$ for all $j> 1$ to be the $j$-th parent of $Q$ in the grid $\mathcal G$. With these definitions, properties g2.\ and g3.\ yield for any positive integer $\kappa$ the partition
\[
L=\bigsqcup_{Q\in \mathrm{ch}_\kappa (L)} Q, \qquad \mathrm{ch}_\kappa (L)\coloneqq\{Q\in \mathcal G:\, Q^{(\kappa)}=L\}
\]
with the collection  $\mathrm{ch}_\kappa (L)$ being referred to as the $\kappa$-th generation children of $L\in \mathcal G$. Triadic grids have been chosen because of the following convenient property. For each $L\in \mathcal G$ there is a unique $Q=L^{\circ,\kappa}\in \mathrm{ch}_\kappa (L)$ which is concentric with $L$, which is referred to as the $\kappa$-\emph{center} of $L$.
The next   definition highlights   a different role for certain other elements of $\mathrm{ch}_\kappa (L)$.  Define
\[
\begin{split} &
\mathrm{ch}_{\kappa, \square} (L) \coloneqq\{Q\in \mathrm{ch}_{\kappa} (L): \, \mathrm{dist}(Q,L^{\circ,\kappa}) \geq 3^{3-\kappa} \ell(L)\}
\end{split}
\]  
referring to this collection as the \emph{peripheral} $\kappa$-th children of $L$. In the following lemma and hereafter we fix a large value of ${\kappa\coloneqq 12+\log_3 d }$.

%%%%%%%%%%%%%%%%%%%%%%%%%%%%%% LEMMA LEMMA LEMMA
\begin{lemma} \label{l:dyadic}  Let ${\kappa\coloneqq 12+\log_3 d }$. There exist standard triadic grids $\mathcal G_1,\ldots, \mathcal G_{C}$  of pace $(1,m_j)$ on $\R^{d}\equiv e_n^{\perp}$, $1
\leq j \leq C$,  $C=C_{\kappa,d}$, with the following property: For each $\beta\in \Delta'$, $s\in\mathbb S$, there exists $j=j(\beta,s)\in \{1,\ldots, C\}$ and a cube $L=L(\beta,s)\in \mathcal{G}_j$ such that
\begin{itemize} 
\item[\emph{1.}] $3^{3}s^{-1} \leq \ell(L) \leq 3^4s^{-1}$,\vspace{.4em}
\item[\emph{2.}] $ B\left(\beta, 3^{-\kappa} s^{-1}\right) \subset \Pi_{e_n^\perp}^{-1}(L^{\circ,\kappa})$,
\item[\emph{3.}] $ \mathrm{Ann}\left(\beta, \textstyle 3^{-1}s^{-1}  ,3^{2}s^{-1}\right)  \subset \bigcup\left\{ \Pi_{e_n^\perp}^{-1}(Q):\, Q \in \mathrm{ch}_{\kappa, \square} (L)\right\}$.
\end{itemize}
\end{lemma}
%%%%%%%%%%%%%%%%%%%%%%%%%%%%%% LEMMA LEMMA LEMMA

%%%%%%%%%%%%%%%%%%%%%%%%%%%%%% REMARK REMARK REMARK
\begin{remark} \label{r:bijgrid} With reference to the previous lemma,  say that the pair $(\beta,s)\in \Delta'\times \mathbb S$ is of type $j\in\{1,\ldots,C\}$ if $j(\beta,s)=j$.  The restriction $s\in\mathbb S$, see \eqref{e:spatialscales},  and properties 1.\ and 2.\ ensure the inclusion 
\[
L(\beta,s)\subset B\left(\Pi_{e_n^\perp} \beta, 3^{5} \sqrt{d}{s}^{-1}	\right) \subset \Pi_{e_n^\perp}  \widetilde{\Delta'}.
\]
\end{remark}
%%%%%%%%%%%%%%%%%%%%%%%%%%%%%% REMARK REMARK REMARK

%%%%%%%%%%%%%%%%%%%%%%%%%%%%%% PROOF PROOF PROOF
\begin{proof}[Proof of Lemma \ref{l:dyadic}] For this proof we denote by $Q(\xi,r)$ the cube with center $\xi\in\R^d$ and sidelength $2r$ and by $\|\cdot\|_{\infty}$ the $\ell^\infty$-norm on $\R^d$. First of all, using standard grid techniques as in \cite{LeNa} we may find triadic grids $\mathcal G_1,\ldots, \mathcal G_{C}$ with the property that for each cube $P\subset \R^d$ there is $j\in\{1,\ldots, C\}$ and $L(P)\in \mathcal G_j$ with $P\subset L(P)\subset (1+3^{-(\kappa+9)}) P$. As a particular consequence 
\[
\begin{split}
|c(P)- c(L(P))| &\leq \sqrt{d} 3^{-(\kappa+9)} \ell(P) \leq  \sqrt{d} 3^{-(\kappa+9)}\ell(L(P)),
\\
\|c(P)- c(L(P))\|_\infty& \leq3^{-(\kappa+9)} \ell(P)\leq  3^{-(\kappa+9)}\ell(L(P)).
\end{split}
\]
 We have that $\Pi_{e_n ^\perp} B(\beta,3^2s^{-1}) \subset  B(\Pi_{e_n ^\perp} \beta, 3^3 s^{-1})\subset Q(\Pi_{e_n ^\perp} \beta, 3^3 s^{-1})$. There exists $L=L(\beta,s)\in \mathcal G_j$ for some $j\in\{1,\ldots,C\}$ such that 
\[
Q(\Pi_{e_n ^\perp} \beta, 3^3 s^{-1}) \subseteq L \subseteq (1+3^{-(\kappa+9)})Q(\Pi_{e_n ^\perp} \beta, 3^3 s^{-1}) 
\]
from which 1.\ follows immediately. In order to see 2.\ let $\eta\in B(\beta,3^{-\kappa}s^{-1})$ so that $\Pi_{e_n ^\perp}\eta \in B(\Pi_{e_n ^\perp} \beta, 3^{-\kappa+1} s^{-1})$. Then we have the chain of inequalities
\[
\|\Pi_{e_n ^\perp} \eta-c(L)\|_\infty \leq |\Pi_{e_n ^\perp} \beta-\Pi_{e_n ^\perp}\eta|+\|\Pi_{e_n ^\perp} \beta-c(L)\|_\infty\leq 3^{-\kappa+1} s^{-1} +3^{-(\kappa+9)}\ell(L)\leq 3^{-\kappa-1}\ell(L).
\]
This shows that $\Pi_{e_n ^\perp } \eta \in Q(c(L),\, 3^{-\kappa-1}\ell(L))=3^{-1} L^{\circ,\kappa}$ and so $\eta\in \Pi_{e_n ^\perp} ^{-1}(L^{\circ,\kappa})$ as desired.

In order to complete the proof it suffices to show that if $\eta\in B(\beta,3^2s^{-1})\setminus B(\beta,3^{-1}s^{-1})$ then $\Pi_{e_n ^\perp}\eta \in Q$ for some $Q\in \mathrm{ch}_{\kappa, \square} (L)$. To that end let $\eta\notin B(\beta,3^{-1}s^{-1})$; we will have that $\Pi_{e_n ^\perp} \eta \notin B(\Pi_{e_n ^\perp}\beta,3^{-2}s^{-1})$ and so
\[
\begin{split}
\|\Pi_{e_n ^\perp} \eta -c(L)\|_\infty&\geq \frac{1}{\sqrt d}|\Pi_{e_n ^\perp} \eta-\Pi_{e_n ^\perp} \beta| -\|\Pi_{e_n ^\perp} \beta -c(L)\|_\infty \\ &\geq {\frac{3^{-2}s^{-1}\sqrt{d}}{d}-3^{-(\kappa+9)}3^4s^{-1}}  \geq \sqrt d 3^{5-\kappa}\ell(L).\end{split}
\]
Since $\Pi_{e_n ^\perp}\eta \in B(\Pi_{e_n ^\perp}\beta,3^3 s^{-1})\subseteq L$ there exists $Q\in \mathrm{ch}_{\kappa, \square} (L)$ such that $\Pi_{e_n ^\perp}\eta\in Q$. Then for all $y_1\in Q$ and $y_2\in L^{\circ,\kappa}$ we calculate
\[\begin{split}
\|y_1 -y_2\|_\infty&\geq \sqrt{d}3^{5-\kappa}\ell(L)-\|y_1-\Pi_{e_n ^\perp}\eta\|_\infty-\|y_2-c(L)\|_\infty \geq \sqrt{d} 3^{5-\kappa}\ell(L)- 23^{-\kappa}\ell(L)\\ &\geq \sqrt{d} 3^{4-\kappa}\ell(L)
\end{split}
\]
which in turn implies that $\dist(Q,L^{\circ,\kappa})\geq 3^{4-\kappa}\ell(L)$. We have showed that  $Q\in \mathrm{ch}_{\kappa, \square} (L)$ so the proof is complete. 
\end{proof}
%%%%%%%%%%%%%%%%%%%%%%%%%%%%%% PROOF PROOF PROOF
More generally,  if $X$ is any set, any collection  $\mathcal G\subset \mathcal P(X)$ satisfying the grid property g3.\ is referred to as a \emph{grid}. The forthcoming definitions  construct suitable rotated grid systems  in $\R^{n}$ whose purpose is to provide spatial localization. 

%%%%%%%%%%%%%%%%%%%%%%%%%%%%%% SECTION SECTION SECTION
\subsubsection{Spatial grids} Fix a standard triadic grid $\mathcal L$ of pace $(1,0)$ in $\R^d$ and define the subsets of $\R^n$
\[
\mathcal R \coloneqq\left\{L\times [j,j+1)):\; L\in\mathcal L,\, \ell(L)\geq 1, j\in\mathbb Z\right \}.
\]
The collection  $\mathcal R =\mathcal R_
{e_n}$ of \emph{plates parallel to $e_n^\perp\in \Gr(d,n)$} obviously inherits property g3.\ from $\mathcal L$ and the unit length partition of $\R$.
If $\mathcal R \ni R = L\times J$, call $L$ the \emph{horizontal component} and $J$ the \emph{vertical component} of $R$. The condition $\ell(L)\geq 1$ in the definition of $\mathcal R$ guarantees that the elements $L\in\mathcal R$ have horizontal components of larger scale than their vertical components, so they can be in general understood as plates in $\R^{d+1}$ of \emph{scale} $\scl(R) \coloneqq \ell(L)=  3^{k} $, $k\in \N$,  and thickness $1$  along the direction $e_n$.
For $\beta\in \mathbb S^d$  the rotated grid $\mathcal R _{\beta}$ of \emph{plates parallel to $\beta^\perp\in \Gr(d,n)$}  is then defined by
\[
\mathcal R_{\beta}  \coloneqq \left\{O_{\beta^\perp,e_n ^\perp} Q:\, Q\in\mathcal R \right\}.
\]
It is convenient to associate to an element $R\in \mathcal R_{\beta}$   the translation and    anisotropic rescaling operator
\[
\mathrm{Sy}_{R}^{p} f (x) \coloneqq \frac{1}{|R|^{\frac{1}{p}}} f\left( \frac{\Pi_{\beta^{\perp}}(x-c(R))}{\b{|\Pi_{\beta^\perp}R|}}+
\frac{\Pi_{\beta}(x-c(R))}{\b{|\Pi_\beta(R)|}}   \right), \qquad x\in \R^n. 
\]
The sense of this definition is that if  $f$ is a bump function around the origin of $\R^n$ then $\mathrm{Sy}_{R}^{p} f$ is an $L^p$-normalized bump adapted to the  plate $R\in \mathcal R_{\beta}$.

%%%%%%%%%%%%%%%%%%%%%%%%%%%%%% SECTION SECTION SECTION
\subsubsection{Tiles}\label{sec:tiles} Let $\mathcal G$ be any standard triadic grid on $\mathbb R^d$. Define the \emph{admissible tiles}  $t\in \mathcal T_{\mathcal G}$   generated by $\mathcal G$ as the collection of cartesian products
$
t=R_t \times Q_t
$
where
\vspace{.3em}
 \begin{itemize}  \setlength\itemsep{.2em}
 \item[1.] $Q_t \in \mathcal G$, $Q_t \subset  \Pi_{e_n^\perp} \widetilde {\Delta'}$;
 \item[2.]  $R_t=L_t \times I_t\in \mathcal R_{v_t}$, where  {$v_t$ is any fixed choice of vector in the set $\Pi^{-1}_{{e_n ^\perp}} Q_t ^{\circ,\kappa} \subset \mathbb S^d$; we will make that choice specific in \eqref{eq:tile} below;}
 \item[3.] $\scl(R_t) \ell(Q_t )\in [1,3)$.
 \end{itemize} The \emph{scale} of $t\in \mathcal T_\mathcal G$ is $\scl(t)\coloneqq \scl (R_t)$. 
  The \emph{frequency support} $\omega_t\subset \R^n$ of a tile $t$ is given by
\[
\omega_t\coloneqq \left\{\xi \in \mathrm{Ann}\textstyle(\frac12,2): \, \Pi_{e_n^\perp}\xi' \in Q_{t}^{\circ,\kappa}\right\}\subset \R^n.
\]
Namely, $\omega_t$ consists of those points of  $\mathrm{Ann}\textstyle(\frac12,2)$ whose projection on the sphere is contained in the preimage of the $\kappa$-th center of $Q_t$, denoted now by $Q_{t}^{\circ,\kappa}$; see Figure~\ref{fig:tiles}. Note that the collection
$
\Omega_{\mathcal G}\coloneqq\{\omega_t: \, t\in \mathcal T_{\mathcal G}\}
$
is a grid on $\R^n$, as it inherits property g3.\ from $\mathcal G$. 

\begin{figure}[!ht]
\centering
\includegraphics[scale=0.5]{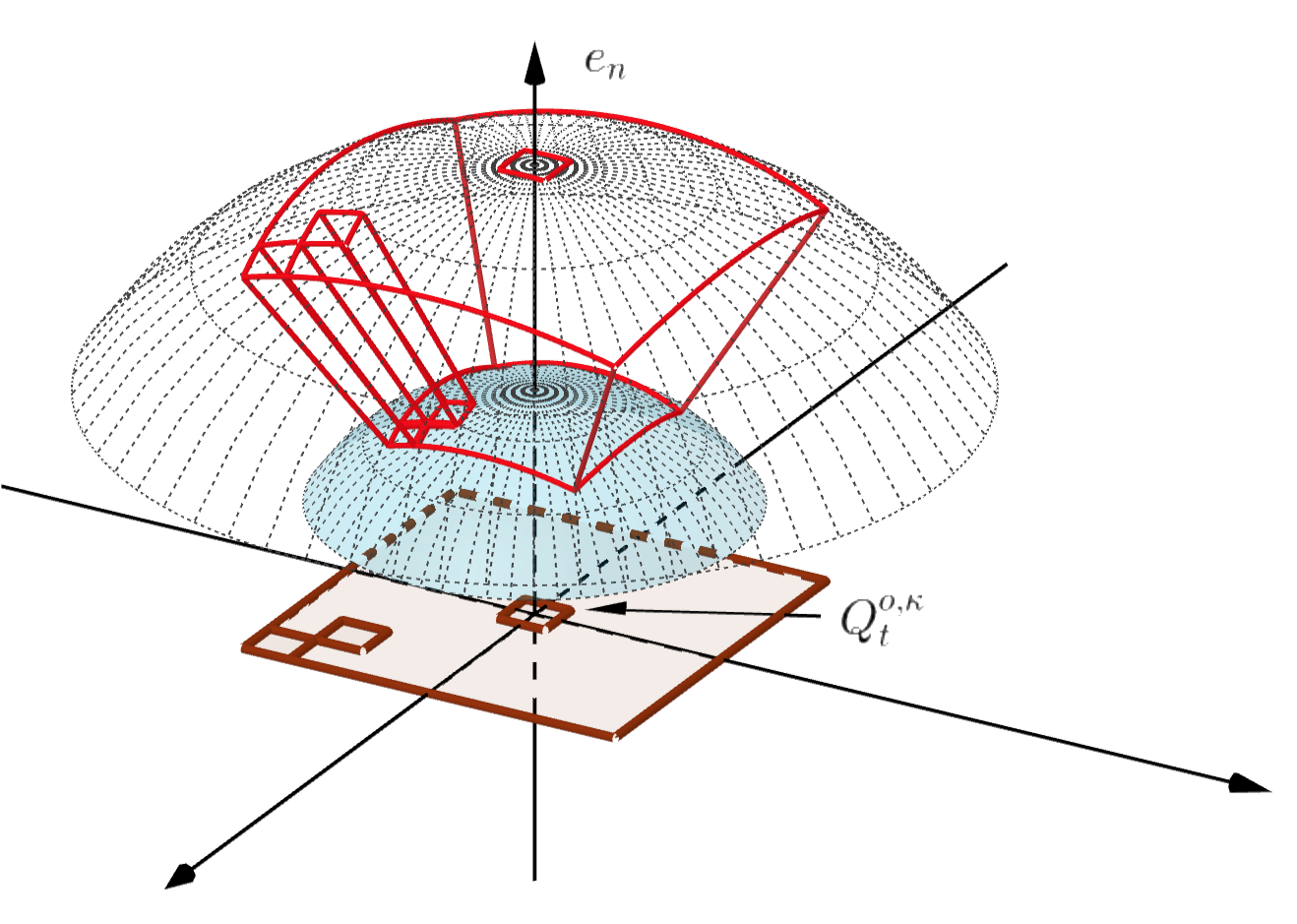}
\label{fig:bolas5}
\caption{The frequency component of a tile and its $\kappa$-children}
\label{fig:tiles}
\end{figure}

Let $Q_0=[0,1)^d$ be considered as a triadic cube and let us fix an arbitrary enumeration of the elements of $\mathrm{ch}_{\kappa, \square}(Q_0)$
\[
\mathrm{ch}_{\kappa, \square}(Q_0)=\bigsqcup_{\tau=1} ^{N} Q_{0,\tau},\qquad N=c_d 3^{\kappa d},
\]
where $c_d$ is a dimensional constant. Note that this indexing of the peripheral $\kappa$-children of $Q_0$ induces an indexing of the peripheral $\kappa$-children in $\mathrm{ch}_{\kappa, \square}(Q_t)$ for any tile $t$, in a way that is independent of the particular choice of tile $t$ and grid $\mathcal G$. The   \emph{directional support}  $\alpha_t \subset \mathbb S^d$ of a tile $t$ is given by
\begin{equation}\label{eq:alfat}
\begin{split}&
\alpha_t\coloneqq \bigsqcup_{\tau=1,\ldots, N} \alpha_{t,\tau}, \\ &\alpha_{t,\tau}\coloneqq \Pi^{-1}_{e_n^\perp}( Q_{t;\tau}), \qquad Q_{t;\tau}\in \mathrm{ch}_{\kappa, \square}(Q_t),\qquad \tau\in\{1,\ldots,N\}.
\end{split}
\end{equation}
The collections 
$
A_{\mathcal G,\tau}\coloneqq\{\alpha_{t,\tau}: \, t\in \mathcal T_{\mathcal G}\}
$
are  grids on $\mathbb S^d$, for each $1\leq \tau \leq N$. 

%%%%%%%%%%%%%%%%%%%%%%%%%%%%%% SECTION SECTION SECTION
\subsubsection{Adapted classes and model sums}\label{sec:classes} Hereafter, let $\Theta_M$ stand for the unit ball of the Banach space of functions
\[
u\in \mathcal C^{M}(\R^n), \qquad \|u\|_{\star,M}\coloneqq \sup_{0\leq |\alpha |\leq M} \left\| \left\langle x \right\rangle^{M} \partial_x ^\alpha u(x)  \right\|_\infty<\infty,
\]
where $\l x\r\coloneqq (1+|x|^2)^\frac12$. We will often use the special bump function $\chi_M\in \mathcal C^\infty$ given by 
\[
\chi_M(x)\coloneqq \l x\r^{-M}.
\]

Given a tile $t=R_t\times Q_t$ and a large positive integer $M\gg 1$ we define $\mathcal F_t^M$ as the collection of complex-valued functions $\varphi$ on $\R^n$ with the properties ($\varphi$ are $L^2$-normalized bump functions adapted to $R_t$)
\begin{itemize}
\item[(i)] $\varphi\in\left\{\mathrm{Mod}_{v_t} \mathrm{Sy}_{R_t}^{2} \Phi:\Phi \in \Theta_M\right\}, \qquad \mathrm{Mod}_v g(x)\coloneqq e^{2\pi i \l x, v\r} g(x),\qquad x\in \R^n, \quad v\in \SS^d$; 
\item[(ii)]  $\mathrm{supp}\,\widehat {\varphi}\subseteq \omega_t$.
\end{itemize}
 Further, define  $\mathcal  A_t ^{M}$ as the  collection of complex-valued  functions ${\vartheta}={\vartheta}(x,\sigma)$ on  $\R^{n} \times \Sigma$  satisfying  
\begin{itemize}
\item[(iii)] ${\vartheta}(\cdot,\sigma)\in \mathcal F_t^M$ for all fixed $\sigma\in \Sigma$;
\item[(iv)] ${\vartheta}(\cdot,\sigma)=0$ for all $\sigma \in \Sigma$ such that $v_\sigma \notin \alpha_t$;
\item[(v)] $\vartheta$ verifies the  adaptedness condition
\[
|{\vartheta} ( \cdot,\sigma) - {\vartheta} (\cdot,\rho) |\leq  \max\left\{    \scl(t) \dist(\sigma,\rho)  , \frac{1}{\log(\e+ \left[\dist(\sigma,\rho) \right]^{-1})} \right\} \mathrm{Sy}_{R_t}^{2} \chi_M
\]
for all $\sigma,\rho \in \Gr(d,n)$ with $\scl(t) \dist(\sigma ,\rho) \lesssim  1$.
\end{itemize} 
Note that both classes are $L^2$-normalized. 

%%%%%%%%%%%%%%%%%%%%%%%%%%%%%% REMARK REMARK REMARK
\begin{remark} The point of condition (v) is that in a tree, cf. Definition~\ref{def:trees},  with top direction $\sigma$  there are $\lesssim |\log(\dist(\sigma(x),\sigma))|$  frequency scales contributing at each point, and the total contribution of the difference $ |{\vartheta}_t ( x,\sigma(x)) - {\vartheta}_t (x,\sigma) | $ is summable.
\end{remark}
%%%%%%%%%%%%%%%%%%%%%%%%%%%%%% REMARK REMARK REMARK
At this point, define the tile coefficient maps 
\begin{equation}
\label{e:wpt}
F_M[f](t) \coloneqq  \sup_{\varphi \in \mathcal {F}_{t} ^M } \left|\langle f, \varphi\rangle \right|
\end{equation}
and, for a fixed measurable choice $\sigma:\R^n\to \Gr(d,n)$ with the property that $\sigma(x) \in \Sigma$ for all $x\in \R^n$, and $1\leq \tau \leq N$
\begin{equation}
\label{e:wpt2}
A_{\sigma,\tau,M}[g] (t) \coloneqq \sup_{\vartheta \in \mathcal {A}_{t} ^M } \left| \left\langle g, \vartheta\left(\cdot, \sigma(\cdot) \right) \cic{1}_{\alpha_{t,\tau}}\left(v_{\sigma(\cdot)}\right)\right\rangle \right|.
\end{equation}
 The model bisublinear operator
 \begin{equation}
\label{e:wp3}
\Lambda_{\mathbb P;\sigma,\tau,M} (f,g) =\sum_{t\in \mathbb P}F_M[f](t) A_{\sigma,\tau,M}[g] (t) 
\end{equation}
is then associated to each subset $\mathbb P\subset \mathcal T_{\mathcal G}$.  Momentarily, it will be shown that the duality form for the maximal operator \eqref{eq:realTstar} is controlled by a suitable combination of  finitely many forms \eqref{e:wp3}.

%%%%%%%%%%%%%%%%%%%%%%%%%%%%%% SECTION SECTION SECTION
\subsection{A homogeneous partition of unity}\label{sec:partition}   Our goal here is to produce a single scale Gabor decomposition of the frequency region $\Delta$ defined in \eqref{eq:truncone} which contains the Fourier support of the multiplier $ P_0 \circ P_{\mathrm{cn}} $. 

Fix a  (spatial) scale $s\in \mathbb S $, cf. \eqref{e:spatialscales}, and recall that if $\beta \in \Delta'$, this choice guarantees that $B ( \beta, 3^{6}s^{-1})\subset \widetilde {\Delta'}$. We remember that {$\kappa=12+\log_3 d$ is fixed.}  Let $  \mathcal B _{s^{-1}}$ be a $3^{-\kappa}s^{-1}$-net on $\Delta'$, in the sense that  
\begin{itemize}
\item [-] t{}he sets $ \left\{    B ( \beta, 3^{-\kappa} s^{-1} )  :\beta \in \mathcal B _{s^{-1}}   \right\}$ form a finitely overlapping cover of $\Delta'$;
\item [-] the  sets $ \left\{ B ( \beta, 3^{-(\kappa+1)}s^{-1})  :\beta \in \mathcal B _{s^{-1}}\right\} $ are pairwise disjoint.
\end{itemize}
%Here, for $x \in \mathbb R^n$ and $r>0$, $B(x,r)$ denotes the open ball in $\mathbb R^n$ centered at $x$ with radius $r$.  
Now let $\{ \theta_{\beta} \}_{\beta \in \mathcal B _{s^{-1}}}$ be a smooth, $0$-homogeneous partition of unity on $\Gamma_1$, consisting of nonnegative real-valued functions, which is subordinate to $ \{  {B ( \beta, {3^{-\kappa}s^{-1}})}:\beta \in \mathcal B _{s^{-1}} \}$. By this, we mean 
\begin{equation}
 \textrm{supp}\, \theta_{\beta} \subset B ( \beta, {3^{-\kappa}s^{-1}})    \qquad {\textrm{and}}\qquad \sum_{\beta \in \mathcal B _{s^{-1}}} \left[ \theta_{ \beta} (\eta' ) \right]^2 = 1 \qquad \forall \eta \in \Gamma_1,\quad \eta'\coloneqq\frac{\eta}{|\eta|}, \label{eq:partition}
\end{equation}
and  that the $k$-th order {tangential} derivatives of $\theta_{\beta} $ are of the order $O_{k,n,\delta}(s^{k})$ uniformly in $\beta \in \mathcal B _{s^{-1}}$. Note also that \eqref{eq:partition} implies that $ \theta_{\beta}\equiv  1$ on $ B( \beta, {3^{-(\kappa+1)}s^{-1}}) $.
For $\beta \in \mathcal B _{s^{-1}}$, consider the function $\phi_{\beta} \in \mathcal S ( \mathbb R^n )$ whose Fourier transform is given by
\begin{equation}
\label{eq:phihat}
\widehat{{\phi}_{ \beta}} (\xi ) \coloneqq \theta_{\beta} \left( \xi' \right) \zeta (|\xi|), \qquad \xi \in \mathbb R^n \setminus \{ 0 \} .  
\end{equation}
By construction,
\begin{equation}\label{eq:phibetasupp}
\supp \widehat{\phi_{\beta}}\subset  \left\{  \xi \in \mathrm{Ann}\left(\textstyle \frac12 ,2\right):\, \xi'\in B\left(\beta, {3^{-\kappa}s^{-1}}\right)  \right\} .
\end{equation}
and \eqref{eq:partition} then ensures
\[
(P_0\circ  P _{\mathrm{cn}} )	f = \sum_{\beta \in \mathcal B _{s^{-1}}}  P_0\circ  P _{\mathrm{cn}}f \ast \phi_{\beta} \ast \phi_{\beta} .
\]
For each $\beta \in \mathcal B _{s^{-1}}$, consider the lattice $ Z(\beta) \coloneqq O_{\beta^\perp,e_n^{\perp}} \left[    s  \mathbb Z^{d}  \times    \mathbb  Z \right] $.  Using Fourier series on the support of each $\widehat{\phi_{s,\beta}}$ we obtain
\begin{equation}
\label{eq:decomp1}
P_0\circ  P _{\mathrm{cn}}f= \lim_{T\to+\infty}\frac{1}{T^n} \int_{[0,T]^n}s^d \sum_{\beta\in \mathcal B _{s^{-1}}}  \sum_{z\in Z(\beta)}    \langle (P_0\circ P _{\mathrm{cn}})f , \mathrm{Tr}_{z+y} \phi_{\beta} \rangle \mathrm{Tr}_{z+y}\phi_{\beta} \, \d y, 
\end{equation}  for all $f\in\mathcal S(\R^n)$, where $\mathrm{Tr}_u h(x)\coloneqq h(x-u)$ for $x,u\in\R^n$.

%%%%%%%%%%%%%%%%%%%%%%%%%%%%%% SECTION SECTION SECTION
\subsection{Discretization of the kernel}\label{sec:model}  We turn to the discretization of the variable kernels in \eqref{eq:realTstar}. Choose a non-negative, smooth radial function $\Psi$ on $\R^n$ with 
\[
\supp  {\Psi} \subset \mathrm{Ann} \left({\textstyle \frac{8}{9}  , \frac{28}{9}    }\right), \qquad 
\sum_{s\in 3^{\mathbb Z}}   {\Psi}(s\xi) = 1 \qquad \forall \xi \in \R^n\setminus\{ 0\},
\]
while $\Phi$ is a smooth radial cutoff equal to $1$ on $\mathrm{Ann}(1,\frac32)$ and vanishing off $\mathrm{Ann}(\frac12,2)$.
These definitions entail
\begin{equation}\label{e:LPtheory}
\begin{split}
& T[ P_0\circ  P _{\mathrm {cn}} f](\cdot,\sigma)= \sum_{s\in 3^{\mathbb Z}} P_0\circ P _{\mathrm {cn}} f* \psi_{s}( \cdot,\sigma), 
\\
 & \psi_{s}(x,\sigma)\coloneqq \int_{\R^n} m_\sigma(O_\sigma \Pi_\sigma \xi)   {\Psi}(s\pr_\sigma \xi) \Phi(\xi)  e^{2\pi i {\langle x, \xi \rangle}} \, \d \xi,\qquad (x,\sigma)\in \R^n\times \Sigma.
\end{split}
\end{equation}
First, we use condition \eqref{eq:Pcone} on the frequency support of $P _{\mathrm {cn}} f$ to deduce that $  P_{\mathrm {cn}} f* \psi_{s}=0$ unless $s\in \mathbb S$, cf.\ \eqref{e:spatialscales}, and consequently restrict the sum in \eqref{e:LPtheory} to $s\in \mathbb S$.
Subsequently,  use \eqref{eq:decomp1} to estimate
\[
\begin{split}
 |T [ P_0\circ P_{\mathrm {cn}} f])(\cdot,\sigma  )|\leq  \lim_{T\to \infty} \int\displaylimits_{[0,T]^{n}} \bigg|\sum_{s\in \mathbb S} s^d\sum_{\beta\in \mathcal B _{s^{-1}}} \sum_{z\in Z(\beta)}    \langle \mathrm{Tr}_{-y} g ,  \mathrm{Tr}_{z} \phi_{\beta} \rangle  \mathrm{Tr}_{y}( \mathrm{Tr}_z\phi_{\beta}*\psi_{s} (\cdot,\sigma))\bigg| \, \frac{\d y}{T^n},
\end{split}
\]
with ${g\coloneqq P_0\circ P_{\mathrm{cn}} f}$. The norm estimates for $U^\star$ of \eqref{eq:realTstar} will follow from corresponding estimates for the maximal operator
\begin{equation}\label{eq:modwotiles}
f\mapsto \sup_{\sigma\in\Sigma} \bigg|\sum_{{s\in \mathbb S}} s^d\sum_{\beta\in \mathcal B _{s^{-1}}} \sum_{z\in Z(\beta)}    \langle   f ,\mathrm{Tr}_{z} \phi_{\beta} \rangle \mathrm{Tr}_{z}\phi_{\beta}*\psi_s (\cdot,\sigma)\bigg|.
\end{equation}
To each triplet  $(\beta,s,z)$ appearing in the sum \eqref{eq:modwotiles}, assume $(\beta,s)$ is of type $j$ in accordance to Remark \ref{r:bijgrid} and associate a tile 
\begin{equation}
\label{eq:tile}
t=R_t\times Q_t=t(\beta,s,z)\in {\mathcal T_{\mathcal G_j}, }
\end{equation}
where {$\mathcal G_j$} is one of the grids appearing in the conclusion of Lemma \ref{l:dyadic}, as follows. Firstly, $Q_t=L(\beta,s)$ {and $v_t\coloneqq \beta \in \Pi^{-1}_{{e_n ^\perp}} Q_t ^{\circ,\kappa} \subset \mathbb S^d$}. Secondly, $R_t$ is the unique element of $\mathcal R_\beta$ with $z\in R_t $ and $\scl(R_t)\ell(L(\beta,s))\in \b{[1,3)}$. Of course, $\scl(t)\sim s$. 
Recall the definitions of the collections  $\mathcal F_t^{M}$ and $\mathcal A_t^{M}$ from  \S\ref{sec:classes}. The following lemma will be used in \S\ref{sec:trees} for $M=50(d+1)$.

%%%%%%%%%%%%%%%%%%%%%%%%%%%%%% LEMMA LEMMA LEMMA
\begin{lemma}  \label{l:annulus} Let $M$ be a large integer. There is $A=A(M,d)$ such that if ${\bf{m}}$ satisfies assumption \eqref{e:admin4} the following holds.  
With reference to the expressions in \eqref{eq:modwotiles}, if $t=t(\beta,s,z)$ define the functions 
\[
\varphi_t(x)\coloneqq \scl(t)^{\frac{d}{2}}\mathrm{Tr}_{z} \phi_{\beta} (x), \quad {\vartheta}_t(x,\sigma)\coloneqq \scl(t)^{\frac{d}{2}}\int_{\R^n
} \phi_{\beta}(y-z)\psi_s (x-y,\sigma)\, \d y, \quad (x,\sigma) \in \R^n \times \Sigma.
\] Then 
\begin{align}
\label{e:uff}
& c\varphi_t \in  \mathcal F_t^{M}, \qquad c{\vartheta}_t \in  \mathcal A_t^{M-1} 
\end{align}
where the constant $c>0$ may be chosen uniformly in $\beta,s,z$. 
%\b{and $A$ is such that the assumption of Theorem~\ref{thm:main} is satisfied.} 
\end{lemma}
%%%%%%%%%%%%%%%%%%%%%%%%%%%%%% LEMMA LEMMA LEMMA

%%%%%%%%%%%%%%%%%%%%%%%%%%%%%% PROOF PROOF PROOF
\begin{proof}Firstly, property (ii) in the definition of the class $\mathcal F_t ^M$ follows by \eqref{eq:phibetasupp} together with condition 2. of Lemma~\ref{l:dyadic}. We now check condition (i) for $\varphi_t$. By rotation and translation we can assume that $t=t(e_n,s,0)${, namely $v_t=\beta=e_n$,} and write $\R^d\otimes \R\ni x\eqqcolon(y,x_n)$, where $s=\scl(t)$. Using property (ii) we readily see that the function 
\[
\widetilde {\mathbf 1}_B(x)\coloneqq \mathrm{Mod}_{-e_n} s^{d/2} \varphi_t(sy,x_n),\qquad x=(y,x_n)\in \R^d \times \R,
\]
 is a $\mathcal C ^\infty$ bump function with $\|\widetilde {\bm 1}_B\|_1\eqsim 1$, which is adapted of any order to the unit ball of $\R^n$, centered at the origin. Thus $c\widetilde {\bm 1}_B \in \Theta_M$ as defined in \S\ref{sec:classes}  \b{for any order $M$}, where $c>0$ is a normalization constant depending only upon dimension. We get that 
 \[
 c\varphi_t(x)=c \mathrm{Mod}_{e_n} s^{-d/2}\, \widetilde {\mathbf 1}_B(y/s,x_n),\qquad x=(y,x_n)\in \R^d \times \R,
 \]
which shows (i). These two remarks show that $c\varphi_t \in  \mathcal F_t ^M$ for any desired large positive integer $M$.

We move to the verification of property $c\vartheta_t \in\mathcal A_t ^M$. For this let us fix $\sigma \in \Sigma$ and note that $\vartheta_t(\cdot,\sigma)=\phi_\beta(\cdot-z)*\psi_s({\cdot,\sigma})$. In view of \eqref{eq:phihat} and \eqref{e:LPtheory}, this immediately implies the Fourier support condition 
\[
\supp\widehat{\vartheta_t(\cdot,\sigma)}\subseteq \supp \widehat{\phi_t} \subseteq  \omega_t.
\]
We now check condition (iv) for $\vartheta_t$. Recall that the tile $t$ has been chosen so that $\alpha_t$  contains $  \mathrm{Ann}\left(\beta, \textstyle 3^{-1}s^{-1}  ,3^{2}s^{-1}\right)  $, and   that {$\supp \widehat \phi_\beta \subset  \{  \xi \in \mathrm{Ann}( 1/2 ,2):\,\xi'\in B(\beta, {3^{-\kappa}s^{-1}})  \}$} which tells us that whenever $\xi \in \supp ( \mathrm{Tr}_{z}\phi_{\beta}*\psi_s{(\cdot,\sigma)} )^\wedge$ the following inequalities  
\[
\frac49 <s|\Pi_\sigma \xi'|<\frac{56}{9},  \qquad |\xi' 	-\beta |\leq {3^{-\kappa}s^{-1}},\qquad \xi'\coloneqq \xi/|\xi|\in\SS^d,
\]
must hold. Furthermore, {we have}
\[
\begin{split}
|\Pi_\sigma \xi'| +{3^{-\kappa}s^{-1}} \geq |\Pi_\sigma \xi'| + |\xi'-\beta|\geq |\Pi_\sigma \beta|\geq |\Pi_\sigma \xi'| - | \xi' -   \beta| \geq |\Pi_\sigma \xi'| -{3^{-\kappa}s^{-1}} 
\end{split}
\]
which, in view that {$\kappa=12+\log_3 d$} yields
{\[
|\Pi_\sigma \beta| \in s^{-1} \left [\textstyle\frac13,7 \right]\quad\text{which implies}\quad v_\sigma\in \mathrm{Ann}\left(\beta, \textstyle \frac{1}{3} s^{-1}, 9s^{-1}\right).
\]}
This proves that $\vartheta_t(\cdot, \sigma)=0$ unless $v_\sigma\in \mathrm{Ann}\left(\beta, \textstyle 3^{-1} s^{-1}, 3^{2}s^{-1}\right).$
Condition 3.\ of Lemma \ref{l:dyadic} then entails (iv) for $\vartheta=\vartheta_t$. 

We now move to our main task which is to verify that $c\vartheta_t(\cdot,\sigma)$ verifies conditions (iii) and (v) in the definition of the class $\mathcal A_t ^M$. We can again reduce to the case ${v_t}=\beta=e_n$ and $z=0$, {where we recall that $t=t(\beta,s,z)$}; we write $x=({y},x_n)\in \R^d \otimes \R$ and $\sigma\in\Sigma$, {and consider} $v_\sigma\in \mathrm{Ann}\left(e_n, \textstyle 3^{-1} s^{-1}, 3^{2}s^{-1}\right)$. Now, express $x$ in terms of $\Pi_{\sigma}x,\Pi_{\sigma^\perp}x$ as in Lemma \ref{lem:alg}, 
\begin{equation}\label{eq:compar}
\begin{split}
	y - \Pi_\sigma x &=q_1(\langle x, u\rangle,\langle x,v_\sigma\rangle)u+q_2(\langle x, u\rangle, \langle x, v_\sigma\rangle)v_\sigma,
\\
    x_n - \Pi_{\sigma^\perp} x& = q_2(\langle-x,v_\sigma\rangle,\langle x, u\rangle)u +q_1(\langle x, v_\sigma\rangle,\langle-x, u\rangle)v_\sigma,
	\\
	q_1(t,w) &\coloneqq \left[\frac{\sin2\theta}{2}w - (\sin\theta)^2 t\right],\qquad q_2(t,w)\coloneqq \left[\frac{\sin2\theta}{2}t+(\sin\theta)^2w\right].
\end{split}
\end{equation}
Above {$\theta$ is the angle between $v_\sigma$ and $e_n$} and $u\in \sigma$ is a unit vector orthogonal to $\sigma\cap e_n ^\perp$. By the already verified directional support condition (iv) we know that {$\theta\simeq|v_\sigma-e_n|\lesssim s^{-1}$}, which implies that
\begin{equation}\label{e:q}
|q_1(y,w)|+|q_2(y,w)|\lesssim s^{-1} \big( |y|+|w| \big).
\end{equation}
Combining \eqref{eq:compar} with \eqref{e:q} shows that the decay rates
\begin{equation}\label{eq:decayreates}
\left[\left(1+s^{-1}|\Pi_\sigma x|\right)\left(1+|\Pi_{\sigma^\perp} x|\right)\right]^{-M},\qquad \left[\left(1+s^{-1}|y|\right)\left(1+|x_n|\right)\right]^{-{M}}
\end{equation}
{are equivalent}. {Let us define the dilation operator $\mathrm{Dil}_\sigma ^s$ for $\sigma\in\Gr(d,n)$ as 
\[
\mathrm{Dil}_\sigma ^s g(x)\coloneqq s^{d/2}g(s\Pi_\sigma x+\Pi_{\sigma^\perp}x), \qquad x=(\Pi_\sigma x,\Pi_{\sigma^\perp} x)\in  \R^n.
\]}
{In view of \eqref{eq:decayreates} it will be enough to show} that for some dimensional constant $c>0$, the rescaled function
$
\R^n\ni x \mapsto  c\mathrm{Mod}_{-e_n} \mathrm{Dil}_\sigma ^s F(x)
$
is a bump function adapted to the unit ball of $\R^n$, centered at the origin, where 
\[
F=\vartheta_t (\cdot,\sigma) \qquad\text{or}\qquad F= \b{\vartheta_t (\cdot,\sigma)-\vartheta_t(\cdot,\rho),}
\qquad \rho,\sigma \in \Sigma,\quad \rho\neq \sigma.
\]
We begin with the first case where $F=\vartheta_t(\cdot,\sigma)$. Assuming again that $t=t(e_n,s,0)$ we note that for every multiindex $\alpha$
\[
\begin{split}
\partial_x ^\alpha [\mathrm{Mod}_{-e_n} \mathrm{Dil}_\sigma ^s \vartheta_t  (\cdot,\sigma)]& = [\partial_x ^\alpha \mathrm{Mod}_{-e_n} \mathrm{Dil} _\sigma ^s \varphi_t] *  \mathrm{Mod}_{-e_n} s^{d/2}  \mathrm{Dil} _\sigma ^s\psi_s(\cdot,\sigma) 
\\
&\eqqcolon [\partial_x ^\alpha \mathrm{Mod}_{-e_n} \mathrm{Dil} _\sigma ^s \varphi_t] * \mathrm{Mod}_{-e_n} G_1.
\end{split}
\]
We have already seen that the function $\mathrm{Mod}_{-e_n} \b{s^{d/2}} \varphi_t(sy,x_n)$, with $x=(y,x_n)\in\R^d\times \R$ is a smooth bump adapted at every order to the unit ball of $\R^n$ centered at the origin. By the remark following \eqref{e:q} this is equivalent to the fact that the function $ \mathrm{Mod}_{-e_n} \mathrm{Dil} _\sigma ^s \varphi_t(x)$ has the same adaptedness property. Thus, in order to show that $ \mathrm{Mod}_{-e_n} \mathrm{Dil}_\sigma ^s \vartheta_t$ is adapted to the unit ball of $\R^n$ centered at the origin \b{at order $M$} it suffices to show that the function $G_1$ \b{has decay of order $M$ at scale $1$ away from the origin}.
Then
\[
\begin{split}
 \widehat{ G_1}(\xi)\coloneqq ( s^{d/2}  \mathrm{Dil} _\sigma ^s\psi_s(\cdot,\sigma))^\wedge  (\xi)
 =    m_\sigma(\textstyle{\frac1s}O_\sigma\Pi_\sigma\xi)\Psi(\Pi_\sigma\xi)\Phi(\textstyle{\frac1s} \Pi_\sigma\xi+\Pi_{\sigma^\perp}\xi).
\end{split}
\]
We record the easy estimate $|\mathrm{supp}  \, \widehat G_1 |\lesssim_n 1$ which is due to the fact that on the support of the function $\widehat G_1$ we have
\begin{equation}\label{eq:support}
|\Pi_\sigma \xi|\eqsim 1\quad\text{and}\quad |\Pi_\sigma \xi|^2/s^2+|\Pi_{\sigma ^\perp }\xi|^2 \eqsim 1\implies |\Pi_{\sigma^\perp }\xi|\eqsim 1.
\end{equation}
It is then   apparent that for $M\leq A$ we have
\begin{equation}
\label{e:mainbd}
 \left\| \partial_{\xi}^{\gamma} \widehat G_1 (\xi) \right\|_{L^\infty_\xi}
 \lesssim 1, \qquad 0\leq   |\gamma| \leq M.
\end{equation}
 Here, we crucially use \eqref{eq:support}, allowing the exploitation of the H\"ormander--Mihlin condition on   $ m_\sigma$ in the form
 \[
 \left|\partial^{\gamma}_\xi   m_\sigma(s^{-1}O_\sigma \Pi_\sigma \xi) \right| \lesssim s^{-|\gamma|}\left |\textstyle{\frac{1}{s}} \Pi_\sigma \xi\right|^{-|\gamma|}  \|   m_\sigma\|_{\mathcal M_A(d)}  \lesssim 1,\qquad |\gamma|\leq A.
 \]
This completes the proof of adaptation of $F=\vartheta_{t}(\sigma,\cdot)$.
 
{ Let now $F=\vartheta_t(\cdot,\sigma)-\vartheta_t(\cdot,\rho)$. Firstly note that we can assume that $\rho,\sigma$ are such that $v_\sigma,v_\rho \in \alpha_t$ (which was defined in \eqref{eq:alfat}) for $t=t(e_n,s,0)$, see \eqref{eq:tile}. Indeed, if say $\rho\notin \alpha_t$ then $\vartheta_t(\cdot,\rho)=0$ by the directional
support condition (iv), which is already verified. Then $\dist(\rho,\sigma)\gtrsim \scl(t)^{-1}$ and the desired adaptedness property follows from the case $F=\vartheta_t(\cdot,\sigma)$ proved above.

Now, assume that $\rho,\sigma\in\alpha_t$. Using \eqref{e:ncomp} and the definition of $\alpha_t$ in \eqref{eq:alfat}, $\dist(\rho,\sigma)\leq 3\sqrt{d} 3^{-\kappa}\scl(t)^{-1}\leq 3^{-8}s^{-1}$. 
We can apply the same argument as above to write    
\[
\begin{split}
 \partial_x ^\alpha [\mathrm{Mod}_{-e_n}\mathrm{Dil}_\sigma ^s F ]&=\partial_x ^\alpha [\mathrm{Mod}_{-e_n}\mathrm{Dil}_\sigma ^s \varphi_t ]* \mathrm{Mod}_{-e_n}s^{d/2} \mathrm{Dil}_\sigma ^s [  \psi_s  (\cdot,\sigma)-\psi_s  (\cdot,\rho) ]
 \\
 &\eqqcolon \partial_x ^\alpha[\mathrm{Mod}_{-e_n} \mathrm{Dil}_\sigma ^s \varphi_t] * \mathrm{Mod}_{-e_n}G_2.
 \end{split}
\]
As before it will be enough to show that the function $G_2$ {decays with order $M$ at scale $1$ away from the origin; more precisely we will show that $G_2\lesssim c(\sigma,\rho,s)\chi_M$, with $c(\sigma,\rho,s)$ the multiplicative factor appearing in condition (v) of \S\ref{sec:classes}. We will do so by estimating derivatives of sufficiently high order of the Fourier transform of $G_2$.} To calculate the Fourier transform of $G_2$ it is convenient to first note that the anisotropic rescaling of $\psi_s  (\rho,\cdot) $ in the directions of $\sigma$ creates a small error that will be kept under control by the closeness of $\sigma$ and $\rho$.  
Setting
\[
B(\xi)\coloneqq \Pi_\rho \Pi_\sigma \xi +s \Pi_\rho \Pi_{\sigma^\perp}\xi - \Pi_\rho \xi = (s-1)\Pi_\rho \Pi_{\sigma ^\perp }\xi,\quad |B(\xi)|\lesssim s \dist(\sigma,\rho)|\xi| \ll |\xi|,
\]
an explicit computation tells us that 
\begin{multline*}
\widehat{G_2}(\xi) =   \left[     m_\sigma\left(\textstyle{\frac{1}{s}}  O_\sigma \Pi_\sigma \xi\right) \Psi(\Pi_\sigma \xi) -
  m_\rho\left(\textstyle{\frac{1}{s}}O_\rho(\Pi_\rho \xi + B(\xi) )\right) \Psi(\Pi_\rho \xi + B(\xi) ) \right] 
\\ 
\times \Phi\left(\textstyle{\frac{1}{s}}\Pi_\sigma\xi+ \Pi_{\sigma^\perp}\xi\right).
\end{multline*}
Note that since $\sigma \in \Sigma$ we always have that $|\xi|\eqsim |\Pi_{\sigma^\perp} \xi|\eqsim 1$ for $s^{-1}\Pi_\sigma \xi +\Pi_{\sigma^\perp} \xi \in \mathrm{supp}(\Phi)$ and in particular we recover that $|\Pi_\sigma \xi|\eqsim |\Pi_\rho \xi|\eqsim |\xi|\eqsim 1$ on the support of $\widehat {G_2}$.
Three mean value estimates combined with \eqref{e:admin4} and \eqref{e:opnormest} tell us that
\begin{equation}
\label{e:mve1}
\begin{split} &  
\left|
\partial^\gamma_\xi \left[\Psi(\Pi_\rho \xi + B(\xi)  ) -\Psi(\Pi_\sigma \xi )  \right] \right| \lesssim |B(\xi)| + |\Pi_\rho \xi-\Pi_\sigma \xi | \lesssim  s\dist(\sigma,\rho) ,
\\ 
&   \left|
\partial^\gamma_\xi \left[  m_{\rho}\left(\textstyle{\frac{1}{s}}O_{\sigma}\Pi_\sigma \xi\right) - 
\ m_\rho\left(\textstyle{\frac{1}{s}} O_{\rho}\Pi_\sigma \xi\right) \right] \right| 
\lesssim  \| O_{\rho}-O_{\sigma}\|\lesssim \dist(\sigma,\rho),
\\ 
&  \left|
\partial^\gamma_\xi \left[ m_{\rho}\left(\textstyle{\frac{1}{s}}O_{\rho}\Pi_\sigma \xi\right)- 
\ m_\rho\left(\textstyle{\frac{1}{s}} O_{\rho}(\Pi_\rho \xi + B(\xi))\right) \right] \right| 
 \lesssim |B(\xi)| + |\Pi_\rho\xi-\Pi_\sigma \xi |  \lesssim s\dist(\sigma,\rho) .
\end{split}
\end{equation}
 Hence, we may replace $ \widehat{G_2}$ by 
\[ 
\begin{split}
\widehat{G_3}(\xi) &\coloneqq   \left[  m_\sigma\left(\textstyle{\frac{1}{s}} O_{\sigma}\Pi_\sigma \xi\right)   -
  m_\rho\left(\textstyle{\frac{1}{s}} O_{\sigma}\Pi_\sigma \xi\right) \right] \Psi(\Pi_\rho \xi + B(\xi) ) \Phi(s^{-1}\Pi_\sigma\xi+ \Pi_{\sigma^\perp}\xi)
\end{split}
\]
 controlling  the difference by means of \eqref{e:mve1}. Finally, the main assumption \eqref{e:admin4} on the family ${\bf{m}}$ tells us that
 \[ \partial^\gamma_\xi\left[   m_\sigma\left(\textstyle{\frac{1}{s}} O_{\sigma}\Pi_\sigma \xi\right)   -
  m_\rho\left(\textstyle{\frac{1}{s}} O_{\sigma}\Pi_\sigma \xi\right)   \right] \lesssim  \frac{1}{\log\left(\e+ [\dist(\sigma,\rho)]^{-1}\right)}
 \] 
 on the support of $\widehat {G_3}$;  collecting the above inequalities leads to the conclusion
 \begin{equation}
\label{e:quotdiff}
 \left\| \partial_{\xi}^{\gamma} \widehat G_2 (\xi) \right\|_{L^\infty_\xi}
 \lesssim (1+s) \dist(\sigma,\rho)+ \frac{1}{\log\left(\e+ [\dist(\sigma,\rho)]^{-1}\right)}, \qquad 0\leq   |\gamma| \leq A-1,
\end{equation}
which completes the proof of (v).
}
\end{proof}
%%%%%%%%%%%%%%%%%%%%%%%%%%%%%% PROOF PROOF PROOF

With the notation introduced in Lemma \ref{l:annulus}, the right hand side of  \eqref{eq:modwotiles} may be rewritten as, and subsequently bounded by, as follows:
\begin{equation}\label{eq:modwtiles1}
\begin{split} &
\sup_{\sigma\in\Sigma} \bigg|\sum_{{s\in \mathbb S}} \sum_{\beta\in \mathcal B _{s^{-1}}} \sum_{z\in Z(\beta)}    \langle   f ,\varphi_{t(\beta,s,z)} \rangle   {\vartheta}_{t(\beta,s,z)} (\cdot,\sigma)\bigg| \leq \sum_{j=1}^{C}  \sum_{\tau=1} ^N \mathrm{T}_{j,\tau} f,
\\ 
&\mathrm{T}_{j,\tau} f\coloneqq \sup_{\sigma\in\Sigma} \left|\sum_{\substack{(\beta,s,z)\\ j(\beta,s)=j}}    \langle   f ,\varphi_{t(\beta,s,z)} \rangle   {\vartheta}_{t(\beta,s,z) }(\cdot,\sigma) \cic{1}_{\alpha_{t(\beta,s,z),\tau}} (v_\sigma) \right|.
\end{split}
\end{equation}
Property (iv) of the $\mathcal A_t$ class was used above. It thus suffices to control one of the {$ \mathrm{T}_{j_o,\tau_o} f$} summands for an arbitrary but fixed value of {$(j_o,\tau_o)\in \{1,\ldots,C\}\times\{1,\ldots,N\}$.}
Choose a measurable function $\sigma=\sigma(x):\R^{n}\to \Sigma$, linearizing the corresponding supremum in the definition of $\mathrm{T}_{j_o,\tau_o} f$, and $g\in\mathcal S(\R^n)$ of unit norm in $ L^{p',q'}(\R^n)$ so that  
{\[
\begin{split}
\|\mathrm{T}_{j_o,\tau_o}f \|_{L^{p,q}(\R^n)}& \leq 2 \left|\int  
\sum_{{s\in \mathbb S}} \sum_{\beta\in \mathcal B _{s^{-1}}} \sum_{z\in Z(\beta)}    \langle   f ,\varphi_{t(\beta,s,z)} \rangle   {\vartheta}_{t(\beta,s,z)} (x,\sigma(x))\cic{1}_{\alpha_{t(\beta,s,z),\tau_o}} (v_{\sigma(x)}) g(x)\,\d x \right|
\\
& \lesssim \Lambda_{\mathcal T_{\mathcal G_{j_o},M}; \sigma, \tau_o} (f,g)
\end{split}
\]
}
\subsection{Final reductions to a scale-separated model sum}
By a limiting argument and restricted weak-type interpolation, Proposition \ref{p:mainpf} is reduced to proving the two estimates
\begin{align}\label{eq:model}
&\Lambda_{\mathbb P; \sigma, \tau_o,M} (f,g\cic{1}_E)\lesssim \|f\|_{2} |E|^{\frac12}
\\ &\Lambda_{\mathbb P; \sigma, \tau_o,M} (\cic{1}_F,g\cic{1}_{{E}})\lesssim |F|^{\frac1p}|E|^{1-\frac{1}{p}}
\label{eq:model2}
\end{align}
uniformly over all $f,g\in L^\infty_0(\R^n)$  with the normalization $\|g\|_\infty=1$, all sets $F,E\subset \R^d$ of finite measure, a fixed choice of triadic grid $\mathcal G=\mathcal G_{j_o}$ and  $\mathbb P\subset \mathcal T_{\mathcal G}$.

 {Henceforth, we fix $M=50n$ and a pair $(j_o,\tau_o)$ as above} and write {$\Lambda_{\mathbb P;\tau_o}(f,g)$ in place of $\Lambda_{\mathbb P; \sigma, \tau_o,M} (f,g)$.} {By an additional splitting of the dyadic grid $\mathcal G,$ we can and will assume that the scales are $3^\kappa$ separated: namely if $Q,Q' \in\mathcal G$ with $\ell(Q)<\ell(Q')$ then $\ell(Q)<3^{-\kappa}\ell(Q')$. Note that this implies a corresponding separation of the spatial scales of tiles in $\mathcal T_{\mathcal G}$ because of the uncertainty constraint. We will also make the qualitative assumption that $\mathbb P$ is a finite collection and prove estimates which are uniform in $\mathbb P$.} The proofs of \eqref{eq:model} and \eqref{eq:model2} will be sketched at the end of \S\ref{sec:trees}, relying upon the size, density and tree lemmas, also proved in \S\ref{sec:trees} below.

%%%%%%%%%%%%%%%%%%%%%%%%%%%%%% SECTION SECTION SECTION
\section{Trees, size and density}\label{sec:trees} We fix a finite collection of tiles $\mathbb P\subset \mathcal T_{\mathcal G}$, for a fixed choice of triadic grid $\mathcal G$. We remember that tiles are sets of the form $t=R_t \times Q_t$ where $Q_t\in\mathcal G$ is  a triadic cube in $\R^d \times \{0\}$.  Recall the roles of $Q_t ^{\circ,\kappa}$ and $\mathrm{ch}_{\kappa, \square}(Q)$; since $\kappa$ is a fixed large dimensional constant, we will omit it from the notation from now on and write instead $Q_t ^\circ$ and $\mathrm{ch}_{\square}(Q)$ for the \emph{center} and the \emph{peripheral children} of $Q$, respectively.

We will be sorting these arbitrary collections of tiles into \emph{trees}, defined below. {It will be useful to remember the definition of the directional support $Q_{t,\tau}$ from \eqref{eq:alfat}.}

%%%%%%%%%%%%%%%%%%%%%%%%%%%%%% DEFINITION DEFINITION DEFINITION
\begin{definition}[trees]\label{def:trees} A \emph{tree} is a collection of tiles $\mathbf{T}\subseteq \mathbb P$ such that, there exists a pair  $(\xi_\mathbf{T},R_{\mathbf T})$ with $\xi_\mathbf{T}\in\R^d$ and $R_{\mathbf T}\in \{R_t:\, t\in\mathcal T_{\mathcal G}\}$  such that
\begin{itemize}
	\item [1.] $\xi_{\mathbf T} \in Q_t $ for all $t\in\mathbf T$;
	\item [2.] \b{$\scl(R_t)\leq \scl(R_{\mathbf T})$ and $R_t\cap R_{\mathbf T}\neq \varnothing $ for every $t\in\mathbf T$.
	}
\end{itemize}
{We call $(\xi_\T,R_\T)$ the \emph{top data} of $\T$. Let $\tau\in\{1,\ldots,N\}$; a \emph{tree} $\mathbf T$ will be called $\tau$-\emph{lacunary} if {$\xi_{\mathbf T} \in Q_{t,\tau} \in \mathrm{ch}_{\square}(Q_t)$} for all $t\in\mathbf T$. The tree $\T$ will be called \emph{lacunary} if it is $\tau$-lacunary for some $\tau\in\{1,\ldots,N\}$ while it will be called \emph{overlapping} if $\xi_{\mathbf T}\in Q_t\setminus \cup_{\tau=1} ^N Q_{t,\tau}$.}
\end{definition}
%%%%%%%%%%%%%%%%%%%%%%%%%%%%%% DEFINITION DEFINITION DEFINITION

The following relation of order will be useful: For two tiles $t,t'\in \mathcal T_{\mathcal G}$ we will define
\[
t\leq t' \overset{\mathrm{def}}{\iff}  Q_{t'} \subseteq Q_t   \quad \text{and}\quad \b{R_t \cap R_{t'}\neq \varnothing}
\]
Note that this is \emph{not} a partial order as it is not transitive.  The following geometric lemma will be used in several places throughout the paper.

%%%%%%%%%%%%%%%%%%%%%%%%%%%%%% LEMMA LEMMA LEMMA
\begin{lemma}\label{lem:geom} There exists a constant $K_n$ depending only upon dimension, such that for every $t,t'\in \mathbb P$ with $Q_{t'}\subseteq Q_t$, the following hold.
\begin{itemize}
\item[-] If $R_t \cap R_{t'}\neq \varnothing$ then $R_t \subseteq K_n R_{t'}$.
\item[-]  If $K_n R_t \cap K_n R_{t'}\neq \varnothing$ then $K_n R_t \subseteq K_n ^2 R_{t'}$.
\end{itemize}
\end{lemma}
%%%%%%%%%%%%%%%%%%%%%%%%%%%%%% LEMMA LEMMA LEMMA

%%%%%%%%%%%%%%%%%%%%%%%%%%%%%% REMARK REMARK REMARK
\begin{remark}\label{rmrk:treefacts} The following facts about trees will be used without particular mention.
\begin{itemize}
	\item [-] If $\mathbf T$ is a tree then there exists a \emph{top tile} $\mathrm {top}(\mathbf T)=R_{\mathbf T}\times Q_{\mathbf T}$ such that $\xi_\T\in Q_t$ and $t\leq \mathrm{top}(\mathbf T)$ for all $t\in\mathbf T$.  {In particular we have that for lacunary $\mathbf T$ there holds $t\leq \mathrm{top}(\T)$ and $Q_\T\subseteq Q_{t}\setminus Q_{t} ^\circ$ for all $t\in\T$. Indeed, note that if $\T$ is lacunary then $\ell(Q_t)>3^{\kappa}\ell(Q_\mathbf{T})$ for all $t\in \mathbf{T}$ because of the separation of scales assumption for the grid $\mathcal G$.} 
	\item [-]  Because of the point above, we can and will always assume that $\xi_\T$ has no triadic coordinates.
	\item [-] If $\T$ is a lacunary tree then the collections $\{ Q:\, Q=Q_t ^\circ\,\,\,\text{for some} \,\,\, t\in\T\}$ and  $\{ \omega:\, \omega=\omega_t \,\,\,\text{for some} \,\,\, t\in\T\}$ are pairwise disjoint.
\end{itemize}
\end{remark}
%%%%%%%%%%%%%%%%%%%%%%%%%%%%%% REMARK REMARK REMARK

For the following definition, $E\subset\R^n$ is a measurable set of finite measure. This set remains fixed throughout the paper. Also for any collection {$\mathbf R$ consisting of of rectangular parallelepipeds in $\R^n$} we define the \emph{shadow} of the collection
{
\[
\sh(\mathbf R) \coloneqq  \bigcup_{R\in\mathbf R} R.
\]
}
Recall also the definition of the intrinsic coefficients $F[f](t)$ from \eqref{e:wpt}. The next step consists of defining two coefficient maps respectively tied to the Carleson measure properties of the coefficients \eqref{e:wpt}, \eqref{e:wpt2}

%%%%%%%%%%%%%%%%%%%%%%%%%%%%%% DEFINITION DEFINITION DEFINITION
\begin{definition}[Size] Let $\mathbb P\subset \mathcal T_{\mathcal G}$ be a finite collection of tiles. For fixed $f\in L^\infty_0(\R^n)$, define the map
\[
\size: \mathcal P(\mathbb P) \to [0,\infty), \qquad \size(\mathbb Q) \coloneqq \sup_{\substack{{\mathbf T}\subseteq \mathbb Q\\\mathbf T\text{ lac. tree}}} \Big(\frac{1}{|R_\T|}\sum_{t\in\T} F_{10n}[f](t)^2 \Big)^{\frac12}.
\]
 \end{definition}

\begin{definition}[Density] Let $\mathbb P$ be a finite collection of tiles and $t\in\mathbb P$. For a fixed measurable set $E\subset \R^n$ of finite measure, we define
	\begin{equation} \label{e:densedef}	E_t \coloneqq \left\{x\in E:\, v_{\sigma (x)}\in \alpha_t \right\}, \qquad	\dense (t) \coloneqq { \sup_{\substack{t'\geq t\\t'\in\mathcal T_{\mathcal G}}}\int_{E_{t'}}  \mathrm{Sy}_{R_{t'}} ^{1} \chi_{10n},}
\end{equation}
as well as the map
\[
\dense: \mathcal P(\mathbb P) \to [0,\infty), \qquad
 \dense(\mathbb Q) \coloneqq \sup_{t\in \mathbb Q}\dense(t).
\]
\end{definition}
%%%%%%%%%%%%%%%%%%%%%%%%%%%%%% DEFINITION DEFINITION DEFINITION

%%%%%%%%%%%%%%%%%%%%%%%%%%%%%% SECTION SECTION SECTION
\subsection{Orthogonality estimates for lacunary trees} We will use different orthogonality estimates for wave packets adapted to special collections of tiles. Most of them are rather standard in the literature but we include them here for completeness.

%%%%%%%%%%%%%%%%%%%%%%%%%%%%%% LEMMA LEMMA LEMMA
\begin{lemma}\label{lem:ortho} Let $\mathbf T\subset\mathbb P$ be a lacunary tree. Then the following hold
\begin{itemize}
  \item [1.]  For adapted families $\{\varphi_t:\, \varphi_t\in\mathcal F_t ^M,\, t\in \T\}$, {$\{\psi_t:\, \psi_t\in\mathcal F_t ^M,\, t\in \T\}$} with $M>2n$, we have
\[
\sup_{|\eps_t|=1}\left \|\sum_{t\in\mathbf T} \eps_t\l f,\varphi_t\r \psi_t \right\|_{L^2(\R^n)} \lesssim \left(\sum_{t\in\T} F_M[f](t) ^2 \right)^{\frac12} \lesssim \|f\|_{L^2(\R^n)} .
\]
with implicit dimensional constants.
\item[2.] There holds
\[
\sum_{t\in\T} F_{10n}[f](t)^2 \lesssim \|f\|_{L^\infty(\R^n)} ^2|R_\T|.
\]

\end{itemize}

\end{lemma}
%%%%%%%%%%%%%%%%%%%%%%%%%%%%%% LEMMA LEMMA LEMMA

%%%%%%%%%%%%%%%%%%%%%%%%%%%%%% PROOF PROOF PROOF
\begin{proof} Let $S_\T f \coloneqq\sum_{t\in\T}  \eps_t\l f,\varphi_t\r \psi_t$ and  $Q(\T)\coloneqq\{Q_t ^\circ:\, t\in\T\}$. For $Q\in Q(\T)$ we also write $\T(Q)\coloneqq\{t\in\T:\, Q_t ^\circ=Q\}$. Then
\[
S_\T f =\sum_{Q\in Q(\T)} \sum_{t\in\T(Q)} \eps_t\l f,\varphi_t\r \psi_t\eqqcolon \sum_{Q\in Q(\T)} S_Q(f)
\]
and because of Remark~\ref{rmrk:treefacts} we have
\[
\|S_\T f\|_2 ^2 \leq \sum_{Q\in Q(\T)} \| S_Q f \|_2 ^2.
\]
{We next show} that for each $Q\in Q(\T)$ we have
\[
\|S_Q(g)\|_2 ^2 \lesssim \sum_{t\in\T(Q)} F_M[g](t) ^2.
\]
Expanding the square we get for each $Q\in Q(\T)$ {
\[
\|S_Q(g)\|_2 ^2 \leq \sum_{t,t'\in \mathbf T(Q)}| \l g,\varphi_t\r| | \l g,\varphi_{t'}\r| | \l\psi_t,\psi_{t'}\r | \lesssim \sum_{t\in\T(Q)} F_M[g](t)^2  \sup_{{t'}\in\T(Q)}	\sum_{t\in\T(Q)}|\l \psi_t,\psi_{t'}\r|.
\]
However $\sup_{t'\in\T(Q)}\sum_{t\in\T(Q)}|\l \psi_t,\psi_{t'}\r|\lesssim 1$ by the fast decay of the wave packets $\{\psi_t:\, t\in\mathbf T(Q)\}$} and the fact that the spatial components $\{R_t:\, t\in\mathbf T(Q)\}$ tile $\R^n$. This proves the first estimate of 1.

To see the second estimate {we pick for each $t\in\T$ a $\varphi_t\in \mathcal F_t ^M$} and write
\[
\sum_{t\in\T} {\left| \l f,\varphi_t\r \right|^2} =\left \l f,\sum_{{t}\in \T}\l f,\varphi_t\r \varphi_t\right\r\lesssim \|f\|_{L^2(\R^n)} \left(\sum_{t\in\T} F_M[f](t)^2\right)^{\frac12}
\]
where we used 1. in the last approximate inequality. This readily yields the desired estimate.

Finally for 2. let {$\varphi_t \in \mathcal F_t ^{10n}$ for each $t\in\T$} and note that $R_t\subseteq K_n R_\T$ by Lemma~\ref{lem:geom} and let $t_\T\coloneqq \mathrm{Sy}_{K_nR_\T} ^\infty \chi_{4n}$. Now for $c>0$ we use 2. to estimate{
\[
\sum_{t\in\T}  \left| \l f,\varphi_t\r \right|^2 =\sum_{t\in\T} \left| \l c^{-1}t_\T f, c t_{\T} ^{-1}\varphi_t\r \right|^2\lesssim \|f t_\T\|_{L^2(\R^n)} ^2
\]
}
by noticing that for a suitable choice of $c$ the wave packets $c t_{\T} ^{-1}\varphi_t \in \mathcal F_t ^{6n}$ for each $t\in\T$. The estimate in the last display readily yields 3. 
\end{proof}
%%%%%%%%%%%%%%%%%%%%%%%%%%%%%% PROOF PROOF PROOF

%%%%%%%%%%%%%%%%%%%%%%%%%%%%%% SECTION SECTION SECTION
\subsection{Strongly disjoint families of trees and almost orthogonality} In what follows we will be sorting our collection of tiles $\mathbb P$ into trees selected by means of a greedy algorithm. The selection process will imply certain disjointness properties which we encode in the definition below.

%%%%%%%%%%%%%%%%%%%%%%%%%%%%%% DEFINITION DEFINITION DEFINITION
\begin{definition}[strongly disjoint lacunary families]\label{def:strodis} Let $\mathscr T$ be a family of trees with $\mathbf T\subset \mathbb P$ for every $\mathbf T\in\mathscr T$. The family $\mathscr T$ is called (lacunary) strongly disjoint if
\begin{itemize}
	\item [1.] Every tree $\mathbf T\in\mathscr T$ is lacunary.
	\item [2.] If $t\in \mathbf T{\in\mathscr T}$ and $t'\in \mathbf T'\in\mathscr T$ with $\mathbf T \neq \mathbf T'$ and $Q_{t}\subseteq Q_{t'} ^\circ$ then $R_{t'}\cap  K_n ^2 R_{\mathbf T}= \varnothing$
\end{itemize}
with $K_n$ the constant of $Lemma~\ref{lem:ortho}$.
\end{definition}
%%%%%%%%%%%%%%%%%%%%%%%%%%%%%% DEFINITION DEFINITION DEFINITION
The point of the definition above is that if conclusion 2. failed then the tile $t'$ would essentially qualify to be included in a completion of the tree $\T$, if that tree was chosen first via a greedy selection algorithm. This point will become apparent in the proof of the size lemma, Lemma~\ref{lem:size} below. Furthermore, the definition above implies the following property.

%%%%%%%%%%%%%%%%%%%%%%%%%%%%%% LEMMA LEMMA LEMMA
\begin{lemma}\label{lem:easydisj} Let $\mathbb T\coloneqq \cup_{\T\in\mathscr T} \T$ where $\mathscr T$ is a lacunary strongly disjoint family. Then the collection 
\[
\mathbb T^\circ\coloneqq \{R_t \times Q_t ^\circ:\, t\in\mathbb T\}
\]
is pairwise disjoint.
\end{lemma}
%%%%%%%%%%%%%%%%%%%%%%%%%%%%%% LEMMA LEMMA LEMMA

%%%%%%%%%%%%%%%%%%%%%%%%%%%%%% PROOF PROOF PROOF
\begin{proof} Let $t,t'\in\mathbb T$. If $Q_t ^\circ=Q_{t'} ^\circ$ then either $t\equiv t'$ or $R_{t'}\cap R_t=\varnothing$. Assume then that $Q_t ^\circ \subsetneq Q_{t'} ^\circ$ so that {$\ell(Q_t )<3^{-\kappa}\ell(Q_{t'})$, because of separation of scales assumption, and consequently} $Q_t \subseteq Q_{t'} ^\circ$. Since the trees in the collection are lacunary this implies that $t,t'$ must come from different trees, say $\T\neq \T'$, respectively. If $R_t \cap R_{t'}\neq \varnothing$ then by consecutive applications of Lemma~\ref{lem:geom} we would have that $R_{t'}\subseteq K_n ^2 R_{\T}$, contradicting the definition of a strongly disjoint family.
\end{proof}
%%%%%%%%%%%%%%%%%%%%%%%%%%%%%% PROOF PROOF PROOF

It is well known that strongly disjoint families of trees obey certain almost orthogonality estimates for the corresponding tree projections. The precise statement in the context of this paper is contained in Lemma~\ref{lem:strodis} below.

%%%%%%%%%%%%%%%%%%%%%%%%%%%%%% LEMMA LEMMA LEMMA
\begin{lemma}\label{lem:strodis} Let $M> 4n$. For a (lacunary) strongly disjoint family $\mathscr T$ with $\mathbb T\coloneqq \cup_{\T\in\mathscr T}\T$, there holds
\[
\left(\sum_{\T\in\mathscr T} \sum_{t\in\T} F_M[f](t)^2 \right)^{\frac12}\lesssim \|f\|_{L^2(\R^n)}+\left(\sup_{t\in\mathbb T}\frac{F_M[f](t)}{|R_t|^{\frac12}} \left[\sum_{\T\in\mathscr T}|R_{\T}|\right]^{\frac12}\right)^{\frac13}\|f\|_{L^2(\R^n)} ^{\frac23}.
\]
\end{lemma}
%%%%%%%%%%%%%%%%%%%%%%%%%%%%%% LEMMA LEMMA LEMMA

%%%%%%%%%%%%%%%%%%%%%%%%%%%%%% PROOF PROOF PROOF
\begin{proof} Let $\{\varphi_t:\, \varphi_t\in\mathcal F_t ^M,\, t\in\mathbb T\}$ be any adapted family and note that
\[
S(f)^2\coloneqq\sum_{t\in\mathbb T} \left| \l f,\varphi_t\r \right|^2\leq \|f\|_2 \left\| \sum_{t\in\mathbb T} \l f,\varphi_t\r\varphi_t\right\|_2.
\]
Letting $Q(\mathbb T)\coloneqq \{Q_t:\, t\in\mathbb T\}$ and $\mathbb T(Q)\coloneqq\{t:\, Q_t=Q\}$ for $Q\in Q(\mathbb T)$, we have
{
\[
\begin{split}
\left\| \sum_{t\in\mathbb T} \l f,\varphi_t\r\varphi_t\right\|_2 ^2& \leq \sum_{Q\in Q(\mathbb T)} \sum_{\substack{t,t'\in\mathbb T\\Q_t=Q_{t'}=Q}} \left| \l f,\varphi_t\r \right| \left| \l f,\varphi_{t'}\r \right| \left| \l \varphi_t,\varphi_{t'}\r \right|
\\
&\qquad +2 \sum_{t\in\mathbb T} \sum_{\substack{t'\in\mathbb T\\ Q_{t'}\supsetneq Q_t}}  \left| \l f,\varphi_t\r \right| \left| \l f,\varphi_{t'}\r \right| \left| \l \varphi_t,\varphi_{t'}\r \right|
\\
&\coloneqq S_1(f)^2+2S_2(f)^2.
\end{split}
\]
}
As in the proof of Lemma~\ref{lem:ortho} it is relatively easy to see that $|S_1(f)|\lesssim S(f)$ since for a fixed $Q\in Q(\mathbb T)$, the spatial components $\{R_t:\, Q_t=Q\}$  tile $\R^n$. We thus focus on $S_2(f)$ which can be estimated in the form
\[
\begin{split}
S_2(f)^2 &\leq \sup_{t'\in\mathbb T}\frac{|\l f,\phi_{t'}\r|}{|R_{t'}|^{\frac12}} \sum_{t\in\mathbb T}    |\l f,\varphi_t\r|   \sum_{\substack{t'\in\mathbb T\\ Q_{t'} ^\circ \supseteq Q_t}}  |R_{t'}|^{\frac12} |\l \varphi_t,\varphi_{t'}\r|.
\end{split}
\]
Note that for fixed $t\in \mathbb T$ the collection {$\mathbf R_{\mathbb T}(t)\coloneqq \{R_{t'}:\, t'\in\mathbb T,\, Q_t\subseteq Q_{t'} ^\circ\}$} is pairwise disjoint and {$\sh(\mathbf R_{\mathbb T}(t))\subseteq (K_n ^2 R_\T)^{\mathrm c}$}; this follows immediately from Lemma~\ref{lem:easydisj} since $\mathscr T$ is a (lacunary) strongly disjoint family, and from the definition of a strongly disjoint family itself. At this point we use the following estimate: For $t,t'\in\mathbb P$ such that $R_t\cap R_{t'}= \varnothing$ and $Q_{t}\subseteq Q_{t'}$ there holds
\[
|\l \varphi_t,\varphi_{t'}\r|\lesssim \left(\frac{|R_{t' }|}{|R_{t}|}\right)^{\frac12} \inf_{x\in R_{t'}}(1+\rho_{R_t}(x))^{-M+n}
\]
 where $\rho_{R}(x)\coloneqq \inf\{r>0:\, x\in r R\}$ for any rectangular parallelepiped $R$ and $ M> 2n$, say. Combining these observations and estimates with the estimate for $S_2(f)^2$ above yields
\[
\begin{split}
S_2(f)^2&\lesssim\sup_{t'\in\mathbb T}\frac{|\l f,\phi_{t'}\r|}{|R_{t'}|^{\frac12}} \sum_{\T\in\mathscr T} \sum_{t\in \T}    |\l f,\varphi_t\r|  |R_t|^{-\frac12}  \sum_{\substack{t'\in\mathbb T\\ Q_{t'} ^\circ \supseteq Q_t}} |R_{t'}|   \inf_{R_{t'}}(1+\rho_{R_t})^{-M+n}
\\
&\lesssim \sup_{t'\in\mathbb T}\frac{|\l f,\phi_{t'}\r|}{|R_{t'}|^{\frac12}} \sum_{\T\in\mathscr T} \sum_{t\in \T}   |\l f,\varphi_t\r|  |R_t|^{-\frac12}    \int_{(K_n ^2R_{\T(t)}) ^\mathrm{c}}  (1+\rho_{R_t} (x))^{-M+n}\, \d x
\\
&\lesssim \sup_{t'\in\mathbb T}\frac{|\l f,\phi_{t'}\r|}{|R_{t'}|^{\frac12}} S(f)\left(\sum_{t\in\mathbb T} \int_{(K_n ^2R_{\T(t)}) ^\mathrm{c}}  (1+\rho_{R_t} (x))^{-M+n}\right)^{\frac12},
\end{split}
\]
where $\T(t)$ denotes the unique tree $\T\in\mathscr T$ such that $t\in\T$. In passing to the last line we used the Cauchy--Schwarz inequality together with the simple estimate
\[
\int_{(K_n ^2R_{\T}) ^\mathrm{c}}  (1+\rho_{R_t} (x))^{-M+n}\lesssim |R_t|, \qquad M> 2n.
\]
The estimate for $S_2$ is completed by involving the estimate
\[
\sum_{t\in\mathbb T}\int_{(K_n ^2R_{\T(t)}) ^\mathrm{c}}  (1+\rho_{R_t} (x))^{-M+n}=\sum_{\T\in\mathscr T}\sum_{t\in\T} \int_{(K_n ^2R_{\T}) ^\mathrm{c}}  (1+\rho_{R_t} (x))^{-M+n}\lesssim \sum_{\T\in\mathscr T} |R_\T|
\]
for $M> 2n$.  Balancing the estimates for $S(f)$ in terms of $S_1(f),S_2(f)$  according to which term dominates, and taking into account \eqref{e:wpt}, yield the desired bound.
\end{proof}
%%%%%%%%%%%%%%%%%%%%%%%%%%%%%% PROOF PROOF PROOF

{
%%%%%%%%%%%%%%%%%%%%%%%%%%%%%% SECTION SECTION SECTION
\subsection{Selection by density} The sorting of a collection of trees by density follows the standard procedure as for example in \cite{LTC}; we include the proof for completeness.

%%%%%%%%%%%%%%%%%%%%%%%%%%%%%% LEMMA LEMMA LEMMA
\begin{lemma}[density lemma]\label{lem:dense} Let $\mathbb T\subset \mathbb P$ be a finite collection of tiles. There exists a disjoint decomposition
\[
\mathbb T= \mathbb T^{\mathrm{light}} \cup \bigcup_{\tau=1}^N \T_{{\tau}}
\]
where   $\{\T_\tau :\, 1\leq \tau \leq N\}$ is a finite collection of trees and
\[
\sum_{\tau=1}^N |R_{\T_\tau}|  \lesssim  \dense(\mathbb T)^{-1} |E|
,\qquad  {\dense}(\mathbb T^{\mathrm{light}}) \leq \frac12 \dense(\mathbb T),
\]
where $E$ refers to the measurable $E\subset \R^d$ in the definition of density, with implicit constant depending only upon dimension.
%where $E\subset\R^n$ is the set used in the definition of density.
\end{lemma}
%%%%%%%%%%%%%%%%%%%%%%%%%%%%%% LEMMA LEMMA LEMMA
 
%%%%%%%%%%%%%%%%%%%%%%%%%%%%%% PROOF PROOF PROOF
\begin{proof} We let $\mathbb T^{\text{heavy}}$ be the collection of tiles
\[
\mathbb T^{\mathrm{heavy}}\coloneqq \left\{t\in\mathbb T:\, \dense(t)>\dense(\mathbb T)/2\right\}.
\]
By the definition of density, for each tile $t\in\mathbb T^{\mathrm{heavy}}$ there exists some $\mathcal T_{\mathcal G} \ni t'\geq t$ such that
\[
\int_{E_{t'}}  \mathrm{Sy}_{R_{t'}} ^{1} \chi_{10n}>\dense(\mathbb T)/2.
\]
We let $\mathbb T'$ denote the maximal, with respect to $\leq$, elements of $\{t':\, t'\in \mathbb T^{\mathrm{heavy}}\}$: namely, for all $t'\in \mathbb T'$ there does not exist $s'\in \mathbb T^{\mathrm{heavy}}$ with $t'\leq s'$. It will then suffice to prove
\[
\sum_{t'\in\mathbb T'}|R_{t'}| \lesssim \dense(\mathbb T)^{-1}|E|
\]
as the tiles in $\mathbb T^{\mathrm{heavy}}$ can be organized into trees with top datas in $\mathbb T'$ and by definition 
\[
\dense(\mathbb T^{\mathrm{light}})=\dense(\mathbb T\setminus \mathbb T^{\mathrm{heavy}})\leq \frac12 \dense(\mathbb T).
\]
To that end we define for $k\geq 0$ the collection $\mathbb T_k '$ to be set of all tiles $t'\in\mathbb T'$ such that
\begin{equation}\label{eq:dense}
\left|E_{t'} \cap 2^k R_{t'}\right|\geq \frac14  2^{5kn}\dense(\mathbb T) |R_{t'}|.
\end{equation}
Each element of $\mathbb T'$ is contained in at least one of the sets $\mathbb T_k '$; indeed if not we would have
\[
\int_{E_{t'}} \mathrm{Sy}_{R_{t'}} ^{1} \chi_{10n}\leq  \frac{|E_{t'}\cap R_{t'}|}{|R_{t'}|}+\sum_{k\geq 0}2^{-10kn}\frac{|{E_{t'}\cap (2^{k+1}R_{t'}\setminus 2^k R_{t'}})|}{|R_{t'}|}<\frac12 \dense(\mathbb T).
\]
We fix $k$ for the moment. We will decompose each collection $\mathbb T_k$ into subcollections of tiles with pairwise disjoint spatial components. We do this by choosing $t'\in\mathbb T' _k$ such that  $R_{t'}$ has maximal length and let 
\[
\mathcal E(t')\coloneqq\left\{t''\in \mathbb T' _k:\, 2^kR_{t'}\cap 2^kR_{t''}\neq \varnothing,\, Q_{t'}\cap Q_{t''}\neq \varnothing \right \}.
\]
Note that the collection $\{R_{t''}:\, t''\in\mathcal E(t')\}$ is pairwise disjoint. Indeed if for some $t'' _1,t'' _2\in \mathcal E(t')$ we had $R_{t'' _1}\cap R_{t'' _2}\neq \varnothing$ then, since $Q_{t'' _1}\cap Q_{t'' _2}\neq \varnothing $ we would conclude that either $t'' _1\leq t'' _2$ or $t'' _1\geq t'' _2$ which is impossible because all the tiles in $\mathbb T_k$ were maximal to begin with. Furthermore, we have that for all $t''\in\mathcal E(t')$ there holds $2^kR_{t'}\cap 2^kR_{t''}\neq \varnothing$ and $Q_{t'}\subseteq Q_{t''}$ and so a variation of Lemma~\ref{lem:geom} implies that $2^{2k}R_{t''}\subseteq K_n2^{2k}R_{t'}$ for all $t''\in\mathcal E(t')$. Now we replace $\mathbb T_k$ by $\mathbb T_k \setminus \mathcal E(t')$ and repeat so that $\mathbb T_k =\cup_{j} \mathcal E(t' _j)$ with $\{2^kR_{t' _j}\}_j$ disjoint and
\[
\begin{split}
\sum_{t'\in\mathbb T_k}|R_{t'}|&\leq \sum_j \sum_{t''\in\mathcal E(t' _j)}|R_{t''}|\leq K_n ^n 2^{2kn}\sum_j |R_{t' _j}|
\\
&\lesssim 2^{-3kn} \dense(\mathbb T)^{-1}\sum_j |2^kR_{t' _j} \cap E_{t' _j}|\lesssim 2^{-3kn} \dense(\mathbb T)^{-1}|E|
\end{split}
\]
where we used \eqref{eq:dense} in passing to the second line, and the fact that $\{2^kR_{t' _j}\}_j$ is a pairwise disjoint collection in the last approximate inequality. The conclusion now follows by summing in $k$.
\end{proof}
%%%%%%%%%%%%%%%%%%%%%%%%%%%%%% PROOF PROOF PROOF
}

%%%%%%%%%%%%%%%%%%%%%%%%%%%%%% SECTION SECTION SECTION
\subsection{Selection by size} We describe in the subsection the suitable version of the size (or energy) lemma, that will eventually allows us to sort a collection of tiles into trees of controlled size. The argument is standard in the literature but we introduce a small variation to account for the lack of transitivity in the relation of order that is implicit in the definition of trees. A technical device used in order to prioritize the selection of trees appears only in our higher dimensional setting and is defined below.

%%%%%%%%%%%%%%%%%%%%%%%%%%%%%% DEFINITION DEFINITION DEFINITION
\begin{definition}(signature) Let $\mathcal G$ be a triadic grid of {pace $(\kappa,m)$ in $\R^d$, $m\in [0,\kappa)$,} and $\xi\in \R^d$ having no triadic coordinates. The \emph{signature} of $\xi$ is the number $\mathrm{sig}(\xi)$ defined as follows
\[
\mathrm{sig}(\xi)\coloneqq \sum_{j=1} ^\infty\frac{a_j}{3^j},\qquad a_j\coloneqq 
\begin{cases} 
& 0,\qquad\text{if}\quad \xi\in Q^o\quad\text{for some}\quad Q\in {\mathcal G_{-\kappa j+m}},
\\
& 1,\qquad\text{if}\quad \xi \notin Q^o\quad\forall Q\in {\mathcal G_{-\kappa j+m},}
\end{cases}
\]
where we remember that {$\mathcal G_{u}=\{Q\in\mathcal G,\, \ell(Q)=3^{u}\}$;} see Subsection \ref{sec:dyadic}.
\end{definition}
%%%%%%%%%%%%%%%%%%%%%%%%%%%%%% DEFINITION DEFINITION DEFINITION

%%%%%%%%%%%%%%%%%%%%%%%%%%%%%% REMARK REMARK REMARK
\begin{remark}\label{rmrk:sign} The following point motivates the definition above: if $\T,\T'$ are two different {lacunary} trees and $t\in\T,t'\in\T'$ with $Q_{t}\subseteq Q_{t'} ^\circ$, then $\mathrm{sig}(\xi_{\T})<\mathrm{sig}(\xi_{\T'})$. This fact will play a role in proving that the family of trees extracted in the size lemma below is strongly disjoint.
\end{remark}
%%%%%%%%%%%%%%%%%%%%%%%%%%%%%% REMARK REMARK REMARK

%%%%%%%%%%%%%%%%%%%%%%%%%%%%%% LEMMA LEMMA LEMMA
\begin{lemma}[size lemma]\label{lem:size} Let $\mathbb T\subset \mathbb P$ be a finite collection of tiles. There exists a disjoint decomposition
\[
\mathbb T= \mathbb T^{\mathrm{small}} \cup \bigcup_{\tau =1}^N \T_\tau
\]
where  $\{\T_\tau : 1\leq \tau \leq N\}$ is a finite collection of trees and
\[
\sum_{\tau =1}^N  |R_{\T_\tau}|  \lesssim \size(\mathbb T)^{-2} \|f\|_{{L^2(\R^n)}} ^2,\qquad \size(\mathbb T^{\mathrm{small}}) \leq \frac{ \size(\mathbb T)}{\sqrt{2}},
\]
with implicit constant depending only upon dimension.
\end{lemma}
%%%%%%%%%%%%%%%%%%%%%%%%%%%%%% LEMMA LEMMA LEMMA

%%%%%%%%%%%%%%%%%%%%%%%%%%%%%% PROOF PROOF PROOF
\begin{proof} Throughout the proof we abbreviate $\sigma\coloneqq\size(\mathbb T)$. We remove trees from $\mathbb T$ via the following recursive procedure. First consider all the maximal (with respect to set inclusion) lacunary trees $\T\subseteq \mathbb T$ such that
\begin{equation}\label{eq:largetree}
\sum_{t\in\T}F[f](t)^2 > \frac{\sigma^2 }{2}|R_{\T}|
\end{equation}
{and among those let $\T_1$ be the one with $\xi_{\T_1}$ having the \emph{smallest signature}, where $(\xi_{\T_1},R_{\T_1})$ is the {top data} of $\T_1$. We define} the collection of tiles
{\[
\mathcal E(\T_1)\coloneqq\{t\in\mathbb T:\, \xi_{\T_1}\in Q_t,\, \scl(R_t)\leq \scl(R_{\T_1}),\, R_t \cap K_n ^2 R_{\T_1}\neq \varnothing\}.
\]
}
Clearly $\T_1\subseteq \mathcal E(\T_1)$. Now replace $\mathbb T$ by $\mathbb T\setminus \mathcal E(\T_1)$ and recurse. The selection algorithm terminates when $\size(\mathbb T^{\mathrm small})=\size (\mathbb T\setminus  \bigcup_k \mathcal E(\T_k))<{\sigma/\sqrt{2}}$. 

Note that the {elements of the} collection $\{\mathcal E(\T_k)\}_k$ are not --- strictly speaking --- trees according to Definition~\ref{def:trees} but we can easily remedy that. Indeed consider the rectangular parallelepipeds $R_{\T_{k,j}}$ which belong to the same grid as $R_{\T _k}$ and are such that $R_{\T_{k,j}}\cap K_n ^2R_{\T_k}\neq \varnothing$. Clearly there are at most $O_d(1)$ indices $j$ for which this happens, uniformly in $k$. Now we set 
\[
{\T_{k,1}\coloneqq \left\{t\in\mathcal E(\T_k):\,  R_t \cap  R_{\T_{k,1}}\neq \varnothing\right\} }
\]
 and recursively let $\T_{k,j+1}$ be the maximal tree with top $({\xi_{\T_{k}}} , R_ {\T_{k,j+1}})$ contained in $\mathcal E(\T_k)\setminus\cup_{j'\leq j}\T_{k,j'}$. Each $\T_{k,j}$ is a tree and we have $\sum_k\sum_j |R_{\T_{j,k}}|\lesssim \sum_k |R_{\T_k}|$ so the proof will be complete once we estimate the latter sum. An important fact is that the collection $\mathscr T\coloneqq \{\T_k\}_k$ is strongly disjoint. Indeed if $t\in\T\neq \T' \ni t'$ with $\T,\T'\in\mathscr T$ {and} $Q_t \subseteq Q_{t'} ^\circ$, {then the tree $\T$} was selected first because of Remark~\ref{rmrk:sign}. Thus if we had that $R_{t'}\cap K_n ^2 R_\T\neq \varnothing$ then the tile $t'$ would have been included in the family $\mathcal E(\T)$ and consequently would not be available for inclusion in the tree $\T'$. Now using the selection condition \eqref{eq:largetree} together with Lemma~\ref{lem:strodis} we get
\[
\sum_{\T\in\mathscr T} |R_{\T }| \lesssim \sigma^{-2} \sum_{\T\in\mathscr T}\sum_{t\in\T} F[f](t)^2 \lesssim \sigma^{-2} \max\left(\|f\|_2 ^2,\sigma^{\frac23}\left(\sum_{\T\in\mathscr T}|R_\T|\right)^{\frac13}\|f\|_2 ^{\frac43} \right)
\]
which readily yields the desired estimate.
\end{proof}
%%%%%%%%%%%%%%%%%%%%%%%%%%%%%% PROOF PROOF PROOF

%%%%%%%%%%%%%%%%%%%%%%%%%%%%%% SECTION SECTION SECTION
\subsection{The single tree estimate} The size and density lemmas combined with the discretization of the operator in \eqref{eq:model} reduce matters to the estimate for a single tree which is stated and proved below.

%%%%%%%%%%%%%%%%%%%%%%%%%%%%%% LEMMA LEMMA LEMMA
\begin{lemma}[tree lemma]\label{lem:tree}  {For every tree $\T$ and every $\tau_o\in\{1,\ldots,N\}$, there holds}
\[
{\Lambda_{\T;\tau_o}(f,g\ind_E)}\lesssim \size(\T) \dense(\T)|R_T|
\]
where $\size$ is defined with respect to $f\in L^\infty_0(\R^n)$, $\dense$ is defined with respect to $E\subset \R^n$ and $g\in L^\infty_0(\R^n)$ with $\|g\|_\infty=1$ is arbitrary.
\end{lemma}
%%%%%%%%%%%%%%%%%%%%%%%%%%%%%% LEMMA LEMMA LEMMA

%%%%%%%%%%%%%%%%%%%%%%%%%%%%%% PROOF PROOF PROOF
\begin{proof}  
In view of the definition of the form ${\Lambda_{\mathbf{T};\tau_o}}$ it suffices to prove that 
\begin{equation}\label{tree_estimate}
\sum_{t \in \mathbf{T} } \left| \langle f, \varphi_t \rangle \right| \left| \left\langle \vartheta_t \left(\cdot, \sigma(\cdot) \right) ,{g\ind_{{E_{t,\tau_0}}} } \right\rangle  \right| \lesssim \size (\mathbf{T}) \dense (\mathbf{T}) | R_{\mathbf{T} } |,
\end{equation}
where $E_{t,\tau_0}:=\{x\in E: v_{\sigma}(x) \in \alpha_{t,\tau_0}\}$,
uniformly in choices of adapted families $\{\varphi_t:\, \varphi_t\in\mathcal F_t ^M,\, t\in \T\}$ and $\{\psi_t:\, \varphi_t\in\mathcal A_t ^M,\, t\in \T\}$. We remember that we have fixed $M=50n$.

We begin the proof setup with a reduction. Since the single tree estimate is rotation invariant we can and will assume that $\xi_\T=\Pi_{e_n^\perp} e_n=0\in\R^d$. {A simple geometric argument similar to the one needed in the proof of} Lemma~\ref{lem:geom}  implies that there exists a dimensional constant $C_n >0$ such that for any tile $t \in \mathbf{T}$ there exists a rectangle ${R}_t  ^*$ with sides parallel to the coordinate axes, $\Pi_{e_n ^\perp} ( {R}_{t} ^*) $ is a cube in $\mathbb{R}^d$ and $R_t \subset  R_t  ^* \subset C_n R_t$. 
Furthermore, by e.g. splitting the tree into $O _n(1)$ trees we may assume that all $R_t ^*$ come from a  single grid $\mathcal{L}$ in $\mathbb{R}^n$ that consists of rectangles with sides parallel to the axes and is such that  
\[
\{ \Pi_{e_n^{\perp}} (R) :R \in \mathcal{L}\}\subseteq {\mathcal D_{(\kappa,m)},}
\]
where {$\mathcal D_{(\kappa,m)}$ is a dyadic grid of pace $(\kappa,m)$ in $\mathbb{R}^d$ with $m\in[0,\kappa)$}, and that $\Pi_{e_n}(R^*)$ are {standard} dyadic intervals {from $\mathcal D$} of length $4$. We also fix ${R}_{\mathbf{T} } ^* \in \mathcal{L}$ with $ R_{\mathbf{T}} \subseteq R ^* _{\mathbf{T}} \subseteq C_n R_{\mathbf{T}}$ and $R_t ^*\subseteq R_\T ^*$ for all $t\in\T$. 

We note that $\T^*\coloneqq\{R_t ^* \times Q_t:\, t\in\T\}$ is not --- strictly speaking --- a collection of tiles according to our definition in \S\ref{sec:tiles}. However the previously defined relation of order is still meaningful with the same formal definition and we have that 
\[
\size(\T)\eqsim \size(\T^*),\qquad \dense(\T)\eqsim \dense (\T^*)
\]
with implicit dimensional constants, where in the definitions for $\size(\T^*),\dense(\T^*)$ above we use $t\in\T^*$ instead of the actual tiles of $\T$. With these reductions and remarks taken as understood we will henceforth write $\T=\T^*$ and assume that $\{R_t:\, t\in\T\}\subseteq \mathcal L$ with $\mathcal L$ being a grid as  above. Note that the uncertainty relation will change in an inconsequential way as we will now have $\scl(R_t) \ell(Q_t)\in [c_n ^{-1},c_n)$ for some dimensional constant $c_n>3$.

Now let $\mathcal{P}(e_n ^\perp)$ be the maximal dyadic cubes $Q$ in $\R^d$ such that $3Q\nsupseteq \Pi_{e_n ^\perp}R_t $ for any $t \in \mathbf{T}$; note that $\mathcal{P}(e_n ^\perp)$ partitions $\mathbb{R}^d$. Let $\mathcal P(e_n)$ be the dyadic intervals of length $2$ on the real line and define $\mathcal P\coloneqq \mathcal P(e_n ^\perp)\times \mathcal P(e_n)$.  We will write $\scl(P)\coloneqq \ell(\Pi_{e_n ^\perp} P)$. For $P \in \mathcal{P}$ we write $\mathbf{T}  = \mathbf{T} _P^+ \cup \mathbf{T} _P^-$, where
\[
\mathbf{T} _P^- \coloneqq \left\{ t \in \mathbf{T} : \, \scl( R_t ) \leq  \scl(P)  \right\}    \quad \text{and} \quad \T_P^+ \coloneqq \left\{ t \in \mathbf{T} :\, \scl(R_t)  > \scl(P)   \right\} .
\]
We write the basic estimate
\[
\sum_{t \in \mathbf{T}} \left| \langle f, \phi_t \rangle  \right| \left|  \left\langle   \vartheta_t \left(\cdot, \sigma(\cdot) \right), g\ind_{{{E_{t,\tau_0}}}}\right\rangle \right|  
 = \sum_{P \in \mathcal{P}} \int_P L_P^-  (x) \, \d x +  \sum_{P \in \mathcal{P}} \int_P L_P^+ (x) \, \d x,
\]
where
\[
L_P^{\omega} (x) \coloneqq  \sum_{t \in { \mathbf{T}_P ^\omega}} \eps_t \langle f, \varphi_t  \rangle    \vartheta_t (x, \sigma(x)) g(x)\cic{1}_{{{E_{t,\tau_0}}}}(x)  ,  \quad \omega \in \{ -,+ \},\quad x\in\R^n,
\]
for an appropriate sequence $\{ \eps_t \}_{t \in \mathbf{T}}$ on the unit circle in $\mathbb C$. We shall prove that
 \begin{equation}\label{eq:L-+}
\left| \sum_{P \in \mathcal{P}} \int_P L_P^{\omega} (x) \, \d x \right| \lesssim \size (\mathbf{T})  \dense (\mathbf{T}) | R_{\mathbf{T} } | , \quad \omega \in \{-,+\} . 
\end{equation}

%%%%%%%%%%%%%%%%%%%%%%%%%%%%%% SECTION SECTION SECTION
\subsubsection*{Proof of \texorpdfstring{\eqref{eq:L-+}}{(6.17)} for \texorpdfstring{$\omega=-$}{ω=-}} For $P \in \mathcal{P}$ and  $\|g\|_\infty\leq 1$
 \[
 \begin{split}
\left| \int_P L_P^{-} (x) \, \d x \right|  &\leq \size (\mathbf {T}) {\sum_{t\in \mathbf{T}_P^-}} \int_{P \cap E_{t} } |R_t|^{1/2} \left|  \vartheta_t \left(x, \sigma(x) \right) \right| |g(x)| \, \d x 
\\
&\lesssim \size(\T)\dense(\T){\sum_{t\in \mathbf{T}_P^-}}|R_t|\sup_{x\in P} \mathrm{Sy} ^\infty _{R_t} \chi_{40n}(x).
\end{split}
\]
For $R\in\{R_t:\, t \in  \mathbf{T}_P^-\}$, set
\[
\rho_R (P) \coloneqq {\inf_{x \in P}}\, \rho_R (x),\qquad \rho_R (x) \coloneqq \inf \{r>0: \,x \in r R \},
\]
and note that $\rho_{R_t}(P)\geq \max(1,\rho_{R_\T}(P))\eqqcolon \rho(P,\T)$ by the construction of $\mathcal P$ and the facts that {$\scl(P)\geq \scl(R_t)$ and $R_t\subseteq R_\T$ for $t\in\T_P ^-$.} It follows that
 \[
 \begin{split}
\left| \sum_{P \in \mathcal{P}} \int_P L_P^- (x) \, \d x \right| & {\lesssim} \size (\mathbf {T}) \dense (\mathbf{T})   \sum_{P \in \mathcal{P}}  \sum_{t \in \mathbf{T}_P^-}  |R_t|  (1+\rho_{R_t} (P))^{-40n} .
\end{split}
\]
Note that for $t\in \T_P ^-$ with ${\rho_{R_t}(P)}\in [2^\ell,2^{\ell+1})$ and $\ell(R_t)=2^m$ we have 
 \[
R_t \subseteq 2^{\ell+4} \frac{ 2^{m+1} }{\ell \big( \Pi_{e_n^{\perp}} (P) \big) } \Pi_{e_n^{\perp}} ( P ) \times 2^{\ell+4} \Pi_{e_n} ( P).
\]
We can thus estimate
\[
\begin{split}
\sum_{P \in \mathcal{P}} \sum_{t \in \mathbf{T}_P^-}  |R_t| {\rho_{R_t} (P)^{-40n}} & \lesssim \sum_{P \in \mathcal{P}} \sum_{{2^{\ell+1}} \geq \rho (P,\T)} 2^{-40n\ell } \sum_{{2^m=4}}^{\ell ( \Pi_{e_n^{\perp}} (P) )  }   \left( \frac{ {2^{\ell+1}} 2^m }{ \ell \big( \Pi_{e_n^{\perp}} (P) \big) } \right)^d 2^\ell |P|
% \\ & \lesssim \sum_{P \in \mathcal{P}} \sum_{{2^{\ell+1}} \geq \rho(P,\T)} 2^{-39n\ell}  |P|
\\
&= \sum_{ \substack{ P \in \mathcal{P} \\  P \subseteq 100   R _\T}} \sum_{{\ell \geq -1}} 2^{-39n\ell}  |P| + \sum_{ \substack{ P \in \mathcal{P} \\  P  \nsubseteq  100   R _{\mathbf{T}} }} \sum_{{2^{\ell+1}} \geq \rho_{R_\T} (P)} 2^{-39n\ell }  |P| \eqqcolon \mathrm{I}+\mathrm{II}.
\end{split}
\]
Since the collection $\{P\}_{P\in\mathcal P}$ is pairwise disjoint we readily get that $\mathrm{I}\lesssim | R_{\mathbf{T} } |$.  For $\mathrm{II}$ note that $P \nsubseteq  100 R_{\mathbf{T}}$ implies that $P \subseteq (  4 R_\T)^{\mathrm{c}}$ so that
\[
\mathrm{II} \lesssim \sum_{\substack{P\in\mathcal P\\ P\subseteq  ( 4R_\T) ^{\mathrm c} }}\int_P \rho_{R_\T}(x)^{-39n}\, \d x \lesssim \int_{ (4 R_\T)^{\mathrm c}}\rho_{R_\T}(x)^{-39n}\, \d x\lesssim |R_\T|
\]
and the proof of \eqref{eq:L-+} for $\omega=-$ is complete.

%%%%%%%%%%%%%%%%%%%%%%%%%%%%%% SECTION SECTION SECTION
\subsubsection*{Proof of \texorpdfstring{\eqref{eq:L-+}}{(6.17)} for \texorpdfstring{$\omega=+$}{ω=+}} We make the preliminary observation that if $\T_P ^+\neq \varnothing$ then $\scl(R_\T)\geq 2 \scl(P)$ and $\Pi_{e_n ^\perp} P\subseteq 5\Pi_{e_n ^\perp} R_\T$. For each $P\in\mathcal P$ we define
\[
G_P \coloneqq P \cap \bigcup_{t \in \mathbf{T}_P^+} E_t .
\]

%%%%%%%%%%%%%%%%%%%%%%%%%%%%%% LEMMA LEMMA LEMMA
\begin{lemma}\label{lem:support} {For $P\in\mathcal P$ there holds} 
$ 
| G_P | \lesssim (1+\rho_{R_\T}(P))^{10n} \dense (\mathbf{T}) |P| . 
$
\end{lemma}
%%%%%%%%%%%%%%%%%%%%%%%%%%%%%% LEMMA LEMMA LEMMA

%%%%%%%%%%%%%%%%%%%%%%%%%%%%%% PROOF PROOF PROOF
\begin{proof}  
Let $\widehat{P} \in \mathcal{L}$ be such that $ \Pi_{e_n^{\perp}}(\widehat{P})$ is the dyadic parent of $ \Pi_{e_n^{\perp}}(P)$. By the maximality of $P$ one has $ 3 \Pi_{e_n^{\perp}}(\widehat{P}) \supseteq \Pi_{e_n^{\perp}} ( R_{t_0} )$ for some $t_0 \in \mathbf{T}$.  Hence, there exists an $R \in \mathcal{L}$
such that $\ell ( \Pi_{e_n^{\perp}}(R) ) = \ell ( \Pi_{e_n^{\perp}}(\widehat{P}) )$ and  $ R \supseteq  R_{t_0} $.  

Let $Q $ be the largest cube in $ \mathcal{G}$ such that $\xi_{\mathbf{T}} \in Q \subseteq Q_{t_0} $  and $\scl (R) \scl (Q) \in [c_n^{-1} , c_n)$. Define $t \coloneqq R \times Q$ and note that $t \geq t_0 $ and $G_P \subseteq  P \cap E_t$. As 
\[
|P| \eqsim  |R|,\qquad \mathrm{Sy}^1_R  \chi_{10n}  (x)  \gtrsim (1+\rho_{R_\T}(P))^{-10n} | R |^{-1} \qquad \forall x\in P,
\]
we have
\[\begin{split} \quad &
| G_P | \lesssim | P |   \int_{E_t} \frac{1}{ | R |} \cic{1}_P (x) \d x\\ &  \lesssim (1+\rho_{R_\T}(P))^{10n} |P| \int_{E_t} \mathrm{Sy}^1_R \chi_{10n} (x) \d x  \lesssim  (1+\rho_{R_\T}(P))^{10n}  |P|  \dense(\mathbf{T}), \end{split}
\]
and the proof is complete.
\end{proof}
%%%%%%%%%%%%%%%%%%%%%%%%%%%%%% PROOF PROOF PROOF

We can decompose  $\mathbf{T} = \mathbf{T}_{\text{ov}}\cup_{\tau=1} ^N \mathbf{T}_{\text{lac}_\tau}$, where $\mathbf{T}_{\text{ov}}$ is an overlapping tree and $\mathbf{T}_{\text{lac}_\tau}$ is a $\tau$-lacunary tree for each $\tau\in\{1,\ldots,N\}.$ For $P \in \mathcal{P}$, we set 
$\mathbf{T}_{P, \text{ov}}^+ \coloneq \mathbf{T}_{ \text{ov}} \cap \mathbf{T}_P^+$ and  
$ \mathbf{T}_{P, \text{lac}_\tau }^+ \coloneq \mathbf{T}_{ \text{lac}_\tau} \cap \mathbf{T}_P^+ $ for $\tau\in\{1,\ldots,N\}$. We define
\[
L_{P,\mathrm{type}}^{+} (x) \coloneqq  \sum_{t \in \mathbf{T}_{P,\mathrm{type}} ^+} \eps_t \langle f, \varphi_t  \rangle    \vartheta_t (x, \sigma(x)) g(x)\cic{1}_{{{{E_{t,\tau_0}}}}}(x)  ,  \quad \mathrm{type} \in \left\{ \mathrm{ov},\cup_{\tau}\mathrm{lac}_\tau \right\},\quad x\in\R^n.
\]
{
For reasons of space, until the end of the proof we use the local notation
\[
F\coloneqq \sum_{t\in\T_{\mathrm{lac}_{\tau_o}} ^+} \eps_t \l f,\varphi_t\r\widetilde{\vartheta}_t.
\]
We shall prove the estimates
\begin{align}\label{L+_ov}
&  \|L_{P,\mathrm{type}}^{+} \|_{L^\infty(P)}\lesssim \size(\T)(1+\rho_{R_\T}(P))^{-49 n},\quad \mathrm{type}\in\{\mathrm{ov},\cup_{\tau\neq\tau_o} \mathrm{lac}_\tau\}
\\ 
  \label{L+_lac}
&  \|L_{P,\mathrm{lac}_{\tau_o}}^{+} \|_{L^\infty(P)}\lesssim\frac{1}{(1+ \rho_{R_\T}(P))^{20n}} \inf_P \M_{\mathrm S}F+\frac{\size(\T)}{(1+\rho_{R_\T}(P))^{49n}},
\end{align}
}
where $\{\widetilde{\vartheta}_t:\, t\in\T_{\mathrm{lac}} ^+\}\subset \mathcal A_t ^{30n}$ and $\M_{\mathrm{S}}$ is the strong maximal function. Before doing so let us show that \eqref{L+_ov} and \eqref{L+_lac} imply \eqref{eq:L-+} for $\omega=+$. Indeed, using \eqref{L+_ov} and Lemma~\ref{lem:support} we have {for $\mathrm{type}\in\{\mathrm{ov},\cup_{\tau\neq\tau_o} \mathrm{lac}_\tau\}$}
\[
\begin{split}
\left|\sum_{P\in\mathcal P} \int_P {L^+ _{P,{\mathrm{type}}}}{(x)\, \d x} \right|&\lesssim\size(\T)\dense(\T) \sum_{P\in\mathcal P}|P|(1+\rho_{R_\T}(P))^{-39n}\lesssim \size(\T)\dense(\T)|R_\T|.
\end{split}
\]
On the other hand, using \eqref{L+_lac} and Lemma~\ref{lem:support} and the same calculation as in the last display we get
\[
\begin{split}
\int\left|\sum_{P\in\mathcal P} L_{P,\mathrm{lac}_{\tau_o}} ^+(x)\, \d x\right| &\lesssim \size(\T)\dense(\T)|R_\T|
  + \dense(\T)\sum_{P\in\mathcal P} (1+ \rho_{R_\T}(P))^{-10n} \int_P \M_{\mathrm{S}}F.
\end{split}
\]
Observe that
\[
\begin{split}
	\sum_{\substack{P\in\mathcal P\\P\subseteq 100R_\T}} \int_P  \M_{\mathrm{S}}F \lesssim |R_\T| \size(\T)
\end{split}
\]
by the Cauchy--Schwarz inequality, the boundedness of the strong maximal function on $L^2(\R^n)$ and Lemma~\ref{lem:ortho}. Remembering that $\scl(R_\T)\geq 2\scl(P)$ we get
\[
\begin{split}
&\sum_{\substack{P\in\mathcal P\\P\nsubseteq 100R_\T}} {\textstyle \frac{1}{(1+\rho_{R_\T}(P))^{10n}}}\int_P  \M_{\mathrm{S}}F
 \lesssim  \sum_{\ell\geq 2} 2^{-10\ell n} \sum_{\substack{  P\in\mathcal P\\\rho_{R_\T}(P)\eqsim 2^\ell}} \int_{P} \M_{\mathrm{S}}F \\ &
  \lesssim \sum_{\ell\geq 2} 2^{-10\ell n} |2^{\ell+2} R_\T|^{\frac12}|R_\T|^{\frac12}\size(\T)\lesssim |R_\T|\size(\T)
 \end{split}
\]
where we also used that $\rho_{R_\T}(P)\eqsim 2^\ell\implies P\subseteq 2^{\ell+2}R_\T$ in passing to the penultimate approximate inequality.
%%%%%%%%%%%%%%%%%%%%%%%%%%%%%% SECTION SECTION SECTION
\subsubsection*{Proof of\, \texorpdfstring{\eqref{L+_ov}}{(6.19)}} {Let $\mathrm{type}\in\{\mathrm{ov},\cup_{\tau\neq\tau_o} \mathrm{lac}_\tau\}$ be fixed.} Observe that for fixed $x \in \R^n$ if, for $t,t' \in  \mathbf{T}_{P, \text{type}}^+$, $ \vartheta_t (x , \sigma (x)) \cic{1}_{{{{E_{t,\tau_0}}}}}(x)  \neq 0$ and $ \vartheta_{t'} (x , \sigma (x)) \cic{1}_{{{{E_{t',\tau_0}}}}}(x) \neq 0$ then by property (iv) one has $\scl (R_t) = \scl (R_{t'})$.
 Hence, 
\begin{equation}\label{ess_disj_supp}
\|L_{P,{\mathrm{type}}} ^+\|_{L^\infty(P)}\leq \size(\T)\sup_{ \substack{x \in  P\\ \scl\in 2^\N} }   \sum_{\substack{ t \in \mathbf{T}_{P, {\text{type}}}^+ \\\scl(R_t)=\scl}}|R_t|^{\frac12}| \vartheta_t (x , \sigma (x))  | |g(x)| \cic{1}_{E_t} (x)  .   
\end{equation}
{Define $\omega_t(y,x_n)\coloneqq t^{-d} \chi_{50n}(y/t,x_n)$. For any $\scl\in 2^\N$ and fixed $\sigma$,by   adaptedness of $\vartheta_t(\cdot,\sigma)$
\[
\sum_{\substack{ t \in \mathbf{T}_{P, \text{type}}^+ \\\scl(R_t)=\scl}}  |R_t|^{\frac12}| \vartheta_t (\cdot , \sigma )|  \lesssim \sum_{\substack{ t \in \mathbf{T}_{P, {\text{type}}}^+ \\\scl(R_t)=\scl}}   \mathrm{Sy} _{R_t}^{\infty} \chi_{50n}  \lesssim  \sum_{\substack{ t \in \mathbf{T}_{P, {\text{type}}}^+ \\\scl(R_t)=\scl}}    \omega_{\scl}*\ind_{R_t} \lesssim  \omega_{\scl}*\ind_{R_\T}
 \]
using that the collection $\{R_t\}_{\scl(R_t)=\scl}$ is pairwise disjoint and $R_t\subseteq R_\T$ for all $t\in\T$. Because of the trivial estimate $\omega_{\scl}*\ind_{R_\T} \lesssim 1$ it will be enough to consider $P$ such that $P\cap 3R_\T=\varnothing$; indeed if $P\cap 3R_\T\neq \varnothing$ then $\rho_{R_\T}(P )\lesssim 1$ and the desired estimate follows. So fix such a $P$ and let $c_{R_\T}$ denote the center of $R_\T$. If $x=(y,x_n)\in P$ then 
\[
\begin{split}
\omega_{\scl}*\ind_{R_\T} (y,x_n)&\lesssim \frac{|R_\T|}{\scl^d} \left( |x_n-c_{R_\T}| + \frac{|y-c_{R_\T}|}{\scl}\right)^{-50n}
\\
&\lesssim \frac{\scl(R_\T)^d}{\scl^d} \left( |x_n-c_{R_\T}| + \frac{|y-c_{R_\T}|}{\scl}\right)^{-49n} \left(\frac{\scl(R_\T)}{\scl}\right)^{-n}\lesssim \rho_{R_{\T,\scl}}(P) ^{-49n}
\end{split}
\]
where $R_{\T,\scl}$ is an axis parallel rectangular parallelepiped centered at $c_{R_\T}$ with sidelengths $\scl^d\times 4$. Thus we have proved
\[
\|\omega_{\scl}*\ind_{R_\T} \|_{L^\infty(P)} \lesssim \left(1+\rho_{R_{\T,\scl}}(P) \right)^{-49n}\leq \left(1+\rho_{R_{\T}}(P) \right)^{-49n}
\]
since $R_{\T,\scl}\subseteq R_\T$. This completes the proof of \eqref{L+_ov} since $\|g\|_\infty\le1$.
  }

%%%%%%%%%%%%%%%%%%%%%%%%%%%%%% SECTION SECTION SECTION
\subsubsection*{Proof of \texorpdfstring{\eqref{L+_lac}}{(6.20)}} Remember that $\xi_\T=\Pi_{e_n ^\perp} e_n=0\in\R^d$ which means that property (v) of \S\ref{sec:classes} applies for $\vartheta(\cdot,\sigma(\cdot))-\vartheta(\cdot,\R^d)$ where $\R^d\in\Gr(d,n)$ {is} regarded as an element of the Grassmanian. Writing $\vartheta_t\coloneqq \vartheta_t(\cdot,\R^d)$ we preliminary split
\[
\begin{split}
L ^+ _{P,\mathrm{lac}} & = g\left( \sum_{t \in \mathbf{T}_{P,\mathrm{lac}_{\tau_o}} ^+} \eps_t \langle f, \varphi_t  \rangle    \vartheta_t  \cic{1}_{{{{E_{t,\tau_0}}}}}(x) \right)+g\left( \sum_{t \in \mathbf{T}_{P,\mathrm{lac}_{\tau_o}} ^+} \eps_t \langle f, \varphi_t  \rangle    [\vartheta_t(\cdot,\sigma(\cdot))-\vartheta_t  ]\cic{1}_{{{{E\cap \alpha_{t,\tau_o}}}}}(x) \right)
\\
&\eqqcolon g L_{P,1}+gL_{P,2}.
\end{split}
\] As $\|g\|_\infty\leq 1$ it suffices to prove the required pointwise estimate is satisfied by $L_{P,j}$, $j=1,2$.
We first deal with the term $L_{P,1}$. {Note that the sets   $\{Q_{t,\tau_o}\}_{t\in\T_{\mathrm{lac},\tau_o}}=\{\Pi_{e_n^{\perp}}(\alpha_{t,\tau_o})\}_{t\in\T_{\mathrm{lac}_{\tau_o}}}$  are nested since they all contain $\xi_\T$. Furthermore, for a fixed $x\in\R^n$ a term in the  sum defining $L_{P,1}(x)$ is nonzero if and only if $v_\sigma(x)\in \alpha_{t,\tau_0}$. We can then conclude that there exist positive numbers  $\scl(P)< m(x)\leq  M(x)$ such that the condition $\ind_{E_{t,\tau_0}}(x)\neq 0$ is satisfied if and only if $m(x)\leq \scl(R_t)\leq M(x)$. Letting $t_{R_\T}\coloneqq \mathrm{Sy} _{R_\T} ^\infty \chi_{20n}$ we then have for any $c>0$} 
\[
L  _{P,1}(x)= c^{-1}t_{R_\T}(x) \ind_E(x) \sum_{\substack  {t\in\T_{\mathrm{lac}_{\tau_o} } ^+ \\ m(x) \leq \scl(R_t)\leq M(x)} }\eps_t \l f,\varphi_t\r \widetilde{\vartheta}_t(x)
\]
where $\widetilde {\vartheta}_t \coloneqq c t_{R_\T} ^{-1} \vartheta_t $. Note that for a choice of $c$ depending only upon dimension we will have that $\{\widetilde{\vartheta} _t:\, t\in \T^+ _{\mathrm{lac}_{\tau_o}}\}\subset \mathcal A_t ^{30n}$. In order to conclude the desired estimate for $L_{P,1}$ we consider  $\eta \in\mathcal S(\R^n)$  with  $ \supp ( \widehat{\eta} ) \subseteq [-3,3]^d\times [1/4,4]$ and $ \widehat{\eta} (\xi) = 1 $ for all $ \xi \in  [-2,2]^d\times [1/2,2] $. Then, as $\mathrm{supp} \ (\widetilde{\vartheta}_t) ^\wedge  \subseteq \omega_t $, we may write for any $B>\beta>0$
\begin{align*}
& \sum_{\substack{ t \in \mathbf{T}^+ _{\text{lac}_{\tau_o}}   \\ \beta \leq  \scl( R_t )  \leq B }} \eps_t \langle f, \phi_t \rangle \widetilde{\vartheta}_t  =    \left(  \eta_{\gamma_n B} - \eta_{\gamma_n\beta} \right) *   \sum_{t \in \mathbf{T}_{\text{lac}_{\tau_o}} ^+} \eps_t \langle f, \phi_t \rangle \widetilde{\vartheta}_t   
\end{align*}
for a constant  $\gamma_n>0$, where $\eta_r (x) \coloneqq r^{-d} \sigma (r^{-1} \Pi_{e_n^{\perp}}x + \Pi_{e_n}x )$ for $r>0$ and $x \in \mathbb{R}^n$. Thus
\[
\sup_{x\in P} |L_{P,1}(x)| \lesssim (1+\rho_{R_\T}(P))^{-20n}\inf_P \M_{\mathrm{S}}\left(\sum_{t\in\T_{P,\mathrm{lac}_{\tau_o}} ^+} \eps_t \l f,\varphi_t\r\widetilde{\vartheta}_t\right)
\]
as desired.

We turn to the estimate for $L_{P,2}$. We use the fact that $\scl(t)\lesssim \dist(\sigma(x),\R^d)^{-1}$ for $t\in\T$ together with (v) of \S\ref{sec:classes} to estimate for $x\in P$
\[
\begin{split}
|L_{P,2}(x)|&\leq \size(\T)   \sum_{t \in \mathbf{T}_{P, \text{lac}_{\tau_o}}^+}  |R_t|^{1/2}  \left| \vartheta_t \left(x, \sigma(x) \right) - \vartheta_t \left(x \right) \right| 
\\
& \lesssim  \size (\T)\dist( \sigma (x), \R^d) \sum_{\substack{\scl\in 2^\N \\ \scl \lesssim  \dist( \sigma (x), \R^d )  ^{-1} } } \scl\sum_{\substack{t \in \mathbf{T}_{P, \text{lac}_{\tau_o}}^+\\ \scl(R_t) = \scl }}   \mathrm{Sy}_{R_t}^{\infty} \chi_{50n} (x)  + 
\\
&\qquad + \frac{\size (\T)} { \log \left( e
+ \dist( \sigma (x), \R^d )  ^{-1} \right) }   \sum_{\substack{\scl \in 2^\N \\ \scl  \lesssim  \dist( \sigma (x), \R^d )  ^{-1} } }   \sum_{\substack{t \in \mathbf{T}_{P, \text{lac}_{\tau_o}}^+ \\ \scl(R_t) = \scl }}   \mathrm{Sy}_{R_t}^{\infty} \chi_{50n} (x) 
\\
& \lesssim \size(\T) (1+\rho_{R_\T} (P))^{-49n}.
\end{split}
\]
This concludes the proof of \eqref{L+_lac} and with that the proof of the tree estimate.
\end{proof}
%%%%%%%%%%%%%%%%%%%%%%%%%%%%%% PROOF PROOF PROOF

%%%%%%%%%%%%%%%%%%%%%%%%%%%%%% SECTION SECTION SECTION
\subsection{The proof of \texorpdfstring{\eqref{eq:model}}{(5.27)}-\texorpdfstring{\eqref{eq:model2}}{(5.28)}} Fix a finite collection $\mathbb P$, $g\in L^\infty_0(\R^n)$ with $\|g\|_\infty=1$ throughout this discussion. Also fix a set $E\subset \R^n$ of finite measure. The function $\dense$ below corresponds to this choice of $E$.

We first prove \eqref{eq:model}. For this, fix $f\in L^\infty_0(\R^n)$ so that the  $\size$ map is computed with respect to $f$. {We note that for any collection $\mathbb P$ we have the \emph{apriori} bounds 
\begin{equation}\label{e:sizecontrol}
\dense(\mathbb P)\lesssim \min(1,|E|),\qquad \size(\mathbb P) \lesssim \|f\|_\infty \eqqcolon 2^{-\frac{K(f)}{2}} \|f\|_2,
\end{equation}
the first one being easily verified by checking the definition of $\dense$ and the second one being a consequence of  2.\ from Lemma \ref{lem:ortho}}.  By consecutive applications of the size lemma, Lemma~\ref{lem:size}, and the density lemma, Lemma~\ref{lem:dense}, {and starting from the initial bounds for $\dense$ and $\size$ above,} we can decompose the collection of tiles $\mathbb P$ in the form
\begin{equation}\label{eq:decomp}
\begin{split}
&
\mathbb P =\bigcup_{k\geq K(f)}\bigcup_{{\nu}=1}^{N_k} \T_{k,{\nu}},\\ & \size(\T_{k,\nu})\leq 2^{-\frac {k}{2}}\|f\|_2,\qquad  \dense(\T_{k,\nu}){\lesssim} \min\left\{1,2^{-k}|E|\right\}, \qquad \sum_{\nu=1}^{N_k}  |R_{\T_{k,\nu}}|\lesssim 2^{k}. \end{split}
\end{equation}
Consequently, employing  the tree   Lemma~\ref{lem:tree} we have {for each $\tau\in\{1,\ldots,N\}$}
\[
\begin{split}
\Lambda_{\mathbb P;{\tau}}(f,g\cic{1}_E)&\lesssim \sum_{k\geq K(f)}\sum_{{\nu}=1}^{N_k}\size(\T_{k,\tau}) \dense(\T_{k,{\nu}})|R_{\T_{k,{\nu}}}|  
\lesssim \|f\|_2\sum_{k\in \mathbb Z} 2^{\frac{-k}{2}}\min(1,2^{-k}|E|) \sum_{{\nu}=1}^{N_k}|R_{\T_{k,{\nu}}}|  
\\ &\lesssim \|f\|_2|E|^{\frac 12},
\end{split}
\]
which is the sought after estimate \eqref{eq:model}.

Turning to \eqref{eq:model2}, fix  $F\subset\R^n$ of finite measure. Accordingly, in what follows, we compute the $\size$ with $\cic{1}_F$ in place of $f$. 
If $|F|\leq |E|$ the desired estimate \eqref{eq:model2} actually follows from \eqref{eq:model}, as  \[
\Lambda_{\mathbb P;{\tau}}(\cic{1}_F,g\cic{1}_E)\lesssim \|\cic{1}_F\|_2  |E|^{\frac12} =|F|^{\frac 1p} |F|^{\frac12 - \frac 1p}|E|^{\frac12}\leq    |F|^{\frac1p}|E|^{1-\frac{1}{p}}.
\]
We can therefore assume that $|F|>|E|$. Then decomposition \eqref{eq:decomp}, estimates \eqref{e:sizecontrol} with $\cic{1}_F$ in place of $f$,  and the tree lemma imply that for each $\tau\in\{1,\ldots,N\}$ we have
\begin{multline*}
\quad \Lambda_{\mathbb P;{\tau}}(\cic{1}_F,g\cic{1}_E)\lesssim \sum_{k\in \mathbb Z} \min(1,2^{-\frac k2}|F|^{\frac12})\min(1,2^{-k}|E|)2^{k}
\\
\lesssim \sum_{2^{k}\geq |F|}
2^{-\frac k2}|F|^{\frac12}|E|+\sum_{|F| >2^k\geq |E|}|E|+\sum_{2^k<|E|} 2^{k}
\lesssim |E| \left(1+\log\frac{|F|}{|E|} \right) \lesssim |F|^{\frac1p}|E|^{1-\frac{1}{p}}.
\end{multline*}
whence \eqref{eq:model2} follows.
%%%%%%%%%%%%%%%%%%%%%%%%%%%%%% SECTION SECTION SECTION
\section{Complements} \label{sec:compl}

%%%%%%%%%%%%%%%%%%%%%%%%%%%%%% SECTION SECTION SECTION
\subsection{Maximal truncations and differentiation of Besov spaces}
\label{sub:besov} 

Let $\gamma\in \mathcal S(\R^d)$ be a radial function satisfying
\begin{equation}\label{eq:gamma}
\mathrm{supp}\widehat \gamma \subseteq B^{d}(1),\qquad \widehat \gamma (0)=1,\qquad \|\gamma\|_N \coloneqq \sup_{|\alpha|,|\beta|\leq N} \| x^\alpha D ^\beta \gamma\|_{L^\infty(\R^n)}\leq {C_{d}}.
\end{equation}
A few minor tweaks in the proof yield that the maximally truncated version of \eqref{e:thisismain}
\begin{equation} \label{e:thisismainmax}
T^{\star,\star}_{{\bf{m}}} f(x) \coloneqq \sup_{\sigma \in \mathcal \Gr(d,n)} \sup_{Q\in \mathrm{SO(d)}} \sup_{h>0} \left|\int_{{\R^n}} m_{\sigma}(Q O_\sigma \Pi_\sigma\xi)  \widehat{\gamma}( h{O_\sigma}\Pi_\sigma\xi)\widehat f(\xi)\e^{2\pi i \langle x,\xi \rangle}\d\xi \right|, \quad x\in {\R^n},
\end{equation} also satisfies  the conclusions of Theorem \ref{thm:main},
under the same assumptions therein, provided that $N$ in \eqref{eq:gamma} is sufficiently large, depending upon dimension. Indeed, reductions similar to those in Section \ref{sec:tiles+model} may be performed, observing that the smooth cutoff at frequency scale $h^{-1}$ does not modify adaptedness of wave packets at frequency scales $s^{-1}\lesssim h^{-1}$ while erasing the contribution of frequency scales $\gg h^{-1}$. This  leads, in analogy  with \eqref{e:wp3}, to a  model sum of type
\[
\Lambda^\star_{\mathbb P;\sigma,h,\tau,M} (f,g) =\sum_{t\in \mathbb P}F_M[f](t) A^\star_{\sigma,\tau,M}[g] (t) ,
\]
where the coefficients $A_{\sigma,h,\tau,M}$ have been   replaced by the modified version
\[
A^\star_{\sigma,h,\tau,M}[g] (t) \coloneqq \sup_{\vartheta \in \mathcal {A}_{t} ^M } \left| \left\langle g, \vartheta\left(\cdot, \sigma(\cdot) \right) \cic{1}_{\alpha_{t,\tau}}\left(v_{\sigma(\cdot)}\right)  \cic{1}_{[\scl(t),\infty)}(h(\cdot))\right\rangle \right|,
\] 
corresponding to {measurable selector functions} $h:\R^n\to (0,\infty)$ and $\sigma:\R^n\to \Gr(d,n)$. All arguments of Section \ref{sec:trees} work just as well for the modified coefficients with purely notational change: details are left to the interested reader.

If $m_\sigma\equiv 1$, the maximal operator \eqref{e:thisismainmax} coincides up to a constant factor with the maximal subspace averaging operator
\[
T^{\star,\star} f(x)= \sup_{\sigma \in \mathcal \Gr(d,n)}  \sup_{h>0} \left| A_{{\sigma, h}} f(x)\right|, \quad  A_{\sigma, h} f (x) \coloneqq   \int_{\R^d} f(x - O_\sigma ^{-1} t) \mathrm{Dil}_{h}^1\gamma(t){\, \d t},\quad x\in\R^n.
\]
Thus Theorem~\ref{thm:main} implies the weak $(2,2)$ and strong $(p,p)$, $p>2$, bounds for $T^{\star,\star}{\circ P_0}$, uniformly over $\gamma$ satisfying \eqref{eq:gamma} with $N$ sufficiently large; for example $N=50(d+1)$ is sufficient. {Moreover, using scale invariance and the decay of Schwartz tails we can get that the same boundedness property holds for $T^{\star,\star}$ if $A_{\sigma,h}$ is defined via a Schwartz function $\gamma\in \mathcal S(\R^d)$ satisfying
\[
\widehat \gamma(0)=1,\qquad \|\gamma\|_{100d}\leq 1;
\]
namely we can drop the compact Fourier support assumption.} We can use these bounds to conclude Lebesgue differentiation theorems for functions in the nonhomogeneous Be\-sov spaces $B_{p,1}^s(\R^n)$, $s\geq 0$, with norm
  \[
  \|f\|_{B_{p,1}^s(\R^n)}\coloneqq \| Q_0 f\|_{L^p(\R^n)} + \sum_{k\geq 0 } 2^{ks}\|P_{k} f\|_{L^p(\R^n)},
  \]
 where $Q_0, \{P_k: \,k\geq 0\}$ is any fixed smooth nonhomogeneous Littlewood--Paley decomposition. {In the case that $s\in (0,1)$ we can use an alternative characterization of Besov spaces in terms of finite differences, as for example in \cite{Triebel}, that shows that} $|f|\in B_{p,1} ^s$ whenever $f\in B_{p,1} ^s$, with a corresponding inequality of the norms, and thus we can conclude that the rough maximal averages
\[
M^{\star,\star} f(x)= \sup_{\sigma \in \mathcal \Gr(d,n)}  \sup_{h>0} \langle |f|\rangle_{{x,\sigma, h}},  \quad   \langle f\rangle_{{x,\sigma, h}}\coloneqq    \avgint_{B^n(0,h)\cap \sigma} f(x - O_\sigma ^{-1} t)\, \d t ,\quad x\in\R^n,
\]
have the same boundedness properties as $T^{\star,\star}$.

 %%%%%%%%%%%%%%%%%%%%%%%%%%%%%% THEOREM THEOREM THEOREM
\begin{theorem} \label{thm:cvg}Let $\sigma:\R^n\to \mathrm{Gr}(n-1,n)$ be any measurable function and let $2\le p<\infty$. If $f\in B_{p,1}^0(\R^n) $, then
\[
f(x) = \lim_{h\to 0^+} A_{\sigma(x), h}f(x) \qquad \textrm{a.e.} \,\,x\in 
\R^n,
\]
{whenever $A_{\sigma,h}$ is defined via a bump function with $\widehat \gamma(0)=1$ and $\|\gamma\|_{100n}\leq 1$.} Moreover, if $f\in B_{p,1} ^s$ for some $s>0$ then
\[
 f(x) = \lim_{h\to 0^+} \langle f\rangle_{{x,\sigma, h}}\qquad \textrm{a.e.} \,\,x\in \R^n.
 \]
\end{theorem}
 %%%%%%%%%%%%%%%%%%%%%%%%%%%%%% THEOREM THEOREM THEOREM

%%%%%%%%%%%%%%%%%%%%%%%%%%%%%% PROOF PROOF PROOF
\begin{proof}  By the comments preceding the statement of the theorem it suffices to consider the case $s=0$. By assumption, for each $\eps>0$ fixed there exists $k=k(\eps)\geq 0$ so that 
\[
\sum_{k> k(\eps)} \|P_k f\|_p <\eps.
\]
Let $g=Q_0 f + \sum_{k=0}^{k(\eps)} P_k f$. Note that $g\in C^\infty(\R^n)\cap L^2(\R^n)$ and that
\[
\| T^{\star,\star}  (f-g)\|_{p,\infty} \lesssim \sum_{k\geq k(\eps)} \| T^{\star,\star}  P_k f\|_{p,\infty} \lesssim  \sum_{k\geq k(\eps)} \|   P_k f\|_{p}< \eps
\] 
by an application of the uniform weak-$L^{p}(\R^n)$ estimate for $T^{\star,\star}  P_k$. The remaining part of the  argument is totally analogous to that of the proof of the Lebesgue differentiation theorem using the Hardy--Littlewood maximal theorem and is therefore omitted.
\end{proof}
%%%%%%%%%%%%%%%%%%%%%%%%%%%%%% PROOF PROOF PROOF
The results of Theorem~\ref{thm:cvg} recover a corresponding result from \cite{Mur} for Besov spaces $B_{p,1} ^s$ with $s>0$ and $d=1$; see also \cite{AFN,Naibo}. The codimension $n-d=1$ results as well as the results for zero smoothness appear to be new.

\subsection{Remarks on the bi-parameter problem} In this paragraph, $d$ is no longer necessarily equal to $n-1$ and  we go back to considering values $1\leq d<n$. The rotated codimension $n-d$ singular integral \eqref{eq:defdir-2}, \eqref{eq:defdir-1} may be naturally generalized to the bi-parameter version
\[
Tf(x; \sigma, Q, P) = \int_{\R^n} m(QO_\sigma\Pi_{\sigma} \xi,  PO_\sigma\Pi_{\sigma^{\perp}}\xi) \widehat f(\xi) \e^{2\pi i \langle x,\xi \rangle} \d \xi
\]
associated to a bi-parameter H\"ormander--Mihlin multiplier $m=m(\eta,\zeta)$ on $\R^n=\R^d \otimes \R^{n-d}$, that is, satisfying the estimates
\[\sup_{0\leq |\alpha|, |\beta| \leq A}\, \sup_{\xi \in \R^d   } \sup_{\zeta \in \R^{n-d}   } |\eta|^{|\alpha|} |\zeta|^{|\beta|} \left| D_\eta^\alpha D_\zeta^\beta    m(\eta,\zeta )\right|<\infty
\]
and parametrized by $ Q\in \mathrm{SO}(d), P\in \mathrm{SO}(n-d) $ and $\sigma \in \mathrm{Gr}(d,n)$ or equivalently $\sigma^\perp \in \Gr(n-d,n)$.   By arguments analogous to those in Section \ref{sec:redux} for the removal of the $\mathrm{SO}(d), \mathrm{SO}(n-d) $ invariances and finite splitting,  the  study of the corresponding maximal operator may be reduced to estimates for
\[
T^\star f(x)= \sup_{\sigma \in \Sigma} |Tf(x,\sigma)|, \qquad Tf(x,\sigma)\coloneqq  \int_{\R^n} m( O_\sigma\xi) \widehat f(\xi) \e^{2\pi i \langle x,\xi \rangle} \d \xi, \quad x\in \R^n
\] 
where $\Sigma \subset \mathrm{Gr}(d,n)$ is a small $2^{-9}\eps$-neighborhood of $\R^{d}$. When acting on band-limited functions, the  maximal operator $T^\star$ may be further reduced, up to $L^p$-bounded differences,  to the sum of a codimension $n-d$ and a codimension $d$ maximal singular integral operator. More specifically, let $\{\Gamma_j:j=-1,0,1\}$ be a partition of unity on the sphere $\mathbb S^{n-1}$ subordinated to the covering
\[
\begin{split} &
\Gamma_{-1}\coloneqq\{\xi \in \mathbb{S}^{n-1}: |\Pi_{\R^d} \xi| >1-2\eps\},\\ &\Gamma_0\coloneqq\{\xi \in \mathbb{S}^{n-1}: \eps< |\Pi_{\R^d} \xi| < 1-\eps\}, \\ &\Gamma_1\coloneqq\{\xi \in \mathbb{S}^{n-1}:  |\Pi_{\R^d} \xi| < 2\eps\}.
\end{split}
\]
and $P_{\mathrm{cn}, j}$, $j=-1,0,1$ be the corresponding smooth conical Fourier cutoff to $\Gamma_j$. 
  In analogy to what was done in \eqref{eq:nonsing},  the splitting 
  \[
  T^\star \circ P_0 f\leq  \sum_{j=-1,0,1}  T^\star [ P_0 P_{\mathrm{cn}, j} f]
  \]
  may be performed.
As  the singularities along $\sigma,\sigma^\perp$ sit away from the support of $ P_0 P_{\mathrm{cn}, {0}} f$ when $\sigma\in \Sigma$, $ T^\star [ P_0 P_{\mathrm{cn}, 0} f]$ may be easily controlled by the strong maximal function. Furthermore, when $\sigma \in \Sigma$
$
|\Pi_{\sigma^\perp  } \xi|\sim |  \xi| \sim 1 
$ on the support of $P_0 P_{\mathrm{cn}, -1}$ while $
|\Pi_{\sigma   } \xi|\sim |  \xi| \sim 1 
$ on the support of $P_0 P_{\mathrm{cn}, 1}$. A discretization procedure similar to that of Subsection \ref{sec:model} then shows that e.g.\ $ T^\star [ P_0 P_{\mathrm{cn}, -1} f]$ lies in the convex hull of the model operators
\[
\mathrm{T}f(x)=\sum_{t\in \mathbb P} \langle f, \varphi_t \rangle  \psi_t (x, \sigma(x) ) \cic{1}_{\left[\frac{a}{\scl(t)},\frac{b}{\scl(t)}\right]} (\dist(\sigma(x), \sigma(t))
\]
for suitable constants  $b>a\gg 1$, $ \mathbb{P}$ is a collection of tiles  $t$ whose spatial localization is a parallelepiped with $d$ sides oriented along the subspace $\sigma(t)$ of sidelength $\scl(t)\gg 1$ and $n-d$ short sides of sidelength 1  along  $\sigma(t)^\perp$. Also, for a fixed dyadic value of $\scl(t)=\scl$,  $\sigma(t)$ varies within a $\sim \scl^{-1}$ net in $\mathrm{Gr}(d,n)$ and for each fixed $\sigma$ the spatial localizations of the tiles $t\in \mathbb P$ with $\sigma(t)=\sigma$ tile $\R^n$. The model for $T^\star [ P_0 P_{\mathrm{cn}, 1} f]$ is analogous, up to exchanging $\mathrm{Gr}(d,n)$.

When $n=2,d=1$, then 
$T^\star \circ [ P_0 P_{\mathrm{cn}, \pm 1}]$ are thus completely symmetric and both may be handled with the methods of Section \ref{sec:trees}. When either $d$ or $n-d$ are greater than $1$, the corresponding model sums are substantially harder. One essential reason is that their scope includes Nikodym maximal averaging operators along subspaces of codimension $k$ larger than one, whose critical exponent, conjectured to be $p=k+1$, is above $L^2$: see \cite{DPP_adv} for more details on this conjecture.

%%      ---------------------------------------------------------------------
%%      --------------------------- BIBLIOGRAPHY ----------------------------
%%      ---------------------------------------------------------------------
%% PUT HERE THE BIBLIOGRAPHY IN YOUR FAVOURITE FORMAT
%% Please check that the format of the bibliography is uniform and coherent

\bibliographystyle{amsplain}
\bibliography{Genlib}

\end{document}